\documentclass[12pt]{article}

\usepackage{amsthm, amsmath, amssymb}
\usepackage{verbatim}
\usepackage{enumerate}
\usepackage{graphicx}
\usepackage{multirow, bigdelim}
\usepackage{thmtools, thm-restate}

\usepackage{color}
\definecolor{darkgreen}{rgb}{0,0.35,0}
\definecolor{darkblue}{rgb}{0,0,0.6}
\usepackage[breaklinks,pdftex,colorlinks,citecolor=darkgreen,linkcolor=darkblue]{hyperref}

\newtheorem*{tuttestrianglelemma}{Tutte's Triangle Lemma}
\newtheorem*{bixbycoullardinequality}{Bixby-Coullard Inequality}
\newtheorem*{bixbyslemma}{Bixby's Lemma}
\newtheorem*{tutteslinkingtheorem}{Tutte's Linking Theorem}
\newtheorem*{tutteslinkingtheorem2}{Tutte's Linking Theorem, Version 2}
\newtheorem*{fundamentaltheoremofprojectivegeometry}{Fundamental Theorem of Projective Geometry}

\declaretheorem[numberwithin=section]{theorem}
\declaretheorem[numberlike=theorem]{lemma}

\declaretheorem[numberlike=theorem]{conjecture}
\declaretheorem[numberlike=theorem]{proposition}
\newtheoremstyle{parentheses}{\topsep}{\topsep}{\itshape}{}{}{}{ }{\thmnumber{(#2)}}
\theoremstyle{parentheses}
\newtheorem{claim}{}
\makeatletter\@addtoreset{claim}{theorem}\makeatother

\renewcommand{\d}{\backslash}
\DeclareMathOperator{\si}{si}
\DeclareMathOperator{\cl}{cl}
\DeclareMathOperator{\PG}{PG}
\DeclareMathOperator{\AG}{AG}
\DeclareMathOperator{\F}{\mathbb{F}}

\newcommand{\defn}[1]{\textbf{#1}}

\newcommand{\footremember}[2]{%
   \footnote{#2}
    \newcounter{#1}
    \setcounter{#1}{\value{footnote}}%
}
\newcommand{\footrecall}[1]{%
    \footnotemark[\value{#1}]%
}

\title{Representation of matroids with a modular plane\thanks{This research was partially supported by a grant from the Office of Naval Research [N00014-10-1-0851].}}
\author{Jim Geelen\footremember{co}{Department of Combinatorics and Optimization, University of Waterloo, 200 University Ave W, Waterloo, Ontario N2L 3G1, Canada} \and Rohan Kapadia\footrecall{co} \thanks{This author was partially supported by an NSERC postgraduate scholarship. E-mail address: rohanfkapadia@gmail.com}}

\begin{document}

\maketitle
\begin{abstract}
We prove that if $M$ is a vertically $4$-connected matroid with a modular flat $X$ of rank at least three, then every representation of $M | X$ over a finite field $\mathbb{F}$ extends to a unique $\mathbb{F}$-representation of $M$.
A corollary is that when $\mathbb{F}$ has order $q$, any vertically $4$-connected matroid with a $\PG(2, \mathbb{F})$-restriction is either $\mathbb{F}$-representable or has a $U_{2, q^2+1}$-minor.
We also show that no excluded minor for the class of $\mathbb{F}$-representable matroids has a $\PG(2, \mathbb{F})$-restriction.
\end{abstract}

\section{Introduction} \label{sec:introduction}

We call a restriction $N$ of a matroid $M$ \defn{modular} if, for every flat $F$ of $M$,
\[ r_M(F) + r(N) = r_M(F \cap E(N)) + r_M(F \cup E(N)). \]
We prove the following theorem (recall that $\PG(r-1, \F)$ is the rank-$r$ projective geometry over the finite field $\F$).

\begin{restatable}{theorem}{mainmodularpptheorem} \label{thm:mainmodularpptheorem}
For any finite field $\F$, every vertically $4$-connected matroid with a modular $\PG(2, \F)$-restriction is $\F$-representable.
\end{restatable}

We will prove the following stronger theorem as a consequence of \autoref{thm:mainmodularpptheorem}.

\begin{restatable}{theorem}{generalform} \label{thm:generalform}
 If $M$ is a vertically $4$-connected matroid with a modular restriction $N$ of rank at least three, then every representation of $N$ over a finite field $\F$ extends to an $\F$-representation of $M$. Moreover, such an $\F$-representation of $M$ is unique up to row operations and column scaling.
\end{restatable}

The main focus of this paper is to prove \autoref{thm:mainmodularpptheorem}, and we will derive \autoref{thm:generalform} from it in \autoref{sec:generalform}.
We also prove the following corollary (for a prime power $q$, we denote by $\F_q$ the field of order $q$).

\begin{restatable}{corollary}{representableorline} \label{cor:representableorline}
 For any prime power $q$, any vertically $4$-connected matroid with a $\PG(2, \F_q)$-restriction is either $\F_q$-representable or has a $U_{2, q^2+1}$-minor.
\end{restatable}

The converse of \autoref{thm:mainmodularpptheorem} is well-known: a restriction isomorphic to a projective geometry over $\F$ is modular in any $\F$-representable matroid (see \cite[Corollary 6.9.6]{Oxley}).
For the two-element field $\F_2$, a much stronger result than \autoref{thm:mainmodularpptheorem} holds.

\begin{theorem}[Seymour, \cite{Seymour:onminorsofnonbinarymatroids}] \label{thm:triangle}
Every $3$-connected matroid with a modular $\PG(1, \F_2)$-restriction is binary.
\end{theorem}

\autoref{thm:triangle} follows from Seymour's result that if $M$ is a $3$-connected non-binary matroid and $a, b \in E(M)$, then $M$ has a $U_{2, 4}$-minor whose ground set contains $a$ and $b$ \cite{Seymour:onminorsofnonbinarymatroids}.
Let $M$ be a $3$-connected matroid with a restriction $N \cong \PG(1, \F_2)$ such that $E(N) = \{a,b,c\}$. If $M$ is non-binary, then it has a $U_{2,4}$-minor containing $a$ and $b$, and since $c \in \cl_M(\{a,b\})$ it has a $U_{2,4}$-minor with ground set $\{a,b,c,d\}$, for some $d \in E(M)$.
Then $E(M)$ contains a set $X$ such that $d \in \cl_{M / X}(\{a, b\})$ but $d$ is not a loop or parallel to $a, b$ or $c$ in $M / X$. This means that $r_M(\{a, b\}) + r_M(X \cup \{d\}) - r_M(\{a,b\} \cup X \cup \{d\}) = 1$ but $r_M(\{a,b,c\} \cap \cl_M(X \cup \{d\})) = 0$, a contradiction to the modularity of $N$.
So $M$ is a binary matroid.

When $\F$ is a finite field of order greater than two, however, to show that a matroid $M$ with a modular $\PG(2, \F)$-restriction is $\F$-representable we need $M$ to be vertically $4$-connected rather than just $3$-connected. 
For any prime power $q > 2$, we can pick a prime $p$ less than $q$ such that $p$ does not divide $q$.
There is a rank-$4$ matroid obtained from $\PG(2, \F_q)$ and $\PG(2, \F_p)$ by identifying the elements of a line in $\PG(2, \F_p)$ with $p+1$ collinear elements in $\PG(2, \F_q)$ that is $3$-connected and has a modular $\PG(2, \F_q)$-restriction, but is not $\F_q$-representable.

To prove \autoref{thm:mainmodularpptheorem}, we first prove the following main lemma, whose proof consists of Sections~\ref{sec:duality} to \ref{sec:connectivity}.

\begin{restatable}{lemma}{keylemma} \label{lem:keylemma}
For any finite field $\F$, if $M_0$ is a $3$-connected, non-$\F$-representable matroid with a modular restriction $N_0 \cong \PG(2,\F)$, then $M_0$ has a $3$-connected, non-$\F$-representable minor $M$ such that $N_0$ is a restriction of $M$ and $\lambda_M(E(N_0)) = 2$.
\end{restatable}

In the final section of this paper, we consider the excluded minors of the classes of matroids representable over each finite field.
Geelen, Gerards, and Whittle made the following conjecture about them.

\begin{conjecture}[Geelen, Gerards, Whittle, \cite{GeelenGerardsWhittle:onrotas}] \label{conj:noexcludedminorhasapp}
 For each finite field $\F$, no excluded minor for the class of $\F$-representable matroids has a $\PG(2, \F)$-minor.
\end{conjecture}

They were able to prove that for each finite field $\F$, there is an integer $n$ so that no excluded minor for the class of $\F$-representable matroids has a $\PG(n, \F)$-minor \cite{GeelenGerardsWhittle:onrotas}.
On the other hand, as a corollary of \autoref{lem:keylemma}, we prove the following special case of \autoref{conj:noexcludedminorhasapp}.

\begin{restatable}{corollary}{noexcludedminorhasapp} \label{cor:noexcludedminorhasapp}
 For any finite field $\F$, no excluded minor for the class of $\F$-representable matroids has a $\PG(2, \F)$-restriction.
\end{restatable}

\section{Definitions}

We recall some matroid connectivity terminology. We define the \defn{connectivity function}, $\lambda_M$, on subsets of the ground set of a matroid $M$ by
\[ \lambda_M(X) = r(X) + r(E(M) \setminus X) - r(M). \]
This is equivalent to $\lambda_M(X) = r(X) + r^*(X) - |X|$.

A partition $(A, B)$ of $E(M)$ is called an \defn{$\ell$-separation} if $|A|,|B| \geq \ell$ and $\lambda_M(A) < \ell$ (note that $\lambda_M(A) = \lambda_M(B)$),
and a matroid is \defn{$k$-connected} if it has no $\ell$-separation for any $\ell < k$.

We call a $2$-separation $(A,B)$ of $M$ an \defn{internal $2$-separation} if $|A|, |B| \geq 3$, and we say that a matroid is \defn{internally $3$-connected} if it is connected and has no internal $2$-separations. Equivalently, a connected matroid is internally $3$-connected if for every $2$-separation $(A,B)$, either $A$ or $B$ is a parallel or series pair.

An $\ell$-separation $(A, B)$ of $M$ is called a \defn{vertical $\ell$-separation} if $r_M(A) < r(M)$ and $r_M(B) < r(M)$.
A matroid $M$ is \defn{vertically $4$-connected} if it has no vertical $\ell$-separations with $\ell < 4$.
If $M$ is vertically $4$-connected, then the simplification of $M$, $\si(M)$, is $3$-connected.

We define the \defn{local connectivity} of two sets $A, B \subseteq E(M)$ in a matroid $M$ to be
\[ \sqcap_M(A, B) = r_M(A) + r_M(B) - r_M(A \cup B), \]
and two sets $A$ and $B$ are called \defn{skew} if $\sqcap_M(A, B) = 0$.

\section{Modular sums} \label{sec:modularsum}


Suppose that the ground sets of two matroids $M_1$ and $M_2$ intersect on a set $T$ such that $M_1|T$ is a modular restriction of $M_1$. We can then define a certain matroid on $E(M_1) \cup E(M_2)$ that generalizes the notions of the direct sum and the $2$-sum. The following construction was introduced by Brylawski in 1975.

\begin{proposition}[Brylawski, \cite{Brylawski}] \label{prop:modularsumexists}
Let $M_1$ and $M_2$ be matroids and let $T = E(M_1) \cap E(M_2)$. If $M_1|T = M_2|T$ and $M_1|T$ is modular in $M_1$, then there is a unique matroid $M$ such
that $E(M) = E(M_1) \cup E(M_2)$, $M | E(M_1) = M_1$, $M | E(M_2) = M_2$, and $r(M) = r(M_1) + r(M_2) - r(M_1|T)$.
Moreover, $F \subseteq E(M)$ is a flat of $M$ if and only if $F \cap E(M_1)$ is a flat of $M_1$ and $F \cap E(M_2)$ is a flat of $M_2$, and the rank of a flat $F$ is
\[ r_M(F) = r_{M_1}(F \cap E(M_1)) + r_{M_2}(F \cap E(M_2)) - r_{M_1}(F \cap T). \]
\end{proposition}

We call the matroid $M$ obtained as in \autoref{prop:modularsumexists} the \defn{modular sum} of $M_1$ and $M_2$ and denote it by $M_1 \oplus_m M_2$.
We note that the matroid $M_1 \oplus_m M_2$ is often called the \emph{generalized parallel connection} \cite{Brylawski}.

When $M_1$ and $M_2$ are matroids on disjoint ground sets, $M_1 \oplus_m M_2$ is equal to the direct sum $M_1 \oplus M_2$. When $M_1$ and $M_2$ are matroids whose ground sets intersect in a single element $e$, then the $2$-sum of $M_1$ and $M_2$ is $M_1 \oplus_2 M_2 = (M_1 \oplus_m M_2) \d e$.

We remark that the rank function for $M$ given in \autoref{prop:modularsumexists} tells us that for any $e \in E(M_2) \setminus T$, $M \d e = M_1 \oplus_m (M_2 \d e)$, and for any $e \in E(M_2) \setminus \cl_{M_2}(T)$, $M / e = M_1 \oplus_m (M_2 / e)$.

We now state four facts about modular sums; the first was proved by Brylawski and concerns their representability. A matroid $M$ is called \defn{uniquely representable} over a field $\F$ if it is $\F$-representable and any $\F$-representation of $M$ can be transformed into any other by row operations, scaling columns, and applying an automorphism of $\F$ to all its entries.

\begin{proposition}[Brylawski, \cite{Brylawski}] \label{prop:modularsumpreservesrepresentability}
If $M = M_1 \oplus_m M_2$ is the modular sum of two matroids $M_1$ and $M_2$ that are representable over a field $\mathbb{F}$ and $M|(E(M_1) \cap E(M_2))$ is uniquely representable over $\mathbb{F}$, then $M$ is representable over $\mathbb{F}$.
\end{proposition}

\autoref{prop:modularsumpreservesrepresentability} is especially useful along with the following theorem, sometimes called the Fundamental Theorem of Projective Geometry (for a proof, see \cite[Theorem 5.4.8]{Berger}).

\begin{fundamentaltheoremofprojectivegeometry}
For each finite field $\F$ and integer $n \geq 3$, the projective geometry $\PG(n-1, \F)$ is uniquely representable over $\F$.
\end{fundamentaltheoremofprojectivegeometry}

The next three facts concern connectivity properties of modular sums.

\begin{proposition} \label{prop:connectivitypropertiesofmodularsum}
If $M = M_1 \oplus_m M_2$ is the modular sum of matroids $M_1$ and $M_2$ with $T = E(M_1) \cap E(M_2)$, then
\begin{enumerate}[(i)]
\item for any $X \subseteq E(M_1) \setminus T$, $r_M^*(X) = r_{M_1}^*(X)$ and $\lambda_M(X) = \lambda_{M_1}(X)$, and
\item for any $X \subseteq E(M_2) \setminus T$, $r_M^*(X) = r_{M_2}^*(X)$ and $\lambda_M(X) = \lambda_{M_2}(X)$.
\end{enumerate}
\end{proposition}

\begin{proof}
We let $X \subseteq E(M_1) \setminus T$ and compute the corank of $X$ in $M$.
We see that $r_M^*(X) = |X| - r(M) + r(M \d X)$ is equal to
\[ |X| - (r(M_1) + r(M_2) - r(M_1|T)) + (r(M_1 \d X) + r(M_2) - r(M_1|T)), \]
which is equal to $r_{M_1}^*(X)$.
Hence $\lambda_M(X) = r_M(X) + r_M^*(X) - |X| = r_{M_1}(X) + r_{M_1}^*(X) - |X| = \lambda_{M_1}(X)$.
The same proof shows that a set $X \subseteq E(M_2) \setminus T$ satisfies $r_M^*(X) = r_{M_2}^*(X)$ and $\lambda_M(X) = \lambda_{M_2}(X)$.
\end{proof}

This next fact is due to Brylawski.

\begin{proposition}[Brylawski, \cite{Brylawski}] \label{prop:separatorismodularsum}
 If $M$ is a matroid with a modular restriction $N$ and $M / E(N)$ is not connected, then $M$ is a modular sum of two proper restrictions whose ground sets meet in $E(N)$.
\end{proposition}

Finally, we have a converse to \autoref{prop:separatorismodularsum}.

\begin{proposition} \label{prop:contractingseparates}
If $M = M_1 \oplus_m M_2$ is the modular sum of matroids $M_1$ and $M_2$ with $T = E(M_1) \cap E(M_2)$, then $(E(M_1) \setminus T, E(M_2) \setminus T)$ is a $1$-separation of $M / T$.
\end{proposition}

\begin{proof}
Let $X_1 \subseteq E(M_1) \setminus T$. Then $r_{M / T}(X_1) = r_{M_1}(X_1 \cup T) - r_M(T)$. Similarly, for $X_2 \subseteq E(M_2) \setminus T$, we have $r_{M / T}(X_2) = r_{M_2}(X_2 \cup T) - r_M(T)$.
Also, $r_{M / T}(X_1 \cup X_2) = r_M(X_1 \cup X_2 \cup T) - r_M(T)$, and this is equal to $r_{M_1}(X_1 \cup T) + r_{M_2}(X_2 \cup T) - 2r_M(T)$, which is $r_{M/T}(X_1) + r_{M/T}(X_2)$.

Hence for any $X \subseteq E(M / T)$, $r_{M / T}(X) = r_{M / T}(X \cap E(M_1)) + r_{M / T}(X \cap E(M_2))$.
\end{proof}

\section{Duality} \label{sec:duality}

A \defn{deletion pair} in a $3$-connected matroid $M$ is a pair of elements $\{x, y\}$ such that $M \d x$ and $M\d y$ are $3$-connected and $M \d x, y$ is internally $3$-connected.
A \defn{contraction pair} in $M$ is a deletion pair in the dual, $M^*$.

The proofs in Sections~\ref{sec:strands} and \ref{sec:connectivity} will require a counterexample to \autoref{lem:keylemma} that has a deletion pair.
However, we will be able to prove, in \autoref{sec:deletionpair}, only that it contains either a deletion pair or a contraction pair. We would therefore like a way to show that if a counterexample with a contraction pair exists, then there is another one with a deletion pair.
In this section, we describe a useful matroid construction involving modular sums that will let us prove this fact in \autoref{sec:deletionpair}.

We note that when $B$ is a basis of a matroid $M$, any representation of $M$ can be transformed by row operations into one where the columns indexed by $B$ are an identity matrix. This is called a representation in \defn{standard form with respect to $B$}.
We recall that if we take any $\F$-representation of a matroid $M$ in standard form $(I \; A)$, then $(A^T \; I)$ is an $\F$-representation of its dual, $M^*$ (see \cite[Theorem 2.2.8]{Oxley}).

We fix a finite field $\F$ and let $N_0$ and $N_1$ be matroids isomorphic to $\PG(2, \F)$ on disjoint ground sets. We let $\varphi : E(N_0) \rightarrow E(N_1)$ be an isomorphism between $N_0$ and $N_1$.

We choose some basis $B_0$ of $N_0$ and let $B_1^* = E(N_1) \setminus \varphi(B_0)$; so $B_1^*$ is a basis of $N_1^*$.
We choose $A$ to be a matrix such that $(I \; A)$ is an $\F$-representation of $\PG(2, \F)$ in standard form with respect to $B_0$; note that $(A^T \; I)$ is an $\F$-representation of $N_1^*$ in standard form with respect to the basis $B_1^*$ . We define the $\F$-matrix $C$ with columns indexed by $E(N_0) \cup E(N_1)$ by

\[ C = \bordermatrix{& B_0 & E(N_0) \setminus B_0 & E(N_1) \setminus B_1^* & B_1^* \cr & I & A & I & 0 \cr & 0 & 0 & A^T & I}. \]

We denote by $R = M_{\F}(C)$ the matroid represented by $C$ over $\F$.
We observe that $R \d E(N_1) = N_0$ and $R / E(N_0) = N_1^*$.
Furthermore, since $R$ is $\F$-representable, $N_0$ is a modular restriction of $R$.

We can now state the main result of this section.

\begin{proposition} \label{prop:dualledmatroid}
If $M_0$ is a $3$-connected, non-$\mathbb{F}$-representable matroid with $N_0$ as a restriction and $\lambda_{M_0}(E(N_0)) = 3$, then $M_1 = ((R \oplus_m M_0) \d E(N_0))^*$ is internally $3$-connected with all parallel pairs containing an element of $E(N_1)$, $M_1$ is non-$\mathbb{F}$-representable, $M_1$ has $N_1$ as a restriction, and $\lambda_{M_1}(E(N_1)) = 3$. Moreover, $N_1$ is modular in $M_1$ if and only if $N_0$ is modular in $M_0$.
\end{proposition}

We prove \autoref{prop:dualledmatroid} through a sequence of lemmas.
First, we show that the matroid $M_1$ has $N_1$ as a restriction, then we show that $M_1$ is not $\F$-representable, and finally we prove the required connectivity properties.

For a set $S$ in a matroid $M$ and sets $X \subseteq E(M) \setminus S$, $Y \subseteq S$, we say that \defn{$Y$ subjugates $X$ relative to $S$ in $M$} if 
\[ \sqcap_M(X,S) = \sqcap_M(E(M) \setminus S, Y) = \sqcap_M(X, Y). \]
If for all $X \subseteq E(M) \setminus S$ there is a set $Y \subseteq S$ that subjugates $X$ relative to $S$ in $M$, then we say that $S$ \defn{subjugates} $M$.
Whenever $N$ is a modular restriction of a matroid $M$, the set $E(N)$ subjugates $M$. In particular, for any $X \subseteq E(M) \setminus E(N)$, $\cl_M(X) \cap E(N)$ subjugates $X$ relative to $E(N)$. However, unlike modularity, the property of subjugating a matroid is invariant under matroid duality, as we now show.

\begin{proposition} \label{prop:dualityofwm}
Let $M$ be a matroid and $S \subseteq E(M)$. For any $X \subseteq E(M) \setminus S$ and $Y \subseteq S$, if $Y$ subjugates $(E(M) \setminus S) \setminus X$
relative to $S$ in $M$, then $S \setminus Y$ subjugates $X$ relative to $S$ in $M^*$.
\end{proposition}

\begin{proof}
We start with the following claim:

\begin{claim} \label{clm:dualityofwm-1}
If $(A,B,C)$ is a partition of $E(M)$, $\lambda_M(A) = \sqcap_M(A,B) + \sqcap_{M^*}(A, C)$.
\end{claim}

$\sqcap_M(A,B) + \sqcap_{M^*}(A,C)$ is equal to
\begin{align*}
& r_M(A) + r_M(B) - r_M(A \cup B) + r_{M^*}(A) + r_{M^*}(C) - r_{M^*}(A \cup C) \\
&= \lambda_M(A) + |A| - r_{M / B \d C}(A) - r_{M^* / C \d B}(A) \\
&= \lambda_M(A) + |E(M/B\d C)| - r(M / B \d C) - r((M / B \d C)^*) \\
&= \lambda_M(A).
\end{align*}

Let $X \subseteq E(M) \setminus S$ and $Y \subseteq S$ such that $Y$ subjugates $(E(M) \setminus S) \setminus X$ relative to $S$ in $M$. Then
\[ \sqcap_M((E(M) \setminus S) \setminus X, S) = \sqcap_M(E(M) \setminus S, Y) = \sqcap_M((E(M) \setminus S) \setminus X, Y). \]

By (\ref*{clm:dualityofwm-1}) we have $\sqcap_M((E(M) \setminus S) \setminus X, S) = \lambda_M(S) - \sqcap_{M^*}(X, S)$ and $\sqcap_M(E(M) \setminus S, Y) =
\lambda_M(E(M) \setminus S) - \sqcap_{M^*}(E(M) \setminus S, S \setminus Y)$, implying that $\sqcap_{M^*}(X, S) = \sqcap_{M^*}(E(M) \setminus S, S \setminus
Y)$.

Similarly, from the equality $\sqcap_M(E(M) \setminus S, Y) = \sqcap_M((E(M) \setminus S) \setminus X, Y)$ and (\ref*{clm:dualityofwm-1}) we have $\lambda_M(Y) -
\sqcap_{M^*}(Y, S \setminus Y) = \lambda_M(Y) - \sqcap_{M^*}((S \setminus Y) \cup X, Y)$. From this we have $- r_{M^*}(S \setminus Y) + r_{M^*}(S) = -
r_{M^*}((S \setminus Y) \cup X) + r_{M^*}(S \cup X)$ and hence $\sqcap_{M^*}(X, S) = \sqcap_{M^*}(X, S \setminus Y)$. 
This proves that $S \setminus Y$ subjugates $X$ relative to $S$ in $M^*$.
\end{proof}

Now we can show that the matroid $M_1$ has $N_1$ as a modular restriction.

\begin{lemma} \label{lem:preservesmodularity}
If $M_0$ is a matroid with $N_0$ as a restriction and $\lambda_{M_0}(E(N_0)) = 3$, then $M_1 = ((R \oplus_m M_0) \d E(N_0))^*$ has $N_1$ as a restriction.
Moreover, if $N_0$ is modular in $M_0$ then $N_1$ is modular in $M_1$.
\end{lemma}

\begin{proof}
The fact that $\lambda_{M_0}(E(N_0)) = 3$ means that $N_1^* = R / E(N_0) = M_1^* /  (E(M_0) \setminus E(N_0))$, so $N_1 = M_1 | E(N_1)$ is a restriction of
$M_1$. We now assume that $N_0$ is modular in $M_0$.

We observe that $E(N_1)$ subjugates $R^*$, for $R^*$ is $\F$-representable so $R^* | E(N_1) = N_1 \cong \PG(2,\F)$ is modular in $R^*$. By
\autoref{prop:dualityofwm}, $E(N_1)$ also subjugates $R$.
Then the fact that $E(N_0)$ is coindependent in $R$ implies that for a set $X \subseteq E(M_0) \setminus E(N_0)$, there is a set $Y \subseteq E(N_1)$ with
$\cl_R(Y) \cap E(N_0) = \cl_{M_0}(X) \cap E(N_0)$. Let $r = r_{N_0}(\cl_R(Y) \cap E(N_0))$.
The modularity of $N_0$ in $M_0$ and in $R$ implies that $\sqcap_{M_0}(X, E(N_0)) = r$ and $\sqcap_R(Y, E(N_0)) = r$.
By \autoref{prop:contractingseparates}, $(E(M_0) \setminus E(N_0), E(N_1))$ is a separation of $(R \oplus_m M_0) / E(N_0)$.
This fact, with the coindependence of $E(N_0)$ in $M_0$ and $R$, implies that
\[ \sqcap_{(R \oplus_m M_0) \d E(N_0)}(X, E(N_1)) = \sqcap_{M_0}(X, E(N_0)) = r, \text{and} \]
\[ \sqcap_{(R \oplus_m M_0) \d E(N_0)}(Y, E(M_0) \setminus E(N_0)) = \sqcap_R(Y, E(N_0)) = r. \]
Moreover we have $r \leq \sqcap_{M_1^*}(X, Y) \leq \sqcap_{M_1^*}(X, E(N_1)) = r$.
This proves that $E(N_1)$ subjugates $M_1^*$, and then \autoref{prop:dualityofwm} implies that $E(N_1)$ subjugates $M_1$.

Let $F$ be a flat of $M_1$, let $X = F \setminus E(N_1)$, and let $Y' \subseteq E(N_1)$ be a set that subjugates $X$ relative to $E(N_1)$ in $M_1$. As $E(N_1)$
is independent in $R$, it is coindependent in $M_1$. Therefore, $\sqcap_{M_1}(E(M_1) \setminus E(N_1), Y') = r_{M_1}(Y')$, so $\sqcap_{M_1}(X, Y') =
r_{M_1}(Y')$ and $\sqcap_{M_1}(X, E(N_1)) = r_{M_1}(Y')$.
The first equation implies that $Y' \subseteq \cl_{M_1}(X)$, which with the second implies that $\sqcap_{M_1}(X, E(N_1)) \leq r_{M_1}(\cl_{M_1}(X) \cap
E(N_1))$.
Then as $F$ is closed, we have $\sqcap_{M_1}(X, E(N_1)) \leq r_{M_1}(\cl_{M_1}(X) \cap F \cap E(N_1)) \leq \sqcap_{M_1}(X, F \cap E(N_1))$. Thus
$\sqcap_{M_1}(X, E(N_1)) = \sqcap_{M_1}(X, F \cap E(N_1))$, implying that $r_{M_1}(E(N_1)) - r_{M_1}(F \cup E(N_1)) = r_{M_1}(F \cap E(N_1)) -
r_{M_1}(F)$. This proves that $N_1$ is a modular restriction of $M_1$.
\end{proof}

Next, we show that for a matroid $M_0$ with $N_0$ as a restriction, the operation $M_0 \mapsto ((R \oplus_m M_0) \d E(N_0))^*$ is an involution, in the
following sense.

\begin{lemma} \label{lem:involution}
If $M_0$ is a matroid with $N_0$ as a restriction, $\lambda_{M_0}(E(N_0)) = 3$, and $M_1 = ((R \oplus_m M_0) \d E(N_0))^*$, then $M_0 = ((R^* \oplus_m
M_1) \d E(N_1))^*$.
\end{lemma}

\begin{proof}

By \autoref{lem:preservesmodularity}, $M_1$ has $N_1$ as a restriction, so we can define the matroid $M_2 = ((R^* \oplus_m M_1) \d E(N_1))^*$.
Let $N_0'$ be a copy of $N_0$ on a disjoint ground set.
Let $R'$ and $M_0'$ be the matroids obtained from $R$ and $M_0$, respectively, by relabelling each element in $E(N_0)$ by its copy in $E(N_0')$.
The next claim is a straightforward calculation.

\begin{claim} \label{clm:involution-0}
$r((R^* \oplus_m (R' \oplus_m M_0')^*)^*) = r((R^* \oplus_m R'^*)^*) + r(M_0') - r(N_0')$.
\end{claim}

\begin{claim} \label{clm:involution-1}
$(R^* \oplus_m (R' \oplus_m M_0')^*)^* = (R^* \oplus_m R'^*)^* \oplus_m M_0'$.
\end{claim}

We observe that $(R^* \oplus_m (R' \oplus_m M_0')^*)^*$, when restricted to the sets $E((R^* \oplus_m R'^*)^*)$ and $E(M_0')$, yields the matroids $(R^* \oplus_m R'^*)^*$ and $M_0'$, respectively. Along with (\ref*{clm:involution-0}), this proves (\ref*{clm:involution-1}) because of the uniqueness of the modular sum as in
\autoref{prop:modularsumexists}.

\begin{claim} \label{clm:involution-2}
$M_2 = ((R^* \oplus_m R'^*)^* \oplus_m M_0') \d E(N_0') / E(N_1)$.
\end{claim}

We have $M_2 = ((R^* \oplus_m ((R' \oplus_m M_0') \d E(N_0'))^*) \d E(N_1))^* = (R^* \oplus_m ((R' \oplus_m M_0')^* / E(N_0')))^* / E(N_1)$.
Since $N_1^* = R' / E(N_0') = (R' \oplus_m M_0') / E(M_0')$, $N_1$ is a restriction of $(R' \oplus_m M_0')^*$. Hence we have $R^* \oplus_m ((R' \oplus_m
M_0')^* / E(N_0')) = (R^* \oplus_m (R' \oplus_m M_0')^*) / E(N_0')$. This implies that $M_2 = ((R^* \oplus_m (R' \oplus_m M_0')^*) / E(N_0'))^* /
E(N_1)$, and (\ref*{clm:involution-2}) now follows from (\ref*{clm:involution-1}).

\begin{claim} \label{clm:involution-3}
$(R^* \oplus_m R'^*)^* / E(N_1)$ is the matroid obtained from $N_0'$ by adding each element of $E(N_0)$ in parallel to its copy in $E(N_0')$.
\end{claim}

We observe that $R^*$ is represented over $\F$ by the matrix
\[
 \begin{array}{ccccc@{}l@{\;}l}
                     & \multicolumn{2}{c}{E(N_0)} & \multicolumn{2}{c}{E(N_1)} &                    &   \cr
  \ldelim({2}{0.5em} & A^T & I                    & 0 & 0                      & \rdelim){2}{0.5em} & \multirow{2}*{.} \cr
                     & I   & 0                    & I & A                      &                    &   \cr
 \end{array}
\]
Thus we have the following representation of $R^* \oplus_m R'^*$
\[
 \begin{array}{ccccccc@{}l@{\;}l}
                    & \multicolumn{2}{c}{E(N_0)} & \multicolumn{2}{c}{E(N_1)} & \multicolumn{2}{c}{E(N_0')} &                    & \cr
 \ldelim({3}{0.5em} & A^T & I                    & 0 & 0                      & 0 & 0                       & \rdelim){3}{0.5em} & \multirow{3}*{,} \cr 
                    & I & 0                      & I & A                      & I & 0                       &                    & \cr 
                    & 0 & 0                      & 0 & 0                      & A^T & I                     &                    & \cr
 \end{array}
\]
hence $(R^* \oplus_m R'^*)^*$ has the representation
\[
 \begin{array}{ccccccc@{}l@{\;}l}
                    & \multicolumn{2}{c}{E(N_0)} & \multicolumn{2}{c}{E(N_1)} & \multicolumn{2}{c}{E(N_0')} &                    & \cr
 \ldelim({3}{0.5em} & I & A                      & I & 0                      & 0 & 0                       & \rdelim){3}{0.5em} & \multirow{3}*{,} \cr
                    & 0 & 0                      & A^T & I                    & 0 & 0                       &                    & \cr 
                    & 0 & 0                      & I & 0                      & I & A                       &                    & \cr
 \end{array}
\]
which is row-equivalent to
\[
 \begin{array}{ccccccc@{}l@{\;}l}
                    & \multicolumn{2}{c}{E(N_0)} & \multicolumn{2}{c}{E(N_1)} & \multicolumn{2}{c}{E(N_0')} &                    & \cr
 \ldelim({3}{0.5em} & I & A                      & 0 & 0                      & -I & -A                     & \rdelim){3}{0.5em} & \multirow{3}*{,} \cr
                    & 0 & 0                      & 0 & I                      & -A^T & -A^TA                &                    & \cr 
                    & 0 & 0                      & I & 0                      & I & A                       &                    & \cr
 \end{array}
\]
which proves (\ref*{clm:involution-3}).
\\

By (\ref*{clm:involution-2}) we have $M_2 = ((R^* \oplus_m R'^*)^* \oplus_m M_0') \d E(N_0') / E(N_1)$. By (\ref*{clm:involution-3}), $(R'^* \oplus_m R^*)^* / E(N_1)$ has
$N_0'$ as a restriction, and $M_2 = ((R'^* \oplus_m R^*)^* / E(N_1) \oplus_m M_0') \d E(N_0') = M_0$, as required.
\end{proof}

The next lemma shows that our modular sum operation preserves $\F$-representability.

\begin{lemma} \label{lem:preservesrepresentability}
If $M_0$ is a matroid with $N_0$ as a restriction and $\lambda_{M_0}(E(N_0)) = 3$, then $M_1 = ((R \oplus_m M_0) \d E(N_0))^*$ is $\F$-representable if and only
if $M_0$ is.
\end{lemma}

\begin{proof}
First we assume that $M_0$ is $\F$-representable.
Then $M_0$ has an $\F$-representation of the form
\[ D = \bordermatrix{& E(M_0) \setminus E(N_0) & B_0 & E(N_0) \setminus B_0 \cr & A_1 & 0 & 0 \cr & A_2 & I & A} \]
where $(I \; A)$ is, as before, the matrix representing $N_0 \cong \PG(2,\F)$.
Then 
\[ \bordermatrix{& E(M_0) \setminus E(N_0) & B_0 & E(N_0) \setminus B_0 & B_1^* & E(N_1) \setminus B_1^* \cr & A_1 & 0 & 0 & 0 & 0 \cr & A_2 & I & A & 0
& I \cr & 0 & 0 & 0 & I & A^T} \]
is an $\F$-representation of $R \oplus_m M_0$ and therefore $M_1 = ((R \oplus_m M_0) \d E(N_0))^*$ is $\F$-representable.

Next we assume that $M_1$ is $\F$-representable.
By \autoref{lem:preservesmodularity}, $M_1$ has $N_1$ as a restriction, and the matroid $R^* \oplus_m M_1$ exists.
By the same argument we applied to $R \oplus_m M_0$, $R^* \oplus_m M_1$ is $\F$-representable.
But by \autoref{lem:involution}, $M_0 = ((R^* \oplus_m M_1) \d E(N_1))^*$, and so $M_0$ is $\F$-representable.
\end{proof}

Next, we show that $M_1$ is internally $3$-connected.

\begin{lemma} \label{lem:preservesstability}
Let $M_0$ be a matroid with $N_0$ as a restriction, $\lambda_{M_0}(E(N_0)) = 3$, and $M_1 = ((R \oplus_m M_0) \d E(N_0))^*$. Then $\lambda_{M_1}(E(N_1)) = 3$
and if $M_0$ is $3$-connected then $M_1$ is internally $3$-connected and each of its parallel pairs contains an element of $E(N_1)$.
\end{lemma}

\begin{proof}
$M_1$ has $N_1$ as a restriction and $\lambda_{M_1}(E(N_1)) = \lambda_{R \oplus_m M_0}(E(N_1))$, which is three by
\autoref{prop:connectivitypropertiesofmodularsum}, part (i).

We assume that $M_0$ is $3$-connected. 
If $M_1$ is not connected, it has a separation $(W, Z)$. Since $\lambda_{M_0}(E(N_0)) = 3$, $(R \oplus_m M_0) \d E(N_0) / (E(M_0) \setminus E(N_0)) = R / E(N_0) = N_1^*$, which is $3$-connected. Therefore, we may assume that $E(N_1) \subseteq W$. Then since $E(N_0) \subseteq \cl_R(E(N_1))$, $(W \cup E(N_0), Z)$ is a separation of $R \oplus_m M_0$. This implies that $(W \cup E(N_0) \setminus E(N_1), Z)$ is a separation of $M_0$, a contradiction; so $M_1$ is connected.

If $M_1$ is not internally $3$-connected, then $M_1^* = (R \oplus_m M_0) \d E(N_0)$ has an internal $2$-separation $(U, V)$.
Since $M_1^* / (E(M_0) \setminus E(N_0)) = R / E(N_0) = N_1^*$ is $3$-connected we may assume that all but at most one element of $E(N_1)$
is contained in $U$.
If there is an element $e \in E(N_1) \cap V$ then as $V \setminus \{e\} \subseteq E(M_0)$ we have $r_{M_1^*}(V) = r_{M_1^*}(V \setminus \{e\}) + 1$, implying that $(U \cup \{e\}, V \setminus \{e\})$ is a $2$-separation of $M_1^*$.
Therefore, $M_1^*$ has a $2$-separation $(A, B)$ with $E(N_1) \subseteq A$ (either $A = U$ or $A = U \cup \{e\}$). Since $E(N_0) \subseteq \cl_R(E(N_1))$, $(A \cup E(N_0), B)$ is a $2$-separation of $R \oplus_m M_0$, and thus $(A \cup E(N_0) \setminus E(N_1), B)$ is a separation or $2$-separation of $M_0$, a contradiction.

Suppose that $M_1$ has a parallel pair $X$ that is disjoint from $E(N_1)$. Then $X$ is a series pair of $(R \oplus_m M_0) \d E(N_0)$. Since $E(N_0) \subseteq \cl_{R \oplus_m M_0}(E(N_1))$, $X$ is also a series pair of $R \oplus_m M_0$. But then \autoref{prop:connectivitypropertiesofmodularsum}, part (ii) implies that $X$ is also a series pair of $M_0$, a contradiction. Therefore, $M_1$ is internally $3$-connected and each of its parallel pairs contains an element of $E(N_1)$.
\end{proof}

We conclude this section by remarking that Lemmas~\ref{lem:preservesmodularity}, \ref{lem:involution}, \ref{lem:preservesrepresentability}, and \ref{lem:preservesstability} together prove \autoref{prop:dualledmatroid}.

\section{Finding a deletion pair} \label{sec:deletionpair}

We recall that a deletion pair in a $3$-connected matroid $M$ is a pair of elements $x,y$ such that $M \d x$ and $M\d y$ are $3$-connected and $M \d x,y$ is internally $3$-connected.
The purpose of this section is to show that if there exists a counterexample to \autoref{lem:keylemma}, then there is one that has a deletion pair.
We will start with several useful facts on connectivity. The first, Bixby's Lemma \cite[Theorem 1]{Bixby:asimpletheoremon3connectivity}, is one we will use many times throughout this paper.

\begin{bixbyslemma}
If $M$ is a $3$-connected matroid and $e \in E(M)$ then at least one of $M \d e$ and $M / e$ is internally $3$-connected.
\end{bixbyslemma}

A \defn{triangle} is a three-element circuit and a \defn{triad} is a three-element cocircuit.
Next we state a useful lemma of Tutte \cite{Tutte:connectivityinmatroids}.

\begin{tuttestrianglelemma}
If $T = \{a,b,c\}$ is a triangle in a $3$-connected matroid $M$ with $|E(M)| \geq 4$ then either $M \d a$ is $3$-connected, $M\d b$ is $3$-connected or there is a triad of $M$ that contains $a$ and exactly one of $b$ and $c$.
\end{tuttestrianglelemma}

The next is a corollary of Tutte's Triangle Lemma and is proved in \cite{GeelenGerardsWhittle:onrotas}.

\begin{lemma}[Lemma 2.7, \cite{GeelenGerardsWhittle:onrotas}] \label{lem:lemma2.7}
If $T$ is a triangle in a $3$-connected matroid $M$ with $|E(M)| \geq 4$ then there exists $e \in T$ such that $M \d e$ is internally $3$-connected.
\end{lemma}

A \defn{fan} in a matroid is a sequence $(s_1, s_2, \ldots, s_n)$ of distinct elements such that:
\begin{itemize}
 \item $\{s_i, s_{i+1}, s_{i+2}\}$ is a triangle or a triad for each $i = 1, 2, \ldots, n-2$, and
 \item if $\{s_i, s_{i+1}, s_{i+2}\}$ is a triangle then $\{s_{i+1}, s_{i+2}, s_{i+3}\}$ is a triad, and if $\{s_i, s_{i+1}, s_{i+2}\}$ is a triad then $\{s_{i+1}, s_{i+2}, s_{i+3}\}$ is a triangle, for each $i = 1, 2, \ldots, n-3$.
\end{itemize}

An easy fact about fans is that $\lambda_M(S) \leq 2$ for any set $S$ forming a fan in a matroid $M$. Recall that $\lambda_M(S) = r_M(S) + r^*_M(S) - |S|$.

\begin{lemma} \label{lem:fan3separates}
 If $(s_1, \ldots, s_n)$ is a fan in a matroid $M$, then $\lambda_M(\{s_1, \ldots, s_n\}) \leq 2$.
\end{lemma}

\begin{proof}
By duality we may assume that $\{s_i, s_{i+1}, s_{i+2}\}$ is a triad of $M$ for odd $i$ and a triangle for even $i$. Then $r_M(\{s_1, \ldots, s_n\}) \leq r_M(\{s_1, s_2, s_3\}) + |\{5 \leq i \leq n : i \textrm{ is odd}\}|$.
In $M^*$, we have $r_{M^*}(\{s_1, \ldots, s_n\}) \leq r_{M^*}(\{s_2, s_3, s_4\}) + |\{6 \leq i \leq n : i \textrm{ is even}\}|$.
Hence $\lambda_M(\{s_1, \ldots, s_n\}) \leq r_M(\{s_1, s_2, s_3\}) + r^*_M(\{s_2, s_3, s_4\}) + n - 4 - n \leq 2$.
\end{proof}

An element $e$ of a $3$-connected matroid $M$ is called \defn{essential} if neither $M \d e$ nor $M / e$ is $3$-connected.
The next two results of Oxley and Wu are specializations of the statements of Lemma 8.8.6 and Theorem 8.8.8 in \cite{Oxley}.

\begin{lemma}[Oxley, Wu, \cite{OxleyWu}] \label{lem:oxleywulemma}
 If $M$ is a $3$-connected matroid containing a projective plane restriction and $S = (s_1, \ldots, s_n)$ is a maximal fan in $M$ with $n \geq 4$, then the set of non-essential elements in $S$ is $\{s_1, s_n\}$.
\end{lemma}

\begin{theorem}[Oxley, Wu, \cite{OxleyWu}] \label{thm:oxleywutheorem}
 If $M$ is a $3$-connected matroid and $e$ is an essential element of $M$ that is in a four-element fan, then either 
 \begin{enumerate}[(i)]
 \item $e$ is in a unique maximal fan in $M$, or 
 \item $e$ is in exactly three maximal fans each of which has exactly five elements, the union $X$ of these three fans has exactly six elements, and one of $M | X$ and $M / (E(M) \setminus X)$ is isomorphic to $M(K_4)$.
 \end{enumerate}
\end{theorem}

In the following lemma, we find either a deletion pair or a contraction pair in a matroid.
It is a generalization of a lemma of Geelen, Gerards and Whittle \cite[Lemma 2.8]{GeelenGerardsWhittle:onrotas}, and we partly follow the outline of their proof. We make extensive use of Tutte's Triangle Lemma, as well as the following consequence of it: when a matroid has no element in a both a triangle and a triad, every triangle contains at least two elements whose deletion leaves the matroid $3$-connected.

\begin{lemma} \label{lem:notlemma2.8}
If $M$ is a $3$-connected matroid such that
\begin{enumerate}[(a)]
\item $M$ has $N_0 \cong \PG(2,\F)$ as a modular restriction with $\lambda_M(E(N_0)) = 3$, \label{hyp:notlemma2.8-hasplane}
\item $M / E(N_0)$ is connected and non-empty, and  \label{hyp:notlemma2.8-onebridge}
\item no element of $M$ is in both a triangle and a triad,  \label{hyp:notlemma2.8-notriangletriad}
\end{enumerate}
then either
\begin{enumerate}[(i)]
\item $M$ has a restriction $K \cong M(K_5)$ with a cocircuit $\{a, b, c, d\}$ such that $M = K \oplus_m (M \d a,b,c,d)$, \label{out:notlemma2.8-modularsum}
\item $M$ has a deletion pair $x,y \in E(M) \setminus E(N_0)$, or \label{out:notlemma2.8-deletionpair}
\item $M$ has a contraction pair $x,y \in E(M) \setminus E(N_0)$. \label{out:notlemma2.8-contractionpair}
\end{enumerate}
\end{lemma}

\begin{proof}
We let $M$ be a counterexample; so $M$ is $3$-connected and satisfies (\ref*{hyp:notlemma2.8-hasplane}), (\ref*{hyp:notlemma2.8-onebridge}), and (\ref*{hyp:notlemma2.8-notriangletriad}), but none of the conclusions (\ref*{out:notlemma2.8-modularsum}), (\ref*{out:notlemma2.8-deletionpair}), or (\ref*{out:notlemma2.8-contractionpair}) hold.
The fact that $M$ is simple and $N_0$ is modular in $M$ means that $E(N_0)$ is closed in $M$.

We let $\Lambda$ denote the set of elements $e \in E(M) \setminus E(N_0)$ such that $M \d e$ is $3$-connected and $\Lambda^*$ the set of elements $e \in E(M) \setminus E(N_0)$ such that $M / e$ is $3$-connected.
The first two claims are straightforward.

\begin{claim}
Let $e, f$ be distinct elements of $E(M) \setminus E(N_0)$. If $e \in \Lambda$ then $M \d e, f$ is not $3$-connected, and if $e \in \Lambda^*$ then $M / e, f$ is not $3$-connected.
\end{claim}

\begin{claim} \label{clm:notlemma2.8-3}
If $N$ is a $3$-connected matroid with $|E(N)| \geq 4$ and there are elements $e, f$ such that $N \d e / f$ is $3$-connected, then either $N / f$ is $3$-connected or there is a triangle of $N$ containing $e$ and $f$.
\end{claim}

\begin{claim} \label{clm:notlemma2.8-4}
Each element of $E(M) \setminus E(N_0)$ is either in $\Lambda \cup \Lambda^*$ or is in a triangle that contains an element of $E(N_0)$ and not in any triangle disjoint from $E(N_0)$.
\end{claim}

Suppose that $e \in E(M) \setminus E(N_0)$ is not in $\Lambda \cup \Lambda^*$ and is either not in a triangle containing an element of $E(N_0)$ or is in
a triangle disjoint from $E(N_0)$.
By Bixby's Lemma, either $M \d e$ or $M / e$ is internally $3$-connected. 
Since neither $M / e$ nor $M \d e$ is $3$-connected, $e$ is either in a triangle or a triad. If it is contained in a triangle, then it is contained in a
triangle disjoint from $E(N_0)$. If it is contained in a triad, then since $N_0$ is $3$-connected, the triad is disjoint from $E(N_0)$.

We assume that $e$ is contained in a triangle $T = \{e, a, b\}$ disjoint from $E(N_0)$. A dual argument covers the case where $T$ is a triad disjoint from $E(N_0)$.
Neither $a$ nor $b$ is in a triad and $M \d e$ is not $3$-connected, so by Tutte's Triangle Lemma, both $M\d a$ and $M \d b$ are $3$-connected. We will prove
(\ref*{clm:notlemma2.8-4}) by showing that $M \d a,b$ is internally $3$-connected so that $M$ satisfies (\ref*{out:notlemma2.8-deletionpair}). Let $(A, B)$ be a
$2$-separation in $M \d e$ with $a \in A$. Then $b \in B$. Since neither $a$ nor $b$ is in a triad, $|A|, |B| \geq 3$. Since $|E(M)| \geq 8$ (from (\ref*{hyp:notlemma2.8-hasplane}) and (\ref*{hyp:notlemma2.8-onebridge})),
by possibly
swapping $A$ and $B$ we may assume $|A| \geq 4$. Note that $(A, B \cup \{e\})$ is a $3$-separation of $M$, and $a \in \cl_M(B \cup \{e\})$. Thus $(A \setminus
\{a\}, B \cup \{e\})$ is a $2$-separation of $M/a$ and hence $(A \setminus \{a\}, B \cup \{e\} \setminus \{b\})$ is a $2$-separation in $M/a \d b$, and it is an
internal $2$-separation. Thus by Bixby's Lemma, $M \d a, b$ is internally $3$-connected, contradicting the fact that (\ref*{out:notlemma2.8-deletionpair}) does not hold for $M$.
This proves (\ref*{clm:notlemma2.8-4}).

\begin{claim} \label{clm:notlemma2.8-4.5}
If $e \in E(M) \setminus E(N_0)$ is in a triad, then $e \in \Lambda^*$, and if $e$ is in a triangle disjoint from $E(N_0)$, then $e \in
\Lambda$.
\end{claim}

If $e \in E(M) \setminus E(N_0)$ is in a triad, then $e \not\in \Lambda$ and by (c) $e$ is not in a triangle, so (\ref*{clm:notlemma2.8-4}) implies that $e \in
\Lambda^*$. If $e \in E(M) \setminus E(N_0)$ is in a triangle disjoint from $E(N_0)$, then $e \not\in \Lambda^*$ and (\ref*{clm:notlemma2.8-4}) implies that $e
\in \Lambda$.

\begin{claim} \label{clm:notlemma2.8-5}
If $T$ is a triangle of $M$ disjoint from $E(N_0)$, then $\Lambda \subseteq T$, and if $T^*$ is a triad of $M$, then $\Lambda^* \subseteq T^*$.
\end{claim}

Suppose that there is a triangle $T$ disjoint from $E(N_0)$ and an element $e$ of $\Lambda \setminus T$.
Then $M \d e$ is $3$-connected, so by \autoref{lem:lemma2.7}, there exists $f \in T$ such that $M\d e,
f$ is internally $3$-connected. Also, since $f$ is in a triangle disjoint from $E(N_0)$, it follows from (\ref*{clm:notlemma2.8-4}) that $M \d f$ is $3$-connected.
Then $e,f$ is a deletion pair of $M$, contradicting the fact that $M$ does not satisfy (\ref*{out:notlemma2.8-deletionpair}).
This proves the first part of (\ref*{clm:notlemma2.8-5}) and the dual argument, along with the fact that all triads of $M$ are disjoint from $E(N_0)$, proves the
second.

\begin{claim} \label{clm:notlemma2.8-6}
If $M$ has a triangle $T$ disjoint from $E(N_0)$, it is unique and $\Lambda = T$, and if $M$ has a triad $T^*$, then it is unique and $\Lambda^* = T^*$.
\end{claim}

This follows immediately from (\ref*{clm:notlemma2.8-4.5}) and (\ref*{clm:notlemma2.8-5}).

\begin{claim} \label{clm:notlemma2.8-7}
If $e \in \Lambda$ and $f \in E(M) \setminus (E(N_0) \cup \{e\})$, then either $M \d e,f$ is not internally $3$-connected, or $e$ is in a triangle containing exactly one element of $E(N_0)$, or $\Lambda$ is a triangle.
\end{claim}

Suppose that $M \d e, f$ is internally $3$-connected. Then $M \d f$ is not $3$-connected, since $M$ does not satisfy (\ref*{out:notlemma2.8-deletionpair}). Let $(A, B)$ be a $2$-separation in
$M \d f$ with $e \in A$. If $|A| = 2$ then $A \cup \{f\}$ is a triad of $M$, contradicting the fact that $M \d e$ is $3$-connected.
Thus $|A| \geq 3$, and $(A \setminus \{e\}, B)$ is a $2$-separation in $M \d e, f$, which is internally $3$-connected, so $|A| = 3$. As $(A, B \cup \{f\})$ is
a $3$-separation of $M$, $A$ is a triangle (not a triad since $e \in \Lambda$). Since $E(N_0)$ is a closed set and $e \not\in E(N_0)$, $A$ contains at most one element of $E(N_0)$. Then either $A$ is a triangle containing exactly one element of $E(N_0)$, or
(\ref*{clm:notlemma2.8-6}) implies that $\Lambda = A$.

\begin{claim} \label{clm:notlemma2.8-8}
If $e \in \Lambda^*$ and $f \in E(M) \setminus (E(N_0) \cup \{e\})$, then either $M / e,f$ is not internally $3$-connected, or $\Lambda^*$ is a triad.
\end{claim}

Suppose that $M / e,f$ is internally $3$-connected. Then $M / f$ is not $3$-connected, since $M$ does not satisfy (\ref*{out:notlemma2.8-contractionpair}). Let $(A,B)$ be a $2$-separation in $M / f$ with $e \in A$. If $|A| = 2$ then $A \cup \{f\}$ is a triangle of $M$, contradicting the fact that $M / e$ is $3$-connected.
Thus $|A| \geq 3$, and $(A \setminus \{e\}, B)$ is a $2$-separation in $M / e,f$, which is internally $3$-connected,
so $|A| = 3$. As $(A, B \cup \{f\})$ is a $3$-separation of $M$, $A$ is a triad (not a triangle since $e \in \Lambda^*$), and (\ref*{clm:notlemma2.8-6})
implies that $ \Lambda^* = A$.

\begin{claim}\label{clm:notlemma2.8-9}
There is no cocircuit $\{a,b,c,d\}$ of $M$ such that $\sqcap_M(\{a, b\}, E(N_0)) = 1$ and either $\sqcap_M(\{a, c\}, E(N_0)) = 1$ or $\sqcap_M(\{c, d\}, E(N_0)) = 1$.
\end{claim}

Let $\{a,b,c,d\}$ be a cocircuit of $M$ with $\sqcap_M(E(N_0), \{a,b,c,d\}) \geq 2$.
If $\lambda_M(\{a,b,c,d\}) = 2$ then by (\ref*{hyp:notlemma2.8-onebridge}), $E(M) = E(N_0) \cup \{a,b,c,d\}$ and $\lambda_M(E(N_0)) = 2$, a contradiction. So $\lambda_M(\{a,b,c,d\}) = 3$, and $\{a,b,c,d\}$ is independent in $M$.

Suppose that $\sqcap_M(E(N_0), \{a,b,c,d\}) = 3$. Then $r_{M / E(N_0)}(\{a, b, c, d\}) = 1$ and since $\{a, b, c, d\}$ is a cocircuit of $M$, it is a rank-one cocircuit of $M / E(N_0)$, which means it is a component of $M / E(N_0)$. Then since $M / E(N_0)$ is connected, $E(M) = E(N_0) \cup \{a,b,c,d\}$.
By the modularity of $N_0$, for each pair of distinct elements $f, g \in \{a, b, c, d\}$, there is an element $e_{fg} \in E(N_0)$ such that $\{f, g, e_{fg}\}$ is a triangle in $M$; let $X$ be the set of these six elements. For each triple of distinct elements $f, g, h \in \{a, b, c, d\}$, $\{e_{fg}, e_{gh}, e_{hf}\}$ is a triangle. Hence $M | (\{a, b, c, d\} \cup X) \cong M(K_5)$. Moreover, $M = (M | (\{a, b, c, d\} \cup X)) \oplus_m N_0$, so outcome (\ref*{out:notlemma2.8-modularsum}) holds. So we may assume that $\sqcap_M(E(N_0), \{a,b,c,d\}) = 2$.

First, we assume that $\sqcap_M(\{a, c\}, E(N_0)) = 1$; then also $\sqcap_M(\{b, c\}, E(N_0)) = 1$ and there are elements $e_{ab}, e_{bc}, e_{ca} \in E(N_0)$ such that $\{a,b,e_{ab}\}$, $\{b,c,e_{bc}\}$ and $\{c,a,e_{ca}\}$ are triangles.
We claim that $a,b,c \in \Lambda$.
If not, by symmetry we may assume that $M \d b$ is not $3$-connected. Then it has a $2$-separation, $(U, V)$, with $e_{ab} \in U$ and $a \in V$. We have $a
\not\in \cl_M(U)$ and $e_{ab} \not\in \cl_M(V)$. Hence since $e_{ca}$ is in triangles of $M \d b$ with each of $e_{ab}$ and $a$, $e_{ca} \in \cl_M(U) \cap
\cl_M(V)$. Then we have $c \in V$ and $e_{bc} \in U$. Since $|V \setminus \{e_{ca}\}| \geq 2$, we may assume that $e_{ca} \in U$.
Now $(U, V \cup \{b\})$ is a $3$-separation of $M$ and $U \cap \cl_M(V \cup \{b\}) = \{e_{ab}, e_{bc},
e_{ca}\}$. So $\lambda_{M / E(N_0)}(U \setminus E(N_0)) = 0$, implying by (\ref*{hyp:notlemma2.8-onebridge}) that $U = E(N_0)$, since $\cl_M(V) \cap E(N_0) = \{e_{ca}\}$. This contradicts the fact that $\lambda_M(E(N_0)) = 3$.
Therefore, $a,b,c \in \Lambda$ and by (\ref*{clm:notlemma2.8-6}), $d$ is not in a triangle of $M$.

We now show that one of $M \d a,b$ and $M \d b,c$ is internally $3$-connected.
Suppose that $M \d b,c$ is not internally $3$-connected. Then $M \d b,c$ has an internal $2$-separation $(W,Z)$. 
Since $|E(M)| > 8$ (by (\ref*{hyp:notlemma2.8-hasplane}) and (\ref*{hyp:notlemma2.8-onebridge})), one of $W$ and $Z$ has size at least four. So since $\{a, d\}$ is a series pair of $M \d b, c$, we may assume that $a, d \in W$.
Since $M \d b$ is $3$-connected, $e_{ca} \in Z$ and $a \not\in \cl_M(Z)$.
Therefore, since $|W| \geq 3$, $(W \setminus \{a\}, Z)$ is a $2$-separation of $M \d b,c / a$. Since $\{c,e_{ca}\}$ is a parallel pair in $M \d b / a$, $(W
\setminus \{a\}, Z \cup \{c\})$ is a $2$-separation of $M \d b / a$.
The unique triangle of $M \d b$ containing $a$ is $\{a, c, e_{ca}\}$, so $W \setminus \{a\}$ is not a parallel pair of $M \d b / a$; nor is it a series pair as $M \d b$ is $3$-connected. Hence $|W \setminus \{a\}| > 2$ and $(W \setminus \{a\}, Z \cup \{c\})$ is an internal $2$-separation of $M \d b / a$. By Bixby's Lemma, $M \d a, b$ is internally $3$-connected.
Therefore, outcome (\ref*{out:notlemma2.8-deletionpair}) holds. 

We may therefore assume that $\sqcap_M(\{c, d\}, E(N_0)) = 1$. We let $e_{ab}, e_{cd} \in E(N_0)$ be the elements such that $\{a, b, e_{ab}\}$ and $\{c, d, e_{cd}\}$ are triangles. Then by Tutte's Triangle Lemma, we may assume that $b, c \in \Lambda$, and by (\ref*{clm:notlemma2.8-6}), $a$ and $d$ are not contained in any triangles except these two.
If $M \d b, c$ is internally $3$-connected, then outcome (\ref*{out:notlemma2.8-deletionpair}) holds, so there is an internal $2$-separation $(W, Z)$ of $M \d b, c$. Since $\{a, d\}$ is a series pair of $M \d b, c$, we may assume that $a, d \in W$. Then since $M \d b$ is $3$-connected, $e_{cd} \in Z$. Since $(W, Z)$ is an internal $2$-separation, $(W \setminus \{d\}, Z)$ is a $2$-separation of $M \d b, c / d$. Then as $c$ is parallel to $e_{cd}$ in $M \d b / d$, $(W \setminus \{d\}, Z \cup \{c\})$ is a $2$-separation of $M \d b / d$. But $d$ is in a unique triangle of $M$, so $W \setminus \{d\}$ is not a parallel pair of $M \d b / d$. Nor is it a series pair since $M \d b$ is $3$-connected. So $|W \setminus \{d\}| > 2$ and $(W \setminus \{d\}, Z \cup \{c\})$ is an internal $2$-separation of $M \d b / d$. By Bixby's Lemma, $M \d b, d$ is internally $3$-connected. Thus either outcome (\ref*{out:notlemma2.8-deletionpair}) holds, or $M \d d$ is not $3$-connected.
We let $(X, Y)$ be a $2$-separation of $M \d d$ with $c \in X$ and $E(N_0) \in Y$ (this exists as $\{c, d, e_{cd}\}$ is a triangle). Since $M \d c$ is $3$-connected, $c$ is not in a triad of $M$ so $|X| > 2$. Thus if $\{a, b\} \subseteq Y$ then as $\{a, b, c\}$ is a triad of $M \d d$, $(X \setminus \{c\}, Y \cup \{c\})$ is also a $2$-separation of $M \d d$; but then $(X \setminus \{c\}, Y \cup \{c, d\})$ is a $2$-separation of $M$, a contradiction. So at least one of $a$ and $b$ is in $X$. Then as $\{a, b, c\}$ is a triad of $M \d d$, $(X \cup \{a, b\}, Y \setminus \{a, b\})$ is also a $2$-separation of $M \d d$, hence we may assume that $a, b \in X$. But then $(X \cup \{d\}, Y)$ is a $3$-separation of $M$. Also, since $e_{cd}$ and $e_{ab}$ are in $Y \cap \cl_M(X \cup \{d\})$, $\sqcap_M(X \cup \{d\}, E(N_0)) = \sqcap_M(X \cup \{d\}, Y) = 2$. Then $X \cup \{d\}$ is a component of $M / E(N_0)$, which is connected, so $Y = E(N_0)$. This contradicts the fact that $\lambda_M(E(N_0)) = 3$.
This proves (\ref*{clm:notlemma2.8-9}).

\begin{claim} \label{clm:notlemma2.8-no4pointline}
 $M$ has no $U_{2, 4}$-restriction whose ground set is not contained in $E(N_0)$.
\end{claim}

Suppose $\{a, b, c, d\}$ is a set of rank two in $M$ not contained in $E(N_0)$. We may assume that $a, b, c \not\in E(N_0)$. By Tutte's Triangle Lemma, we may assume that $M \d a$ is $3$-connected. Then by Tutte's Triangle Lemma applied to the triangle $\{b, c, d\}$ in $M \d a$, we may assume that $M \d a, c$ is $3$-connected. Then $\{a, c\}$ is a deletion pair and $M$ satisfies (\ref*{out:notlemma2.8-deletionpair}), proving (\ref*{clm:notlemma2.8-no4pointline}).

\begin{claim}\label{clm:notlemma2.8-10}
There is at most one triangle of $M$ containing exactly one element of $N_0$.
\end{claim}

By (\ref*{clm:notlemma2.8-no4pointline}), no two triangles that are not contained in $E(N_0)$ can meet in more than one element. Thus there are three ways in which $M$ could have two triangles that each contain exactly one element of $N_0$: either they meet in an element of $E(M) \setminus E(N_0)$, they are disjoint, or they meet in an element of $E(N_0)$.

First, we assume that there are two triangles $\{a, b, d\}$ and $\{a, c, e\}$ where $a \not\in E(N_0)$ and $d,e \in E(N_0)$.
By the modularity of $N_0$, there is an element $f \in E(N_0) \cap \cl_M(\{b, c\})$, so $M$ has a triangle containing $\{b, c\}$. By Tutte's Triangle Lemma and the fact that no element of $M$ is in both a triangle and a triad, at least one of $M \d b$ and $M \d c$ is $3$-connected.
By symmetry between $b$ and $c$ we may assume that $M \d b$ is $3$-connected. Since $M$ does not satisfy (\ref*{out:notlemma2.8-deletionpair}), neither $M \d b,c$ nor $M \d b,a$ is $3$-connected. So by Tutte's Triangle Lemma applied to the triangle $\{a,c,e\}$, $a$ and $c$ are both in triads of $M \d b$.
But since $e$ is not in a triad, $a$ and $c$ are contained in the same triad. Since $a$ and $c$ are not in a triad of $M$, this means that $\{a,b,c\}$ is contained in a $4$-element cocircuit of $M$, contradicting (\ref*{clm:notlemma2.8-9}).

Next, we assume there are two triangles $T_1$ and $T_2$ that do not meet in any element of $E(M) \setminus E(N_0)$. Let $\{a, b, e\}$ and $\{c, d,
f\}$, where $e, f \in E(N_0)$. So $a, b, c, d$ are distinct, but $e$ and $f$ may not be distinct.
By Tutte's Triangle Lemma applied to $\{c, d, f\}$, either $M \d c$ or $M \d d$ is $3$-connected; by symmetry we may assume $M \d c$ is.
Since $M$ does not satisfy (\ref*{out:notlemma2.8-deletionpair}), $M \d c, a$ and $M \d c, b$ are not $3$-connected. Thus by Tutte's Triangle Lemma and the fact that no triad of $M \d c$ contains an element of $N_0$, $\{a, b\}$ is contained in a triad of $M \d c$. Therefore, as $a$ is not in any triad of $M$, $\{c, a, b\}$ is contained
in some $4$-element cocircuit of $M$. Since this cocircuit does not contain exactly one element of $\{c, d, e\}$ and does not contain exactly one element of $N_0$, it contains $d$. So $\{a, b, c, d\}$ is a cocircuit of $M$. If $e$ and $f$ are distinct, then $\sqcap_M(\{a, b, c, d\}, E(N_0)) = 2$, contradicting (\ref*{clm:notlemma2.8-9}). Hence we may assume that $e = f$.

We may assume by Tutte's Triangle Lemma applied to $\{a, b, e\}$ that $M \d a$ is $3$-connected (recall that we have already applied Tutte's Triangle Lemma to $\{c, d, e\}$ to assume that $M \d c$ is $3$-connected). Hence $M \d a, c$ is not internally $3$-connected or $\{a, c\}$ would be a deletion pair. We let $(X, Y)$ be an internal $2$-separation of $M \d a, c$ with $E(N_0) \subseteq \cl_M(X)$.
Since $M \d a$ and $M \d c$ are $3$-connected, we have $b, d \in Y \setminus \cl_M(X)$. Suppose that $|Y| > 3$. Then recall that $\{b, d\}$ is a series pair of $M \d a, c$, so $(X \cup \{b, d\}, Y \setminus \{b, d\})$ is also a $2$-separation of $M \d a, c$. But as $a, c \in \cl_M(X \cup \{b, d\})$, we have a $2$-separation $(X \cup \{a, b, c, d\}, Y \setminus \{b, d\})$ of $M$, a contradiction. Hence $|Y| = 3$, and the unique element $w$ of $Y \setminus \{b, d\}$ lies in $\cl_M(X) \cap \cl_M(\{b, d\})$.
Applying Tutte's Triangle Lemma to $\{b, d, w\}$, one of $M \d b$ and $M \d d$ is $3$-connected; by symmetry we may assume that $M \d b$ is $3$-connected.
Now since $M \d a$ and $M \d b$ are both $3$-connected, we have restored the symmetry between $a$ and $b$, and so by a symmetric argument $\{a, d\}$ is also contained in a triangle; we let $z$ be its third element.
Since $\{a, b, c, d\}$ is a cocircuit of $M$ and $u, w, z \in \cl_M(\{a, b, c, d\})$, it follows that $\{u, w, z\}$ is a triangle of $M$. So by Tutte's Triangle Lemma and the symmetry between $w$ and $z$, we may assume that $M \d w$ is $3$-connected.
Hence $M \d w, a$ is not internally $3$-connected, or $\{w, a\}$ would be a deletion pair. We let $(A, B)$ be an internal $2$-separation of $M \d w, a$ with $E(N_0) \subseteq \cl_M(A)$. Since $M \d w$ and $M \d a$ are $3$-connected and $\{u, w, z\}$ and $\{u, a, b\}$ are triangles, we have $z, b \in B \setminus \cl_M(A)$. Also, since $\{b, w, d\}$ and $\{d, u, c\}$ are triangles, we have $d \in A$ and $c \in \cl_M(A)$.
Suppose that $|B| = 3$ and let $y$ be the unique element of $B \setminus \{b, z\}$. Then $B$ is a triangle of $M$ and $y \in \cl_M(A)$. A triangle cannot contain just one element of the cocircuit $\{a, b, c, d\}$. But $a, d \in A$, and if $c \in B$ then $c$ and $y$ are both in $\cl_M(A) \cap \cl_M(B)$ and hence parallel, a contradiction. This proves that $|B| > 3$.
Now $b$ is in the closure of $B \setminus \{b, z\}$, for if not then $(A \cup \{w, b, z\}, B \setminus \{b, z\}))$ is a $2$-separation of $M \d a$. This is a contradiction because $\{a, b, c, d\}$ is a cocircuit that is disjoint from $B \setminus \{b, z\}$.
This proves (\ref*{clm:notlemma2.8-10}).

\begin{claim} \label{clm:notlemma2.8-notbothsmall}
Either $\Lambda$ is not a triangle or $\Lambda^*$ is not a triad.
\end{claim}

Suppose that $\Lambda$ is a triangle and $\Lambda^*$ is a triad. Then $\Lambda$ and $\Lambda^*$ are disjoint, and by (\ref*{clm:notlemma2.8-4}) and (\ref*{clm:notlemma2.8-10}), $E(M) \setminus E(N_0)$ consists of the union of $\Lambda$, $\Lambda^*$, and at most one triangle, which contains an element of $N_0$. 

Suppose that $E(M) \cup E(N_0) = \Lambda \cup \Lambda^*$. 
Since $\Lambda^*$ has corank two, $\lambda_{M \d \Lambda^*}(E(N_0)) \geq \lambda_M(E(N_0)) - 2 = 1$, so $\sqcap_M(\Lambda, E(N_0)) \geq 1$. By the modularity of $N_0$, this means that $\cl_M(\Lambda)$ contains an element of $N_0$, hence $M$ has a four-point line not contained in $E(N_0)$, contradicting (\ref*{clm:notlemma2.8-no4pointline}).
Therefore, $E(M) \setminus E(N_0)$ is the union of $\Lambda \cup \Lambda^*$ with a triangle $T$ that contains precisely one element, $u$, of $N_0$.

We write $\Lambda = \{a, b, c\}$ and $\Lambda^* = \{d, e, f\}$. Tutte's Triangle Lemma implies that $T$ cannot be disjoint from $\Lambda$, so we may assume that $c \in T$. We let $w$ denote the element of $T \setminus \{u, c\}$.
Since $w \not\in \Lambda$, $M \d w$ is not $3$-connected, and being in a triangle, $w$ is not in a triad, so $M \d w$ is not internally $3$-connected. Let $(X, Y)$ be an internal $2$-separation of $M \d w$ with $E(N_0) \subseteq \cl_M(X)$. Then $c \in Y \setminus \cl_M(X)$ and we may assume $a, b \in Y$. Moreover, we may assume that $\Lambda^*$ is contained in either $X$ or $Y$. If $\Lambda^* \subseteq Y$ then $(E(N_0), \Lambda \cup \Lambda^* \cup \{w\})$ is a $3$-separation of $M$, contradicting the fact that $\lambda_M(E(N_0)) = 3$. So $\Lambda^* \subseteq X$. Hence $\lambda_M(\{a, b, c, w\}) = 2$ and so $\{a, b, c, w\}$ is a cocircuit of $M$.
We let $d, e$, and $f$ denote the elements of $\Lambda^*$. 
Recall that $M / f$ and $M / d$ are $3$-connected; so $M / d, f$ is not internally $3$-connected or we would have a contraction pair. So $M / f$ has a $3$-separation $(A, B)$ with $E(N_0) \subseteq \cl_{M / f}(A)$, $d \in \cl_{M / f}(A) \cap \cl_{M/f}(B)$, and $|B \setminus \{d\}| \geq 3$. If $\Lambda \subseteq \cl_{M/f}(A)$ then $e$ would be a coloop of $M/f$; so we may assume that $\Lambda \subseteq B$.
We assume that $w$ is in $\cl_{M/f}(A)$. Then so is $c$, and hence $e \in B \setminus \cl_{M/f}(A)$ or else $\{a, b\}$ would be a series pair of $M / f$. Then $\{c, d\}$ is contained in $\cl_{M/f}(A) \cap \cl_{M/f}(B)$, which implies that $E(N_0)$ is disjoint from $\cl_{M/f}(B)$ because otherwise $\{c, d\}$ would be contained in a triangle by the modularity of $N_0$. Thus $\{a, b, c, e\}$ and $E(N_0)$ are skew in $M / f$, and so $\{a, b, c, e, f\}$ and $E(N_0)$ are skew in $M$. But then $\sqcap_M(\{a, b, c, d, e, f, w\}, E(N_0)) \leq 2$, contradicting the fact that $\lambda_M(E(N_0)) = 3$. This proves that $w \in B$.
If $e \in A$ then we have $d \in \cl_{M / f}(\{a, b, c, w\})$. But then $\sqcap_M(\{f, d\}, \{a, b, c, w\}) = 1$ so $(\{a, b, c, w, d, f\}, E(N_0))$ is a $2$-separation of $M \d e$, which means $\lambda_M(E(N_0)) \leq 2$, a contradiction. So $e \in B$. Then since $u$ is the unique element of $N_0$ in $\cl_{M / f}(B)$, $\cl_M(\{a, b, c, e, f\})$ contains no element of $N_0$ except possibly $u$. So $\lambda_{M \d w, d}(E(N_0)) \leq 1$, but since $w$ is in the closure of $\{a, b, c, u\}$, we have also $\lambda_{M \d d}(E(N_0)) \leq 1$, and then $\lambda_M(E(N_0)) \leq 2$, a contradiction.
This proves (\ref*{clm:notlemma2.8-notbothsmall}).

\begin{claim} \label{clm:notlemma2.8-actually8part1}
If $\Lambda^*$ is not a triad then for each $e \in \Lambda^*$, $E(M) \setminus (E(N_0) \cup \{e\}) \subseteq \Lambda$.
\end{claim}

We assume that $\Lambda^*$ is not a triad and that there exists $e \in \Lambda^*$. Let $f \in E(M) \setminus (E(N_0) \cup \{e\})$. By (\ref*{clm:notlemma2.8-8}), $M / e, f$ is not internally $3$-connected. Thus by Bixby's Lemma, $M / e \d f$ is internally $3$-connected. But $M$ has no triads by (\ref*{clm:notlemma2.8-6}) so $M / e \d f$ is actually $3$-connected. Also, as $M$ has no triads, the dual of (\ref*{clm:notlemma2.8-3}) implies that $M \d f$ is $3$-connected. Thus $f \in \Lambda$, proving (\ref*{clm:notlemma2.8-actually8part1}).

\begin{claim} \label{clm:notlemma2.8-actually8part2}
If $\Lambda$ is not a triangle then for each $e \in \Lambda$ that is not in a triangle of $M$ meeting $E(N_0)$, each $f \in E(M) \setminus (E(N_0) \cup \{e\})$ that is not in a triangle meeting $E(N_0)$ is in $\Lambda^*$.
\end{claim}

We assume that $\Lambda$ is not a triangle and that there exists $e \in \Lambda$ that is not in a triangle of $M$ meeting $E(N_0)$. Let $f \in E(M) \setminus (E(N_0) \cup \{e\})$ such that $f$ is not in a triangle meeting $E(N_0)$. By (\ref*{clm:notlemma2.8-7}), $M \d e, f$ is not internally $3$-connected. Thus by Bixby's Lemma, $M \d e / f$ is internally $3$-connected. But $M$ has no triangles disjoint from $E(N_0)$ by (\ref*{clm:notlemma2.8-6}) so $M \d e / f$ is actually $3$-connected. Also, as $M$ has no triangles disjoint from $E(N_0)$, (\ref*{clm:notlemma2.8-3}) implies that $M / f$ is $3$-connected. Thus $f \in \Lambda^*$, proving (\ref*{clm:notlemma2.8-actually8part2}).

\begin{claim} \label{clm:notlemma2.8-nottriangleandnottriad}
$\Lambda$ is not a triangle and $\Lambda^*$ is not a triad.
\end{claim}

Since $\lambda_M(E(N_0)) = 3$, $|E(M) \setminus E(N_0)| \geq 4$. If this is an equality then $E(M) \setminus E(N_0)$ is a cocircuit, contradicting (\ref*{clm:notlemma2.8-9}), so $|E(M) \setminus E(N_0)| \geq 5$.

It follows from (\ref*{clm:notlemma2.8-4}) and (\ref*{clm:notlemma2.8-10}) that $E(M) \setminus (E(N_0) \cup \Lambda \cup \Lambda^*)$ is contained in a triangle meeting $E(N_0)$. By Tutte's Triangle Lemma, such a triangle contains an element of $\Lambda$, so $|E(M) \setminus (E(N_0) \cup \Lambda \cup \Lambda^*)| \leq 1$. Hence $|\Lambda \cup \Lambda^*| \geq 4$.

Assume that $\Lambda$ is a triangle. Then by (\ref*{clm:notlemma2.8-notbothsmall}), $\Lambda^*$ is not a triad. We have $|\Lambda| = 3$, so $\Lambda^*$ is not empty. Thus by (\ref*{clm:notlemma2.8-actually8part1}), for any $e \in \Lambda^*$, $E(M) \setminus (E(N_0) \cup \{e\})$, which has size at least four, is contained in the triangle $\Lambda$. This proves that $\Lambda$ is not a triangle.

Next, assume that $\Lambda^*$ is a triad. We have $|\Lambda^*| = 3$, so $\Lambda$ is not empty. Suppose that there exist two distinct elements $e, f \in E(M) \setminus (E(N_0) \cup \Lambda^*)$ that are not in a triangle meeting $E(N_0)$. Then $e \in \Lambda$ and by (\ref*{clm:notlemma2.8-actually8part2}), $f \in \Lambda^*$, a contradiction. So there is at most one element of $E(M) \setminus (E(N_0) \cup \Lambda^*)$ that is not in a triangle meeting $E(N_0)$. Hence such a triangle, $T$, exists.
Denote the elements of $\Lambda^*$ by $\{x, y, z\}$.
Suppose that $E(M) \setminus E(N_0) = \Lambda^* \cup T$. Then $r(M) = 5$, and $r(M / x, y) = 3$ and $\si(M / x, y)$ is $3$-connected. So either $\{x, y\}$ is a contraction pair, or $M / x, y$ has a parallel class of size at least three, in which case $M / x$ has a line $L$ with at least four points and $y \in L$. If $L$ contains $T \setminus E(N_0)$, then $\sqcap_M(\{x, y\}, T) = 1$, but then $(E(N_0), (T \setminus E(N_0)) \cup \{x, y\})$ is a $2$-separation of $M \d z$, contradicting the fact that $\lambda_{M \d z}(E(N_0)) \geq \lambda_M(E(N_0)) - 1 = 2$. Otherwise, $L$ contains $\{y, z\}$ and exactly one element $a \in T$. But then $r_{M/x}(T \cup \{y, z\}) = 3$ so $r_M(T \cup \Lambda^*) = 4$; but as $r^*_M(T \cup \Lambda^*) < 4$, this means $\lambda_M(E(N_0)) < 3$, a contradiction.
Therefore, we may assume that $E(M) \setminus E(N_0)$ contains an element, $w$, that is not in $\Lambda^* \cup T$.
If $r(M \d \Lambda^*) = 4$, then $\{w\} \cup (T \setminus E(N_0))$ is a triad of $M \d \Lambda^*$, and the modularity of $N_0$ implies that $w$ is in two triangles, but we know that $T$ is the unique triangle not contained in $E(N_0)$. So $r(M \d \Lambda^*) = 5$ and hence $r(M) = 6$.
As $w \not\in \Lambda^*$, $M$ has a $3$-separation $(A, B)$ with $w \in \cl_M(A) \cap \cl_M(B)$ and $E(N_0) \subseteq \cl_M(A)$. The set $B \setminus \cl_M(A)$ is non-empty because $w$ is not in a triangle. Then $M$ being $3$-connected, we must have $|B \setminus \cl_M(A)| \geq 3$. If $B \setminus \cl_M(A)$ contains just one element of $\Lambda^*$, then it is a triad, but the unique triad of $M$ is $\Lambda^*$. So $(A \setminus \Lambda^*, B \cup \Lambda^*)$ is also a $3$-separation of $M$ and we may assume that $\Lambda^* \subseteq B$. If any one element of $T \setminus E(N_0)$ is in $B \setminus \cl_M(A)$, then both are; but then $\cl_M(A) = E(N_0)$ and $\lambda_M(E(N_0)) = 2$, a contradiction. So $T \subseteq \cl_M(A)$. Then $w \in \cl_M(A \setminus \cl_M(B))$ because $(A, B)$ is a $3$-separation. This means that $r(M \d \Lambda^*) = 4$ and $r(M) = 5$, a contradiction.
This proves (\ref*{clm:notlemma2.8-nottriangleandnottriad}).

\begin{claim} \label{clm:notlemma2.8-lastclaim}
 $\Lambda = E(M) \setminus E(N_0)$ and $\Lambda^*$ contains all elements of $E(M) \setminus E(N_0)$ except possibly two, which are in a common triangle.
\end{claim}

Recall that $|E(M) \setminus E(N_0)| \geq 5$ and $M$ has at most one triangle not contained in $E(N_0)$, so the set $X$ of elements of $E(M) \setminus E(N_0)$ not in a triangle meeting $E(N_0)$ has size at least three.
Thus either $|\Lambda^*| \geq 2$ or $|X \cap \Lambda| \geq 2$. If $|\Lambda^*| \geq 2$ then it follows from (\ref*{clm:notlemma2.8-actually8part1}) that $\Lambda = E(M) \setminus E(N_0)$. So in either case $|X \cap \Lambda| \geq 2$. Therefore, by (\ref*{clm:notlemma2.8-actually8part2}), $X \subseteq \Lambda^*$. Hence by (\ref*{clm:notlemma2.8-actually8part1}) again we have $\Lambda = E(M) \setminus E(N_0)$, proving (\ref*{clm:notlemma2.8-lastclaim}).

\begin{claim} \label{clm:notlemma2.8-essential}
 For any distinct $e, f \in \Lambda^*$ and $g \in \Lambda$, $M \d e / f$ is $3$-connected and $g$ is essential in $M \d e / f$.
\end{claim}

We have $e, f \in \Lambda$, so as $\{e, f\}$ is not a deletion pair, $M \d e, f$ is not internally $3$-connected and Bixby's Lemma says that $M \d e / f$ is internally $3$-connected. Since $f$ is not in a triangle of $M$, $M \d e / f$ is $3$-connected.
Suppose that $M \d e, g / f$ is $3$-connected.
Then $M \d e, g$ is internally $3$-connected, and $g \in \Lambda$ so $\{e, g\}$ is a deletion pair.
Next, suppose that $M \d e / f, g$ is $3$-connected.
Then $M / f, g$ is internally $3$-connected. Also, $g$ is not in a triangle of $M$, for if it were then $M \d e / f, g$ would have a parallel pair since $e$ is not in a triangle. Therefore, $g \in \Lambda^*$ and $\{f, g\}$ is a contraction pair.
Thus neither $M \d e, g / f$ nor $M \d e / f, g$ is $3$-connected, meaning that $g$ is essential in $M \d e / f$.
This proves (\ref*{clm:notlemma2.8-essential}).
\\

Let $e, f \in \Lambda^*$. By (\ref*{clm:notlemma2.8-essential}), every $g \in E(M) \setminus (E(N_0) \cup \{e, f\})$ is essential in $M \d e / f$.
However, Bixby's Lemma implies that one of $M \d e, g / f$ and $M \d e / f, g$ is internally $3$-connected, which means that $g$ is in a triangle or a triad of $M \d e / f$. But if $g$ is in a triangle then by Tutte's Triangle Lemma, it is also in a triad, and if $g$ is in a triad, then by the dual of Tutte's Triangle Lemma, it is also in a triangle. A circuit and a cocircuit of a matroid cannot meet in exactly one element, and a $3$-connected matroid has no triangle that is also a triad. Hence $g$ is contained in a four-element fan of $M \d e / f$.

Let $S = (s_1, \ldots, s_n)$ be a maximal fan of $M \d e / f$ containing $g$. By \autoref{lem:oxleywulemma}, the set of non-essential elements of $S$ is $\{s_1, s_n\}$. Note that all elements in $E(N_0)$ are non-essential, because deleting a point from a projective plane leaves a $3$-connected matroid. By (\ref*{clm:notlemma2.8-essential}), every element not in $E(N_0)$ is essential.
Hence $S \cap E(N_0) = \{s_1, s_n\}$. Moreover, no element of $E(N_0)$ is in a triad, so $\{s_1, s_2, s_3\}$ and $\{s_{n-2}, s_{n-1}, s_n\}$ are triangles and $\{s_2, s_3, s_4\}$ is a triad.

Suppose that $S$ is not the unique maximal fan of $M$ containing $g$. Then we apply \autoref{thm:oxleywutheorem} to $g$ and conclude that $S$ has five elements and there is an element $s \not\in S$ such that if $X = \{s, s_1, \ldots, s_5\}$ then $X$ contains another $5$-element fan and one of $M | X$ and $M / (E(M) \setminus X)$ is isomorphic to $M(K_4)$. But if we contract all but two or three elements of a projective plane, the remaining two or three elements are loops and hence cannot be part of an $M(K_4)$. So we have $M | X \cong M(K_4)$. Then either $s \in \cl_{M \d e / f}(\{s_2, s_5\}) \cap \cl_{M \d e / f}(\{s_1, s_4\})$ or $s \in \cl_{N_0}(\{s_1, s_5\}) \cap \cl_{M \d e / f}(\{s_2, s_3\}$. But in the former case, the triad $\{s_2, s_3, s_4\}$ meets the triangle $\{s_2, s, s_5\}$ in only one element, a contradiction. So $s \in E(N_0)$ and all maximal fans containing $g$ consist of $S \setminus E(N_0)$ and some two elements of $E(N_0)$.

Hence we can partition the elements of $E(M \d e / f) \setminus E(N_0)$ into sets $Y^1, \ldots, Y^k$ such that for each $Y^j$ there is an ordering $(y^j_2, \ldots, y^j_{n-1})$ of $Y^j$ and elements $y^j_1, y^j_n \in E(N_0)$ so that $(y^j_1, y^j_2, \ldots, y^j_n)$ is a fan of $M \d e / f$, and every maximal fan not contained in $E(N_0)$ consists of the union of some $Y^j$ with two elements of $E(N_0)$.
Note that every element of $E(M \d e / f) \setminus E(N_0)$ is in at most two triangles, and each triangle not contained in $E(N_0)$ has exactly one element in no other triangle.
Moreover, this is true for any choice of $e, f \in \Lambda^*$.

\begin{claim} \label{clm:notlemma2.8-nosmallfan}
 $|Y^j| > 3$ for all $j = 1, \ldots, k$.
\end{claim}

Suppose $|Y^j| = 3$ for some $j$. Then $Y^j$ is a triad of $M \d e / f$, and at least one element $y \in Y^j$ is in $\Lambda^*$. In $M \d y / f$, at least one of the two elements of $Y^j \setminus \{y\}$ is in a triangle with $e$, otherwise they are contained in a triangle that meets no other triangle outside $E(N_0)$. Let $z$ be the element in a triangle with $e$. If $k > 1$ then there is another part $Y^{j'}$. We choose any element $e' \in Y^{j'} \cap \Lambda^*$. Then $z$ is in three triangles of $M \d e' / f$, a contradiction. Hence $k = 1$, and so $j = 1$ and $E(M \d e / f) = Y^1$. But then $Y^j \cup \{e\}$ is a four-element cocircuit of $M / f$, and $M / f, e$ is internally $3$-connected, a contradiction. This proves (\ref*{clm:notlemma2.8-nosmallfan}).

\begin{claim} \label{clm:notlemma2.8-inclosureoffan}
 For all $j = 1, \ldots, k$, $e \in \cl_{M / f}(Y^j)$.
\end{claim}

Let $j \in \{1, \ldots, k\}$. By (\ref*{clm:notlemma2.8-nosmallfan}), $|Y^j| > 3$. Hence $Y_j$ contains a triangle $T = \{y^j_{i-1}, y^j_i, y^j_{i+1}\}$ of $M \d e / f$ disjoint from $E(N_0)$. We consider $M \d y^j_i / f$. Now the triangles $\{y^j_{i-3}, y^j_{i-2}, y^j_{i-1}\}$ and $\{y^j_{i+1}, y^j_{i+2}, y^j_{i+3}\}$ each have two elements in no other triangle. So one of $y^j_{i-2}, y^j_{i-1}$ is in a triangle with $e$, and one of $y^j_{i+1}, y^j_{i+2}$ is in a triangle with $e$. Let $T_1$ and $T_2$ be these triangles. Note that they are disjoint from $Y^m$ for all $m \neq j$: this is because every element of $Y^m$ is either in two other triangles in that fan, or is in one triangle of that fan in which it is the unique element not in any other triangle. Suppose $T_1$ and $T_2$ meet $E(N_0)$. Then whichever of $y^j_{i-1}$ and $y^j_{i-2}$ is in $T_1$ is in three triangles of $M \d y^j_{i+2} / f$, a contradiction if $y^j_{i+2} \in \Lambda^*$. But if $y^j_{i+2} \not\in \Lambda^*$, then $y^j_{i-2} \in \Lambda^*
$,
 and then whichever of $y^j_{
i+1}$ and $y^j_{i+2}$ is in $T_2$ is in three triangles of $M \d y^j_{i-2} / f$, a contradiction. This proves that one of $T_1$ or $T_2$ is disjoint from $E(N_0)$, hence contained in $Y^j \cup \{e\}$. Thus $e \in \cl_{M / f}(Y^j)$, proving (\ref*{clm:notlemma2.8-inclosureoffan}).
\\

If $k = 1$, then (\ref*{clm:notlemma2.8-inclosureoffan}) implies that $(Y^1 \cup \{e\}, E(N_0))$ is a $3$-separation of $M / f$, and so $\lambda_{M / f}(E(N_0)) = 2$. But since $f \not\in \cl_M(E(N_0))$, we have $3 = \lambda_M(E(N_0)) = \lambda_{M / f}(E(N_0)) = 2$, a contradiction. Hence $k \geq 2$. But note that $\sqcap_{M / f}(Y^1, E(N_0)) = 2$ and $\sqcap_{M / f}(Y^2, E(N_0)) = 2$. So the fact that $e \not\in \cl_{M / f}(E(N_0))$ but $e \in \cl_{M / f}(Y^1) \cap \cl_{M / f}(Y^2)$ means that $\lambda_{M / f}(Y^1) \geq 3$. But since $e \in \cl_{M / f}(Y^2)$, we then have $\lambda_{M \d e / f}(Y^1) = 3$. This contradicts \autoref{lem:fan3separates}, which said that a fan always has connectivity at most two.
\end{proof}

We can now prove the main result of this section. It asserts that if there exists a counterexample to \autoref{lem:keylemma}, then there exists one with a deletion pair disjoint from $E(N_0)$, and moreover that we can choose it such that deleting this pair results in a matroid with at most one series pair.

\begin{lemma} \label{lem:getcounterexamplewithdeletionpair}
Let $M_0$ be a matroid that is $3$-connected, non-$\F$-representable and has a modular $\PG(2, \F)$-restriction, $N_0$, with $\lambda_{M_0}(E(N_0)) = 3$, such that no proper minor of $M_0$ that has $N_0$ as a restriction is $3$-connected and non-$\F$-representable. 
Then there exists such a matroid $M_0$ with a deletion pair $x, y \in E(M_0) \setminus E(N_0)$ such that $M_0 \d x, y$ has at most one series pair.
\end{lemma}

\begin{proof}
We use the matroid $R$ defined in \autoref{sec:duality}.
First we show that a contraction pair in one of $M_0$ and $M_1 = ((R \oplus_m M_0) \d E(N_0))^*$ is a deletion pair in the other.

\begin{claim} \label{clm:getcounterexamplewithdeletionpair-transformcontractiontodeletion}
Let $x,y \in E(M_0) \setminus E(N_0)$.
If $\{x, y\}$ is a contraction pair in $M_0$, then $\{x, y\}$ is a deletion pair in $M_1$, and any series pair in $M_1 \d x, y$ is a parallel pair in $M_0 / x, y$.
If $\{x, y\}$ is a contraction pair in $M_1$, then $\{x, y\}$ is a deletion pair in $M_0$, and any series pair in $M_0 \d x, y$ is a parallel pair in $M_1 / x, y$.
\end{claim}

We prove the first statement, and the second follows by \autoref{lem:involution}.
By \autoref{prop:dualledmatroid}, $\si(M_1)$ is a $3$-connected, non-$\F$-representable matroid with a modular restriction $N_1 \cong \PG(2,\F)$ and $\lambda_{M_1}(E(N_1)) = 3$.
Let $X$ be the set of elements of $M_1$ parallel to an element of $E(N_1)$.
By \autoref{lem:involution}, $M_0 = ((R^* \oplus_m M_1) \d E(N_1))^*$, so $M_0 / X = ((R^* \oplus_m \si(M_1)) \d E(N_1))^*$. By \autoref{prop:dualledmatroid} applied to $M_1$, $\si(M_0 / X)$ is a $3$-connected, non-$\F$-representable matroid with $N_0$ as a modular restriction. Thus the minimality of $M_0$ implies that $\si(M_0 / X) = M_0$, and so $X = \emptyset$ and $M_1$ is $3$-connected.

We have $M_1 \d x = ((R \oplus_m (M_0 / x)) \d E(N_0))^*$ so by \autoref{lem:preservesstability} applied to $M_0 / x$, $\si(M_1 \d x)$ is $3$-connected; but $M_1$ is $3$-connected so this implies that $M_1 \d x$ is $3$-connected. Similarly, $M_1 \d y$ is $3$-connected.

The modularity of $N_0$ in the $3$-connected matroid $M_0 / x$ implies that $y \not\in \cl_{M_0 / x}(E(N_0))$ so $M_0 / x,y$ has $N_0$ as a restriction.
Let $S_1$ be a set consisting of one element from each parallel pair of $M_0 / x,y$ that is disjoint from $E(N_0)$, and let $S_2$ be the set of elements of $M_0 /x,y$ parallel to an element of $E(N_0)$.
Then by \autoref{lem:preservesstability}, $(R \oplus_m (M_0 / x,y \d S_1,S_2)) \d E(N_0)$ is internally $3$-connected with no parallel pairs.
But recall that $M_1^* = (R \oplus_m M_0) \d E(N_0)$ is $3$-connected, so $(R \oplus_m (M_0 / x,y \d S_1,S_2)) \d E(N_0)$ has no series pairs either and is thus $3$-connected.
Therefore, $(R \oplus_m (M_0 / x,y \d S_1)) \d E(N_0)$ is also $3$-connected as it is obtained by adding non-parallel elements without increasing the rank.
So $M_1 \d x,y / S_1$ is $3$-connected.
Each parallel pair of $M_0 / x,y$ disjoint from $E(N_0)$ is a series pair of $M_1 \d x, y$, and contracting an element from each of these pairs results in the $3$-connected matroid $M_1 \d x, y / S_1$, so $M_1 \d x, y$ is internally $3$-connected and $\{x, y\}$ is a deletion pair of $M_1$.

Any series pair of $M_1 \d x, y$ is a parallel pair of $(R \oplus_m M_0) \d E(N_0) / x, y$ and hence a parallel pair of $M_0 / x, y$.
This proves (\ref*{clm:getcounterexamplewithdeletionpair-transformcontractiontodeletion}).
\\

We see that $M_0 / E(N_0)$ is connected, for if not then by \autoref{prop:separatorismodularsum} $M_0$ is a modular sum of two proper restrictions of $M_0$ containing $N_0$. Both of these are $3$-connected by \autoref{prop:connectivitypropertiesofmodularsum} and at least one of them is not $\F$-representable by \autoref{prop:modularsumpreservesrepresentability} and the Fundamental Theorem of Projective Geometry, contradicting the minimality of $M_0$.
Moreover, $E(M_0) \setminus E(N_0)$ is non-empty as $M_0$ is not $\F$-representable.

For the following claim, we use the fact that any binary matroid is uniquely representable over any field over which it is representable \cite{BrylawskiLucas}.

\begin{claim} \label{clm:getcounterexamplewithdeletionpair-fan}
No element of $M_0$ is in both a triangle and a triad.
\end{claim}

Suppose $f \in E(M_0)$ is in both a triangle and a triad. Since a circuit and a cocircuit cannot intersect in a single element and a triangle in a $3$-connected matroid cannot be a triad, $M_0$ has elements $\{e, f, g, h\}$ such that $\{e, f, g\}$ is a triad and $\{f, g, h\}$ is a triangle.
Since $\{e, f, g\}$ is not a triangle, $e$ is not in a four-point line of $M_0$.
Suppose that $M_0 / e$ has an internal $2$-separation, $(A, B)$. Then $(A \cup \{e\}, B)$ is a $3$-separation of $M_0$ and $e \in \cl_{M_0}(A) \cap \cl_{M_0}(B)$. 
Since $|A|, |B| \geq 3$ but $e$ is not in a four-point line in $M_0$, neither $A$ nor $B$ is contained in a parallel class of $M_0 / e$. So $(A, B)$ is a vertical $2$-separation of $M_0 / e$ and hence $(A \cup \{e\}, B)$ is a vertical $3$-separation of $M_0$. Therefore, $e \in \cl_{M_0}(A \setminus \cl_{M_0}(B))$ and $e \in \cl_{M_0}(B \setminus \cl_{M_0}(A))$. Then as $\{e, f, g\}$ is a cocircuit, one of $f$ and $g$ is contained in $A \setminus \cl_{M_0}(B)$ and the other in $B \setminus \cl_{M_0}(A)$. But then $\{f, g\}$ cannot be contained in a triangle. So $M_0 / e$ is internally $3$-connected.

Let $N$ be a copy of $M(K_4)$ such that $E(N) \cap E(M_0) = \{f, g, h\}$ and $\{f, g, h\}$ is a triangle of $N$.
Then $M_0$ is isomorphic to $(N \oplus_m (M_0 / e)) \d f, g$. Since $N$, being graphic, is $\F$-representable and $U_{2, 3}$ is uniquely $\F$-representable, $M_0 / e$ is not $\F$-representable by \autoref{prop:modularsumpreservesrepresentability}.

So $M_0$ has a proper minor $\si(M_0 / e)$ that is $3$-connected and non-$\F$-representable and has $N_0$ as a restriction, contradicting our minimal choice of $M_0$. This proves (\ref*{clm:getcounterexamplewithdeletionpair-fan}).
\\

Since $M_0 / E(N_0)$ is connected and non-empty and no element of $M_0$ is in both a triangle and a triad, all assumptions of \autoref{lem:notlemma2.8} hold for $M_0$.
Suppose that $M_0$ has a restriction $K \cong M(K_5)$ with a cocircuit $\{a, b, c, d\}$ such that $M_0 = K \oplus_m (M_0 \d a, b, c, d)$. Then $K \d a, b, c, d \cong M(K_4)$, which is binary hence uniquely $\F$-representable. Therefore, $M_0 \d a, b, c, d$ is non-$\F$-representable by \autoref{prop:modularsumpreservesrepresentability}. Also, it is $3$-connected by \autoref{prop:connectivitypropertiesofmodularsum}, contradicting our minimal choice of $M_0$. This means that outcome (\ref*{out:notlemma2.8-modularsum}) of \autoref{lem:notlemma2.8} does not hold for $M_0$ and so $M_0$ has either a deletion or contraction pair $x, y \in E(M_0) \setminus E(N_0)$.

By (\ref*{clm:getcounterexamplewithdeletionpair-transformcontractiontodeletion}), $\{x, y\}$ is a deletion pair in a matroid $M \in \{M_0, M_1\}$. We choose such a pair $\{x, y\}$ so that the number of series pairs of $M \d x, y$ is minimum. We may therefore assume that for any contraction pair $\{u, v\}$ of $M$ disjoint from $E(N_0)$ and $E(N_1)$, $M / u, v$ has at least as many parallel pairs as $M \d x, y$ has series pairs.
We now show that $M \d x, y$ has at most one series pair.

We denote the series pairs of $M \d x, y$ by $S_1 = \{a_1, b_1\}, \ldots, S_k = \{a_k, b_k\}$.

\begin{claim}\label{clm:getcounterexamplewithdeletionpair-oneseriespair:2}
Let $c \in S_1 \cup \cdots \cup S_k$ such that $c$ is not in a triangle of $M$.
If $\{x, y\} \not\subseteq \cl_M(\{c\} \cup S_j)$ for all $j$, then $M / c$ is $3$-connected.
\end{claim}

By symmetry we may assume that $c \in S_1$.
It suffices to show that $M/c$ is internally $3$-connected.
Suppose that $M/c$ has an internal $2$-separation $(A,B)$ with $y \in B$. If $x \in A$, then as $M/c\d x,y$ is internally $3$-connected, either $A$ or $B$ is a triad of $M/c$ hence also of $M$, contradicting the fact that $M\d x$ and $M\d y$ are $3$-connected. So $x,y \in B$.
Since $\{c,x,y\}$ is not a triangle, $|B| > 2$, and the fact that $M \d x$ is $3$-connected implies that $B$ is not a triad, so $|B| > 3$.
Then since $M / c \d x,y$ is internally $3$-connected, $B \setminus \{x,y\}$ is a series pair of it, so $B = S_j \cup \{x,y\}$ for some $j>1$.
Since $M/c$ is connected and $\lambda_{M/c}(S_j \cup \{x,y\}) = 1$, either $r_{M/c}(\{x,y\} \cup S_j) = 2$ or $r^*_{M/c}(\{x,y\} \cup S_j) = 2$; but $r^*_{M / c \d x,y}(S_j) = 1$ so we have $r_{M/c}(\{x,y\} \cup S_j) = 2$. But this means $\{x,y\} \subseteq \cl_{M/c}(S_j)$ so $\{x,y\} \subseteq \cl_M(\{c\} \cup S_j)$, a
contradiction. This proves (\ref*{clm:getcounterexamplewithdeletionpair-oneseriespair:2}).

\begin{claim} \label{clm:getcounterexamplewithdeletionpair-oneseriespair:0}
$M \d x,y$ does not have exactly two series pairs.
\end{claim}

Suppose that $M \d x,y$ has exactly two series pairs.
First, we assume that $\{x,y\} \not\subseteq \cl_M(S_1 \cup S_2)$; by symmetry we may assume that $y \not\in \cl_M(S_1 \cup S_2)$.
Then there is at most one triangle containing an element of $S_1 \cup S_2$; either it contains $\{x,y\}$ or it contains $x$ and an element of each of $S_1$ and $S_2$. So we may assume that $a_1$ and $a_2$ are not contained in any triangles, and if $b_2$ is contained in a triangle then it is $\{x, b_1, b_2\}$.
Then by (\ref*{clm:getcounterexamplewithdeletionpair-oneseriespair:2}), $M/a_1$ and $M/a_2$ are $3$-connected.
Since $M \d x,y / a_1,a_2$ is $3$-connected, $\si(M / a_1,a_2)$ is $3$-connected and all parallel pairs of $M/a_1,a_2$ contain $x$ or $y$. So if $M / a_1,a_2$ has a parallel class of size greater than two, then it contains $x$ and $y$, and so $\{a_2,x,y\}$ is a triangle of $M/a_1$. But then $\{x,b_1,b_2\}$ is not a triangle of $M$, so by (\ref*{clm:getcounterexamplewithdeletionpair-oneseriespair:2}), $M/b_2$ is also $3$-connected, and $M/a_1,b_2$ has no parallel classes of size greater than two.
So by possibly swapping the labels of $a_2$ and $b_2$, we may assume that $M / a_1,a_2$ is internally $3$-connected.
By our choice of deletion pair $x,y$, there are exactly two parallel pairs in $M/ a_1,a_2$.
We call them $\{x,u\}$ and $\{y,v\}$.

We consider $M \d x,a_2$, which has $\{b_2,y\}$ as a series pair.
In $M \d x / y$, $S_1 \cup S_2$ is a $4$-element cocircuit. Also, $\lambda_{M \d x / y}(S_1 \cup S_2) = 3$ because $y \not\in \cl_M(S_1 \cup S_2)$. Moreover, $v \in \cl_{M \d x / y}(\{a_1,a_2\})$ but $v \not\in \cl_{M \d x / y}(\{a_1,b_1,b_2\})$, so $\lambda_{M \d x,a_2 / y}(\{a_1,b_1,b_2,v\}) = 3$ and $M \d x, a_2 / y$ is $3$-connected.
Thus $M \d x,a_2$ is internally $3$-connected with a unique series pair, $\{b_2,y\}$.
Then $M \d a_2$ is $3$-connected because $x \not\in \cl_M(E(M) \setminus (S_2 \cup \{x,y\}))$ and $x \not\in \cl_M(\{b_2,y\})$. This contradicts our choice of deletion pair $\{x, y\}$.
Therefore, we may assume that $\{x,y\} \subseteq \cl_M(S_1 \cup S_2)$.

Suppose $\{x,y\}$ is contained in a triangle; we may assume that either $\{b_1,x,y\}$ is the only such triangle or $\{b_1,x,y\}$ and $\{b_2,x,y\}$ are the only
two.
Then (\ref*{clm:getcounterexamplewithdeletionpair-oneseriespair:2}) implies that $M / a_2$ is $3$-connected. Also, $\{b_1,b_2,x,y\}$ is a $4$-point line of $M / a_2$, so $M / a_2 \d b_1$ and $M /
a_2 \d b_2$ are $3$-connected. Then $M \d b_1$ and $M \d b_2$ are internally $3$-connected and any series pair of each contains $a_2$. But neither $\{a_2,b_1\}$ nor $\{a_2,b_2\}$ is contained in a triad of $M$, so $M \d b_1$ and $M \d b_2$ are $3$-connected. In $M \d b_1,b_2$, $\{x,y\}$ is a series pair.
In $M \d b_1,b_2 / y$, $\{a_1,a_2,x\}$ is a triad which is not a triangle, so $\lambda_{M \d b_1,b_2/y}(\{a_1,a_2,x\}) = 2$ and $M \d b_1,b_2/y$ is $3$-connected. Therefore, $M\d b_1,b_2$ is internally $3$-connected with a single series pair, $\{x,y\}$, contradicting our choice of deletion pair $\{x, y\}$.
So $\{x,y\}$ is not contained in a triangle.

At least one of $\{a_1,x\}$ and $\{b_1,x\}$ is skew to $S_2$, for if not then since $\{a_1,b_1,x\}$ is not a triangle we have $\sqcap_M(S_1 \cup \{x\}, S_2) =
2$, contradicting the fact that $r_M(S_1 \cup S_2) = 4$. By symmetry we may assume that $\{a_1,x\}$ is skew to $S_2$.
We claim that $M \d x,a_1$ is internally $3$-connected with one series pair, $\{b_1,y\}$. If not, then $M \d x,a_1/b_1$ is not $3$-connected. But $M \d x /
b_1$ is $3$-connected and has a triad $\{a_2,b_2,y\}$, so $\{a_2,b_2,y\}$ is also a triad of $M \d x, a_1 / b_1$. Then we may choose a $2$-separation $(A,B)$
of $M \d x,a_1/b_1$ with $a_2,b_2,y \in A$. But then $a_1 \in \cl_{M \d x / b_1}(\{a_2,b_2,y\}) \subseteq \cl_{M\d x / b_1}(A)$, a contradiction. So $M \d
x,a_1$ is internally $3$-connected with a unique series pair. Moreover, $M \d a_1$ is $3$-connected because $x \not\in \cl_M(\{b_1,y\})$ and $x \not\in
\cl_M(E(M) \setminus (S_1 \cup S_2 \cup \{y\}))$, contradicting our choice of deletion pair $\{x, y\}$.
This proves (\ref*{clm:getcounterexamplewithdeletionpair-oneseriespair:0}).

\begin{claim}\label{clm:getcounterexamplewithdeletionpair-oneseriespair:1}
There is at most one $c \in S_1 \cup \cdots \cup S_k$ such that $\{c, x, y\}$ is a triangle.
\end{claim}

If not, then we may assume that there are distinct triangles $\{c_1,x,y\}$ and $\{c_2,x,y\}$ with $c_1 \in S_1$ and $c_2 \in S_2$. Then $x,y \in \cl_M(S_1 \cup
S_2)$. But (\ref*{clm:getcounterexamplewithdeletionpair-oneseriespair:0}) implies that there is a third series pair $S_3$ of $M \d x,y$, and in this case $S_3$ is also a series pair of $M$, which is
$3$-connected. This proves (\ref*{clm:getcounterexamplewithdeletionpair-oneseriespair:1}).

\begin{claim}\label{clm:getcounterexamplewithdeletionpair-oneseriespair:3}
If $M \d x, y$ has more than one series pair, it has exactly three and $\{x, y\}$ is in the closure of their union.
\end{claim}

By (\ref*{clm:getcounterexamplewithdeletionpair-oneseriespair:0}), we may assume that $M \d x,y$ has at least three series pairs, and by (\ref*{clm:getcounterexamplewithdeletionpair-oneseriespair:1}) we may assume that for any $c \in S_2
\cup S_3$, $\{c,x,y\}$ is not a triangle.
Then any triangle of $M$ containing $c$ contains exactly one of $x$ and $y$, say $x$. But if such a triangle exists
then for some $r$, $S_r$ is skew to it and is thus a series pair of $M \d y$, a contradiction. So $c$ is in no triangles of $M$, and by (\ref*{clm:getcounterexamplewithdeletionpair-oneseriespair:2}),
$M / c$ is $3$-connected.
For any two such elements $c \in S_2, d \in S_3$, any parallel class of $M / c, d$ contains $x$ or $y$ since $M / c, d \d x, y$ has no parallel pairs.
If $M / c, d$ is internally $3$-connected, this implies that there are at most two parallel pairs in $M / c, d$, which contradicts our choice of $\{x, y\}$. Thus it suffices to show that $M / c, d$ is internally $3$-connected.

Assume that $M / c,d$ is not internally $3$-connected. Then $M/c,d$ has an internal $2$-separation $(A,B)$ with $y \in B$. If $x \in A$ then, as $M/c,d\d x,y$ is internally $3$-connected, $A$ or $B$ is a triad of $M/c,d$ hence also of $M$, contradicting the fact that $M\d x$ and $M\d y$ are $3$-connected. So $x,y \in B$, and since $M \d x,y / c,d$ is internally $3$-connected, $|B| \leq 4$.

Let $z \in B \setminus \{x,y\}$. 
If $B$ is a parallel class, then $\{c,d,z,x\}$ and $\{c,d,z,y\}$ are circuits of $M$. If $z \in S_1$, then (\ref*{clm:getcounterexamplewithdeletionpair-oneseriespair:3}) holds, while if not then
$\{x,y\}$ is in the closure of $E(M) \setminus S_1$ so $S_1$ is a series pair of $M$, a contradiction. Therefore, $r_{M/c,d}(B) \geq 2$.
Also, $r^*_{M/c,d}(B) = r^*_M(B) \geq 2$ since $M$ is $3$-connected.
It follows that if $|B| = 3$, then $B$ is a triad of $M/c,d$ and hence also of $M$, a contradiction. So $|B| = 4$, $r_{M/c,d}(B) = 2$, and $B \setminus \{x,y\}$ is a series pair $S_i$ of $M \d x,y$. Therefore, $x,y \in \cl_M(S_2 \cup S_3 \cup S_i)$ and (\ref*{clm:getcounterexamplewithdeletionpair-oneseriespair:3}) holds.

\begin{claim} \label{clm:getcounterexamplewithdeletionpair-oneseriespair:4}
$M \d x,y$ has at most one series pair.
\end{claim}

If not, then by (\ref*{clm:getcounterexamplewithdeletionpair-oneseriespair:3}) it has three.
By (\ref*{clm:getcounterexamplewithdeletionpair-oneseriespair:1}) and (\ref*{clm:getcounterexamplewithdeletionpair-oneseriespair:2}), we may assume that $M / a_1$, $M / b_1$ and $M / a_2$ are $3$-connected. $S_3$ is a series pair of
$M / a_1,a_2 \d x,y$ and of $M / b_1, a_2 \d x,y$, which each have a unique $2$-separation. This implies that $M / a_1,a_2 \d y$ is $3$-connected if and only if $x \not\in \cl_{M/a_1,a_2}(S_3)$, and $M / b_1, a_2 \d y$ is $3$-connected if and only if $x \not\in \cl_{M/b_1,a_2}(S_3)$. 
If at least one of $M / a_1,a_2 \d y$ and $M / b_1,a_2 \d y$ is $3$-connected, then one of $M / a_1,a_2$ and $M/ b_1,a_2$ is internally $3$-connected with at most one parallel pair, contradicting our choice of $\{x, y\}$. So we may assume that $x \in \cl_{M/a_1,a_2}(S_3)$ and $x \in \cl_{M/b_1,a_2}(S_3)$. 
But then we have $a_1 \in \cl_M(\{x,a_2\} \cup S_3)$ and $b_1 \in \cl_M(\{x,a_2\} \cup S_3)$, so $r_M(S_1 \cup S_3 \cup \{a_2\}) \leq 4$, a contradiction. This proves (\ref*{clm:getcounterexamplewithdeletionpair-oneseriespair:4}).
\\

%

If $M = M_0$, then we are done by (\ref*{clm:getcounterexamplewithdeletionpair-oneseriespair:4}).
If $M = M_1$, then 
it remains to show that $M_1$ has no $3$-connected, non-$\F$-representable proper minor with $N_1$ as a restriction. If it does have such a minor $M_1 \d D / C$, then by \autoref{lem:involution} we have $M_0 / D \d C = ((R^* \oplus_m (M_1 \d D / C)) \d E(N_1))^*$ and then \autoref{prop:dualledmatroid} implies that $\si(M_0 / D \d C)$ is a $3$-connected, non-$\F$-representable proper minor of $M_0$ with $N_0$ as a restriction, a contradiction.
\end{proof}

\section{Stabilizers}

Two representations of a matroid over a field $\mathbb{F}$ are called \defn{equivalent} if one can be obtained from the other by row operations (including adjoining and removing zero rows) and column scaling; they are \defn{inequivalent} otherwise.

When $A$ is a representation of a matroid $M$ in standard form with respect to a basis $B$, we index the rows of $A$ by the elements of $B$, so that $A_{bb} = 1$ for each $b \in B$. For any $X \subseteq B$ and $Y \subseteq E(M)$, we write $A[X, Y]$ for the submatrix of $A$ in the rows indexed by $X$ and the columns indexed by $Y$.

Whenever $N$ is a minor of a matroid $M$, we can partition $E(M) \setminus E(N)$ into an independent set $C$ and a coindependent set $D$ such that $N = M / C \d D$ (see \cite[Lemma 3.3.2]{Oxley}).
Suppose that $B$ is a basis of $N$ and that $B' = B \cup C$, so $B'$ is a basis of $M$.
Let $\F$ be a field and $A'$ an $\F$-representation of $M$ in standard form with respect to the basis $B'$.
Then the matrix $A = A'[B, E(N)]$ is an $\F$-representation of $N$ in standard form with respect to the basis $B$.
We say that the representation $A'$ of $M$ \defn{extends} the representation $A$ of $N$ and that $A$ \defn{extends to} $A'$. Conversely, $A$ is the representation of $N$ that is \defn{induced} by $A'$.
Any representation of $N$ that is row-equivalent to $A$ is also said to extend to $A'$.

When $N$ is a minor of a matroid $M$, we say that $N$ \defn{stabilizes} $M$ over $\F$ if no $\F$-representation of $N$ extends to two inequivalent $\F$-representations of $M$.
We will use the following fact about stabilizers.

\begin{theorem}[Geelen, Whittle, \cite{GeelenWhittle:projectiveplaneisastabilizer}] \label{thm:modularflatstabilizes}
For any finite field $\F$, if $M$ is a $3$-connected matroid with $\PG(2, \F)$ as a minor, then $\PG(2, \F)$ stabilizes $M$ over $\F$.
\end{theorem}

A matroid $M$ is called \defn{stable} if it is connected and is not a $2$-sum of two non-binary matroids (this definition differs slightly from the original in \cite{GeelenGerardsKapoor} in that we require that $M$ be connected).

We observe that \autoref{thm:modularflatstabilizes} implies that any stable matroid $M$ with $\PG(2, \F)$ as a minor is stabilized by it over $\F$; this follows from the fact that a binary matroid is uniquely representable over any field over which it is representable \cite{BrylawskiLucas}. In particular, if $M$ is a direct sum or a $2$-sum of a $3$-connected matroid $N$ and a binary matroid, then every $\F$-representation of
$N$ extends to a unique (up to equivalence) $\F$-representation of $M$.
For a field $\mathbb{F}$, we call a matroid a \defn{stabilizer for $\mathbb{F}$} if it stabilizes over $\mathbb{F}$ all stable matroids that have it as a minor.

In the next section we will apply the following two lemmas about stabilizers. They were proved by Geelen, Gerards and Whittle \cite{GeelenGerardsWhittle:onrotas} and can also be derived from results of Whittle \cite{Whittle}.

\begin{lemma}[Geelen, Gerards, Whittle, \cite{GeelenGerardsWhittle:onrotas}] \label{lem:uniquerepresentablematroid1}
Let $N$ be a uniquely $\F$-representable stabilizer for a finite field $\F$.
Let $M$ be a matroid with $x, y \in E(M)$ such that $\{x,y\}$ is coindependent and $M \d x,y$ is stable and has an $N$-minor.
If $M \d x$ and $M \d y$ are both $\F$-representable, then there exists an $\F$-representable matroid $M'$ such that $M' \d x = M \d x$ and $M' \d y = M \d y$.
\end{lemma}

We remark that although their statement of the above lemma \cite[Lemma 5.3]{GeelenGerardsWhittle:onrotas} requires that $M$ be $3$-connected, their proof requires only that $\{x, y\}$ be coindependent in $M$.

\begin{lemma}[Geelen, Gerards, Whittle, \cite{GeelenGerardsWhittle:onrotas}] \label{lem:uniquerepresentablematroid2}
Let $\F$ be a finite field and let $M$ and $M'$ be $\F$-representable matroids on the same ground set with elements $x,y \in E(M)$ such that $M \d x = M' \d x$ and $M \d y = M' \d y$. If $M \d x$ and $M \d y$ are both stable, $M \d x,y$ is connected, and $M \d x,y$ has a minor that is a uniquely $\F$-representable stabilizer for $\F$, then $M = M'$.
\end{lemma}

\section{Finding distinguishing sets} \label{sec:strands}

In this section, we show that if $M$ is a matroid with a restriction $N_0 \cong \PG(2, \F)$ and a deletion pair $x, y \not \in E(N_0)$ such that $M \d x$ and $M \d y$ are $\F$-representable, then there is a unique $\F$-representable matroid $M'$ on $E(M)$ whose rank function can differ from that of $M$ only on sets containing $x$ and $y$. Moreover, we find two such sets with special properties that will be used to prove \autoref{lem:keylemma}.

\begin{lemma}\label{lem:uniquegfqmatroid}
Let $M$ be a $3$-connected matroid with a restriction $N_0 \cong \PG(2, \F)$ and a deletion pair $x,y \in E(M) \setminus E(N_0)$. If $M \d x$ and $M \d y$ are $\F$-representable, then there is a unique $\F$-representable matroid $M'$ such that $M' \d x = M \d x$ and $M' \d y = M \d y$.
\end{lemma}

\begin{proof}
The definition of a deletion pair implies that $M\d x$, $M\d y$, and $M\d x,y$ are all stable.
We recall from \autoref{thm:modularflatstabilizes} that $\PG(2,\F)$ is a stabilizer for $\F$, and that $\PG(2, \F)$ is uniquely $\F$-representable.
Therefore, with $N = N_0$ all the hypotheses of \autoref{lem:uniquerepresentablematroid1} are satisfied, and there is an $\F$-representable matroid $M'$ such
that $M \d x = M' \d x$ and $M \d y = M' \d y$. Then by \autoref{lem:uniquerepresentablematroid2}, $M'$ is the unique such matroid.
\end{proof}

The purpose of the remainder of this section is to find two ways to distinguish the matroids $M$ and $M'$ of \autoref{lem:uniquegfqmatroid}; these are
\begin{enumerate}[(a)]
\item elements $e \in E(M)$ such that $M \d e \neq M' \d e$ and $M/e \neq M'/e$, and
\item sets $S \subseteq E(M) \setminus E(N_0)$ such that $\cl_M(S) \neq \cl_{M'}(S)$.
\end{enumerate}

In a matroid $M$ with a restriction $N$, we call a set $S \subseteq E(M) \setminus E(N)$ a \defn{strand} for $N$ if $\sqcap_M(S, E(N)) = 1$.
If $M$ and $M'$ are matroids on the same ground set, both contain a restriction $N$, and $S$ is a strand for $N$ in both $M$ and $M'$, then we say that $S$ \defn{distinguishes} $M$ and $M'$ if $\cl_M(S) \cap E(N) \neq \cl_{M'}(S) \cap E(N)$.

When $B$ is a basis of a matroid $M$ and $e$ is an element not in $B$, then the \defn{fundamental circuit} of $e$ with respect to $B$ is the unique circuit of $M$ contained in $B \cup \{e\}$.
The \defn{fundamental matrix} of a matroid $M$ with respect to a basis $B$ is the matrix $A \in \{0,1\}^{B \times (E(M) \setminus B)}$ such that
for each $e \in E(M) \setminus B$, the column of $A$ indexed by $e$ is the characteristic vector of the fundamental circuit of $e$ with respect to $B$.
If $A'$ is any representation of $M$ in standard form with respect to $B$, then the matrix obtained from $A' | (E(M) \setminus B)$ by replacing each non-zero entry with $1$ is the fundamental matrix of $M$ with respect to $B$.

For each $X \subseteq B$ and $Y \subseteq E(M) \setminus B$, we write $A[X, Y]$ for the submatrix of $A$ in the rows indexed by $X$ and the columns indexed by $Y$.

\begin{lemma} \label{lem:baddeterminant}
Let $M$ and $M'$ be matroids on the same ground set that both have a modular restriction $N_0 \cong \PG(2,\F)$, such that $M$ is $3$-connected and $M \neq M'$, but for some $x,y \in E(M) \setminus E(N_0)$, $M' \d x = M \d x$ and $M' \d y = M \d y$.
There are sets $B$ and $B'$ such that
\begin{enumerate}[(i)]
 \item $B$ is a basis of both $M$ and $M'$ and contains a basis of $N_0$,
 \item $B'$ is a basis of exactly one of $M$ and $M'$,
 \item $|(B \setminus B') \setminus E(N_0)| \in \{1,2\}$, and \label{out:baddeterminant-third}
 \item $|B \Delta B'| = 4$. \label{out:baddeterminant-fourth}
\end{enumerate}
\end{lemma}

\begin{proof}
Since $M \neq M'$, there exists a set $B'$ that is a basis of exactly one of $M$ and $M'$. We choose $B'$ and a basis $B$ of $M \d x,y$ containing a basis
of $N_0$ such that $|B \Delta B'|$ is minimum. As $M$ is $3$-connected, $B$ is a basis of $M$ and also of $M'$.

\begin{claim} \label{clm:baddeterminant-1}
$B' \setminus B \subseteq E(M) \setminus E(N_0)$.
\end{claim}

Suppose there exists an element $u$ of $(B' \setminus B) \cap E(N_0)$. Then by the basis exchange property, there is $v \in B \setminus B'$ such that $B \Delta
\{u,v\}$ is a basis in $M$ or $M'$. But since $B \Delta \{u,v\}$ is contained in $E(M) \setminus \{x,y\}$, it is a basis of both $M$ and $M'$. Because
$u \in E(N_0)$, $B \Delta \{u,v\}$ contains a basis of $N_0$, contradicting the minimality of $|B \Delta B'|$.
This proves (\ref*{clm:baddeterminant-1}).
\\

It follows from the fact that $M \d x = M' \d x$ and $M \d y = M' \d y$ that the fundamental matrices of $M$ and $M'$ with respect to the basis $B$ are equal; we denote this matrix by $A$.

\begin{claim}\label{clm:baddeterminant-2}
$A[(B \setminus B') \setminus E(N_0), (B' \setminus B) \setminus \{x,y\}] = 0$.
\end{claim}

Suppose there exist elements $u \in B' \setminus B \setminus \{x,y\}$ and $v \in (B \setminus B') \setminus E(N_0)$ such that $B \Delta \{u,v\}$ is a basis of
$M$ or $M'$. Then since $B \Delta \{u,v\} \subseteq E(M) \setminus \{x,y\}$, it is a basis of both $M$ and $M'$. Furthermore, $B \Delta \{u,v\}$ contains a
basis of $N_0$ because $v \not\in E(N_0)$, contradicting the minimality of $|B \Delta B'|$. Therefore, for every $u \in (B' \setminus B) \setminus \{x, y\}$ and every $v \in (B \setminus B') \setminus E(N_0)$, the set $B \Delta \{u, v\}$ is dependent in both $M$ and $M'$.
This proves (\ref*{clm:baddeterminant-2}).

\begin{claim}\label{clm:baddeterminant-3}
$|(B \setminus B') \setminus E(N_0)| \leq 2$.
\end{claim}

Let $e$ and $f$ be two elements of $(B \setminus B') \setminus E(N_0)$. It follows from (\ref*{clm:baddeterminant-2}) that $B'$ is a basis of $M'' \in \{M, M'\}$ if and only
if $B \Delta \{x,y,e,f\}$ and $B \Delta ((B \Delta B') \setminus \{x,y,e,f\})$ are both bases of $M''$.
The latter is not a basis of $M''$ in the case when $|(B \setminus B') \setminus E(N_0)| > 2$, as in this case (\ref*{clm:baddeterminant-2}) implies that 
$A[B \setminus (B' \cup \{e, f\}), B' \setminus (B \cup \{x, y\})]$ has a zero row. This contradicts the fact that $B'$ is a basis of exactly one of $M$ and $M'$, proving (\ref*{clm:baddeterminant-3}).

\begin{claim} \label{clm:baddeterminant-4}
$|(B \setminus B') \setminus E(N_0)| \in \{1,2\}$.
\end{claim}

We suppose that $|(B \setminus B') \setminus E(N_0)| = 0$. Then $B \setminus E(N_0) \subseteq B'$. Hence the set $B' \setminus (B \setminus E(N_0))$ is independent in exactly one of the matroids $M / (B \setminus E(N_0))$ and $M' / (B \setminus E(N_0))$. But it follows from the modularity of $N_0$ and the definition of $M'$ that $M / (B \setminus E(N_0)) = M' / (B \setminus E(N_0))$, a contradiction. This proves (\ref*{clm:baddeterminant-4}), showing that (\ref*{out:baddeterminant-third}) holds.
\\

It remains to show that $B$ and $B'$ satisfy (\ref*{out:baddeterminant-fourth}). If not, then there is an element $w$ of $B' \setminus B$ other than $x$ and $y$. By the modularity of $N_0$ and (\ref*{clm:baddeterminant-2}), there is an element $z \in E(N_0)$ such that $w$ is parallel to $z$ in both $M / (B \cap B' \setminus E(N_0))$ and $M' / (B \cap B' \setminus E(N_0))$. Then $B' \Delta \{w,z\}$ is independent in exactly one of $M$ and $M'$. Also, $|B \Delta (B' \Delta \{w,z\})| = |B \Delta B'|$, so (\ref*{clm:baddeterminant-1}) holds with $B' \Delta \{w,z\}$ in place of $B'$, a contradiction.
\end{proof}

We will need the following two facts about fundamental matrices. For matrices $P$ and $Q$ with the same dimensions, we write $P \leq Q$ when $P_{ij} \leq Q_{ij}$ for each row $i$ and column $j$.

\begin{proposition}[Brualdi, \cite{Brualdi}] \label{prop:brualdi}
Let $B$ be a basis of a matroid $M$, $A$ the fundamental matrix of $M$ with respect to $B$, and $X \subseteq B$ and $Y \subseteq E(M) \setminus B$ sets of the
same size.
If $(B \setminus X) \cup Y$ is a basis of $M$ then there exists a permutation matrix $P$ such that $P \leq A[X,Y]$.
\end{proposition}

\begin{proposition}[Krogdahl, \cite{Krogdahl}] \label{prop:krogdahl}
Let $B$ be a basis of a matroid $M$, $A$ the fundamental matrix of $M$ with respect to $B$, and $X \subseteq B$ and $Y \subseteq E(M) \setminus B$ sets of the
same size.
If there is a unique permutation matrix $P$ such that $P \leq A[X,Y]$, then $(B \setminus X) \cup Y$ is a basis of $M$. 
\end{proposition}

Next we find a strand for $N_0$ in one of $M$ and $M'$ that either distinguishes $M$ and $M'$ or is not a strand in the other.

\begin{lemma} \label{lem:strand}
Let $M$ be a $3$-connected, non-$\F$-representable matroid with a modular restriction $N_0 \cong \PG(2,\F)$ and elements $x,y \in E(M) \setminus E(N_0)$.
Let $M'$ be an $\F$-representable matroid such that $M' \d x = M \d x$ and $M' \d y = M \d y$.
Either there exists a strand for $N_0$ that distinguishes $M$ and $M'$, or there is a set $S$ that is a strand for $N_0$ in one of $M$ or $M'$ and skew to
$E(N_0)$ in the other.
\end{lemma}

\begin{proof}
We choose $B$ and $B'$ as in the statement of \autoref{lem:baddeterminant}. We denote the two elements of $B \setminus B'$ by $e$ and $f$ and let $A$ denote the fundamental matrix of $M$ (and also $M'$) with respect to the basis $B$. We observe that all entries of $A[\{e,f\}, \{x,y\}]$ are equal to 1 by Propositions~\ref{prop:brualdi} and \ref{prop:krogdahl}.

We let $N = M / ((B \cap B') \setminus E(N_0))$ and $N' = M' / ((B \cap B') \setminus E(N_0))$.
If there is a strand distinguishing $N$ and $N'$, then since $E(N_0)$ is closed, the union of this strand with $B \cap B' \setminus E(N_0)$ is a strand distinguishing $M$ and $M'$.
We note that by (\ref*{out:baddeterminant-third}) of \autoref{lem:baddeterminant}, $r(N) \leq r(N_0) + 2$ and at most one of $e,f$ is contained in $E(N_0)$.

\begin{claim}\label{clm:strand-1}
If $r(N) = r(N_0) + 2$ then there is a set $S \subseteq E(M) \setminus E(N_0)$ such that one of $\sqcap_M(S,E(N_0)), \sqcap_{M'}(S,E(N_0))$ is $0$ and the other is $1$.
\end{claim}

The set $(B \cap E(N_0)) \cup \{x,y\}$ is independent in exactly one of $N$ and $N'$. So $\{x,y\}$ is a strand for $N_0$ in one of $N$ and $N'$ and skew to $E(N_0)$ in the other. Then $\{x,y\} \cup ((B \cap B') \setminus E(N_0))$ has the same property in $M$ and $M'$, proving (\ref*{clm:strand-1}).
\\

We may assume that $r(N) = r(N_0) + 1$, and by symmetry that $e \not\in E(N_0)$ and $f \in E(N_0)$.

\begin{claim}\label{clm:strand-2}
If $r(N) = r(N_0) + 1$ then either $\{x, y\}$ is independent in both $N$ and $N'$ or there is a set $S \subseteq E(M) \setminus E(N_0)$ such that one of $\sqcap_M(S, E(N_0)), \sqcap_{M'}(S, E(N_0))$ is $0$ and the other is $1$.
\end{claim}

We assume that $\{x, y\}$ is a parallel pair in one of $N$ and $N'$. Then they are parallel in exactly one of $N$ and $N'$ because $B' \cap E(N)$, which contains $\{x, y\}$, is a basis of one of $N$ and $N'$.
We note that $\{x, e\}$ is not a parallel pair in $N$ or $N'$, otherwise $\{x, y, e\}$ is a parallel class of both matroids.
Let $w$ be the element of $N_0$ in $\cl_N(\{x, e\})$ (and hence also in $\cl_{N'}(\{x, e\})$). Then in the matroid in which $x$ and $y$ are parallel, $\{y, e\}$ also spans $w$; this means $\{x, y, e, w\}$ has rank two in both $N$ and $N'$.

In the matroid in which $\{x, y\}$ are independent, $\{x, y\}$ is a strand for $N_0$, while in the other it is skew to $E(N_0)$. Therefore, $S = \{x, y\} \cup (B \cap B' \setminus E(N_0))$ is a strand for $N_0$ in one of $M$ and $M'$ and skew to $E(N_0)$ in the other.
This proves (\ref*{clm:strand-2}).
\\

Now we may assume that $\{x,y\}$ is independent in both $N$ and $N'$.
Since $A_{ex} = A_{ey} = 1$, neither $x$ nor $y$ are in the closure of $E(N_0)$ in $N$ or $N'$. Then $\{x,y\}$ is a strand in both $N$ and $N'$.

We let $D$ and $D'$ be the fundamental matrices of $N$ and $N'$, respectively, with respect to the basis $B \Delta \{e,y\}$. Since $(B \Delta \{e,y\}) \Delta \{f,x\} = B'$ is independent in exactly one of $N$ and $N'$, it follows that exactly one of $D_{fx}$ and $D'_{fx}$ is equal to 1. Thus $D[E(N_0), \{x\}] \neq D'[E(N_0), \{x\}]$. 
If $w$ is the element of $E(N_0) \cap \cl_N(\{x,y\})$ and $w'$ is the element of $E(N_0) \cap \cl_{N'}(\{x,y\})$, then $D[E(N_0), \{w\}] = D[E(N_0), \{x\}]$ and $D'[E(N_0), \{w'\}] = D'[E(N_0), \{x\}]$. But because $N\d x = N' \d x$, $D[E(N_0), \{w\}] = D'[E(N_0), \{w\}]$, so $w \neq w'$. This proves that $\{x,y\}$ is a strand distinguishing $N$ and $N'$ and so there is a strand distinguishing $M$ and $M'$.
\end{proof}

For disjoint sets $S, T$ in a matroid $M$, we define
\[ \kappa_M(S, T) = \min\{\lambda_M(A) : S \subseteq A \subseteq E(M) \setminus T\}. \]

Let $M$ and $M'$ be two matroids on the same ground set. We write $\Sigma(M, M')$ to denote the set 
\[ \Sigma(M, M') = \{e \in E(M) : M \d e \neq M' \d e \textrm{ and } M / e \neq M' / e\}. \]

We now prove the main lemma of this section.

\begin{lemma} \label{lem:strandspicture}
Let $M$ be a $3$-connected, non-$\F$-representable matroid with a modular restriction $N_0 \cong \PG(2, \F)$ such that no proper minor of $M$ with $N_0$ as a restriction is $3$-connected and non-$\F$-representable. 
If $\lambda_M(E(N_0)) = 3$, $x, y \in E(M) \setminus E(N_0)$ are distinct, and $M'$ is an $\F$-representable matroid with $M' \d x = M \d x$ and $M' \d y = M \d y$, then 
\begin{enumerate}[(i)]
 \item $|\Sigma(M, M')| \geq 2$, and
 \item there are non-nested sets $S, T \subseteq E(M) \setminus E(N_0)$ such that $\cl_M(S) \cap E(N_0) \neq \cl_{M'}(S) \cap E(N_0)$, $\cl_M(T) \cap E(N_0) \neq \cl_{M'}(T) \cap E(N_0)$, and $S \Delta T \subseteq \Sigma(M,M')$.
\end{enumerate}
\end{lemma}

\begin{proof}
We start with some short claims.

\begin{claim} \label{clm:strandspicture-connected2}
$M / E(N_0)$ and $M'/E(N_0)$ are connected.
\end{claim}

If $M / E(N_0)$ is not connected, then by \autoref{prop:separatorismodularsum}, $M$ is a modular sum of two proper restrictions $M_1$ and $M_2$ of $M$ with $E(M_1) \cap E(M_2) = E(N_0)$. Both $M_1$ and $M_2$ are $3$-connected by \autoref{prop:connectivitypropertiesofmodularsum} and so both are $\F$-representable by the choice of $M$. Then \autoref{prop:modularsumpreservesrepresentability} implies that $M$ is $\F$-representable, a contradiction.
Suppose that $M' / E(N_0)$ is not connected. Then $M' / E(N_0)$ has at least two components. Let $A$ be that containing $x$. We can choose a basis $B$ of $M'$ that contains a basis of $N_0$ and does not contain $x$ or $y$. Then $B \cap (E(N_0) \cup A)$ spans $x$ in $M'\d y$ and hence also in $M$. Now $M / E(N_0) \d x$ has at least two components, and we can choose one, $A'$, disjoint from $A \setminus \{x\}$.
But $B \cap (E(N_0) \cup A)$ spans $x$ and so $A$ and $A'$ are components of $M / E(N_0)$, contradicting the fact that $M / E(N_0)$ is connected and proving (\ref*{clm:strandspicture-connected2}).

\begin{claim} \label{clm:strandspicture-connected1}
$\lambda_{M'}(E(N_0)) = 3$.
\end{claim}

Suppose that $\lambda_{M'}(E(N_0)) < 3$; then $\lambda_{M'}(E(N_0)) = 2$ since $M' \d x = M \d x$.
We let $L$ be the line of $N_0$ that is spanned by $E(M') \setminus E(N_0)$.
Since $M'$ is $3$-connected, it has a basis $B$ disjoint from $\{x,y\}$; since $\lambda_{M'}(E(N_0)) = 2$, we may further choose $B$ so that it contains at most one element of $E(N_0) \setminus L$. Then the set $B \cap \cl_{M'}(E(M') \setminus E(N_0))$ is a basis for $M' \d (E(N_0) \setminus L)$ and hence spans $\{x,y\}$. This means that $x,y \in \cl_M(B \cap \cl_{M'}(E(M') \setminus E(N_0)))$, and then $\lambda_M(E(N_0)) = \lambda_{M\d x,y}(E(N_0)) < 3$, a contradiction.
This proves (\ref*{clm:strandspicture-connected1}).

\begin{claim} \label{clm:strandspicture-findtwolines}
If $S$ is a strand for $E(N_0)$ in $N \in \{M,M'\}$, then there exist sets $U_1$ and $U_2$ in $E(N) \setminus E(N_0)$ such that 
$\kappa_{N | (E(N_0) \cup U_1)} (S, E(N_0)) > 1$, $\kappa_{N | (E(N_0) \cup U_2)} (S, E(N_0)) > 1$, and $\cl_N(U_1) \cap E(N_0)$ and $\cl_N(U_2) \cap E(N_0)$ are distinct lines containing $\cl_N(S) \cap E(N_0)$.
\end{claim}

Let $L$ be a line of $N_0$ such that $\sqcap_N(S,L) = 1$, and $A = E(N_0) \setminus L$.
Suppose that $\kappa_{N \d A}(S, L) < 2$. If $\kappa_{N \d A}(S, L) = 0$ then $N / A$ is not connected; so $\kappa_{N \d A}(S,L) = 1$.
We have a $2$-separation $(U,V)$ of $N \d A$ with $S \subseteq U$ and $L \subseteq V$. But since $\sqcap_N(S,L) = 1$, we have $\sqcap_N(U,L) = 1$. Then either $N / E(N_0)$ is not connected or $V = L$, which implies that $\lambda_N(E(N_0)) < 3$; this contradicts either (\ref*{clm:strandspicture-connected2}) or (\ref*{clm:strandspicture-connected1}).

Therefore, $\kappa_{N \d A}(S,L) \geq 2$, and there exists a minimal set $U \subseteq E(N\d A)$ such that $\kappa_{N | (L \cup U \cup S)}(S, L) = 2$.
Choosing two lines $L$ of $N_0$ with $\sqcap_N(S,L) = 1$, we obtain the two sets $U_1$ and $U_2$ as required, proving (\ref*{clm:strandspicture-findtwolines}).

\begin{claim} \label{clm:strandspicture-getthreestrands:1}
Let $N''$ be a restriction of $N \in \{M,M'\}$ containing $N_0$ such that $E(N'') \setminus E(N_0)$ is independent.
If $X$ and $Y$ are minimal strands for $N_0$ in $N''$ such that $\cl_N(X) \cap E(N_0) = \cl_N(Y) \cap E(N_0)$, then $X = Y$.
\end{claim}

Suppose there are two distinct minimal strands $X$ and $Y$ for $N_0$ in $N''$ such that $\cl_N(X) \cap E(N_0) = \cl_N(Y) \cap E(N_0)$.
We denote by $e$ the element of $N_0$ spanned by $X$ and $Y$.
By the minimality of $X$ there exists an element $b \in Y \setminus X$.
Then $b \in \cl_N(Y \cup \{e\} \setminus \{b\}) \subseteq \cl_N(X \cup Y \setminus \{b\})$, contradicting the fact that $E(N'') \setminus E(N_0)$ is independent.
This proves (\ref*{clm:strandspicture-getthreestrands:1}).

\begin{claim} \label{clm:strandspicture-getthreestrands}
If $S$ is a strand for $N_0$ in $N \in \{M,M'\}$ and $U$ is a subset of $E(N) \setminus E(N_0)$ containing $S$ such that $\kappa_{N | (E(N_0) \cup U)} (S, E(N_0)) > 1$, then $U$ contains two strands $T_1$ and $T_2$ such that $\cl_N(S) \cap E(N_0)$, $\cl_N(T_1) \cap E(N_0)$ and $\cl_N(T_2) \cap E(N_0)$ are all distinct.
\end{claim}

We may assume that $U$ is minimal and thus independent, and that $S$ is minimal.
We pick any element $z \in S$. Then $\sqcap_N(U \setminus \{z\}, E(N_0)) \geq 1$ so $U \setminus \{z\}$ contains a minimal strand $T_1$. It follows from (\ref*{clm:strandspicture-getthreestrands:1}) that $\cl_N(T_1) \cap E(N_0) \neq \cl_N(S) \cap E(N_0)$.

If $S \cap T_1 \neq \emptyset$, then there is an element $s \in S \cap T_1$ and $U \setminus \{s\}$ contains a minimal strand $T_2$ for $N_0$ distinct from $S$ and $T_1$.
Similarly, if for some $s \in S$ and $t \in T_1$, $\sqcap_N(U \setminus \{s,t\}, E(N_0)) \geq 1$, then $U \setminus \{s,t\}$ contains a minimal strand $T_2$ for $N_0$ distinct from $S$ and $T_1$.
In both cases, (\ref*{clm:strandspicture-getthreestrands:1}) implies that $\cl_N(T_2) \cap E(N_0)$ is distinct from $\cl_N(S) \cap E(N_0)$ and $\cl_N(T_1) \cap E(N_0)$.

Therefore, we may assume that $S$ and $T_1$ are disjoint and that for any $s \in S$ and $t \in T_1$, $\sqcap_N(U \setminus \{s,t\}, E(N_0)) = 0$. This means that for any $u \in U$, $\sqcap_N(U \setminus \{u\}, E(N_0)) = 1$, and $S$ and $T_1$ are in the coclosure of $E(N_0)$.

But because $S$ and $T_1$ are disjoint, $\sqcap_N(U \setminus T_1, E(N_0)) = 1$, so $T_1$ is a series class of $N | (E(N_0) \cup U)$. For the same reason, $S$ is a series class of $N | (E(N_0) \cup U)$.
This contradicts the fact that $\kappa_{N | (E(N_0) \cup U)}(S, E(N_0)) = 2$, proving (\ref*{clm:strandspicture-getthreestrands}).
\\

We now let $N$ and $N'$ be matroids such that $\{N,N'\} = \{M, M'\}$.

\begin{claim} \label{clm:strandspicture-gettwodistinguishers}
If there is an independent strand $S$ for $N_0$ in $N$ such that $\sqcap_{N'}(S, E(N_0)) = 0$, then either
\begin{itemize}
\item there is a strand for $N_0$ distinguishing $N$ and $N'$, or
\item there is a strand $T$ for $N_0$ in $N$ such that $\cl_N(T) \cap E(N_0) \neq \cl_N(S) \cap E(N_0)$ and $\sqcap_{N'}(T,E(N_0)) = 0$.
\end{itemize}
\end{claim}

We let $U$ be a minimal set containing $S$ such that $\kappa_{N | (E(N_0) \cup U)} (S, E(N_0)) > 1$. By (\ref*{clm:strandspicture-getthreestrands}), $U$ contains two strands $T_1$ and $T_2$ for $N_0$ such that $\cl_N(S) \cap E(N_0), \cl_N(T_1) \cap E(N_0)$, and $\cl_N(T_2) \cap E(N_0)$ are distinct.
We may assume that $\sqcap_{N'}(T_1, E(N_0)) > 0$ and $T_1$ is not a strand distinguishing $N$ and $N'$, and that the same holds for $T_2$.
Therefore, $\sqcap_{N'}(U,E(N_0)) \geq 2$.

Let $e \in U \setminus S$. By the minimality of $U$, $\sqcap_N(U \setminus \{e\}, E(N_0)) = 1$. If $\cl_{N'}(U \setminus \{e\}) \cap E(N_0) \not\subseteq \cl_N(S) \cap E(N_0)$, then there is a strand $T$ for $N_0$ in $N'$ such that $S \subseteq T \subseteq U \setminus \{e\}$ and $\cl_{N'}(T) \cap E(N_0) \neq \cl_{N}(S) \cap E(N_0)$, proving (\ref*{clm:strandspicture-gettwodistinguishers}). So we may assume that $\cl_{N'}(U \setminus \{e\}) \cap E(N_0) = \cl_N(S) \cap E(N_0)$ and $\sqcap_{N'}(U \setminus \{e\}, E(N_0)) = 1$.

The fact that $\sqcap_{N'}(U,E(N_0)) \geq 2$ but $\sqcap_{N'}(U \setminus \{e\}, E(N_0)) = 1$ for all $e \in U \setminus S$ implies that $U \setminus S$ is contained in the coclosure of $E(N_0)$ in $N'$.
Then since $\sqcap_{N'}(S, E(N_0)) = 0$, it follows that $r^*_{N'}(U \setminus S) \geq 2$.
Furthermore, for any two elements $e,f \in U \setminus S$ that are not in series in $N'$, $\sqcap_{N'}(U \setminus \{e,f\}, E(N_0)) = 0$.

Therefore, for each series class $X$ of $U \setminus S$, there is a circuit containing $\cl_N(S) \cap E(N_0)$ and $(U \setminus S) \setminus X$.
This implies that for each series class $X$ of $U \setminus S$, $X \subseteq \cl_{N'}(U \setminus X)$. But then $\sqcap_{N'}(U, E(N_0)) = \sqcap_{N'}(U \setminus X, E(N_0)) = 1$, a contradiction.
This proves (\ref*{clm:strandspicture-gettwodistinguishers}).

\begin{claim} \label{clm:strandspicture-strandspicturev1}
If there exists a strand $S$ for $N_0$ in $N$ such that either $S$ distinguishes $N$ and $N'$ or $\sqcap_{N'}(S, E(N_0)) = 0$, then
\begin{itemize}
\item $|\Sigma(N, N')| \geq 2$, and
\item there is another strand $T$ for $N_0$ in $N$ such that $\cl_N(S) \cap E(N_0) \neq \cl_N(T) \cap E(N_0)$, $\cl_N(T) \cap E(N_0) \neq \cl_{N'}(T) \cap E(N_0)$ and $S \Delta T \subseteq \Sigma(N, N')$.
\end{itemize}
\end{claim}

First we suppose that there is no strand that distinguishes $N$ and $N'$; so $\sqcap_{N'}(S, E(N_0)) = 0$. Then by (\ref*{clm:strandspicture-gettwodistinguishers}), there is a strand $T$ for $N_0$ in $N$ such that $\cl_N(T) \cap E(N_0) \neq \cl_N(S) \cap E(N_0)$ and $\sqcap_{N'}(T,E(N_0)) = 0$.
We observe that $S \not\subseteq T$ and $T \not\subseteq S$, and $S \Delta T \subseteq \Sigma(N,N')$.

Therefore, we may assume that $S$ is a strand for $N_0$ that distinguishes $N$ and $N'$.
By (\ref*{clm:strandspicture-findtwolines}), there are two lines $L_1$ and $L_2$ of $N_0$ containing $\cl_N(S) \cap E(N_0)$, and sets $U_1$ and $U_2$ containing $S$ such that $\cl_N(U_1) \cap E(N_0) = L_1$, $\cl_N(U_2) \cap E(N_0) = L_2$, $\kappa_{N | (E(N_0) \cup U_1)} (S, E(N_0)) > 1$, and $\kappa_{N | (E(N_0) \cup U_2)} (S, E(N_0)) > 1$.
Then by (\ref*{clm:strandspicture-getthreestrands}), each of $U_1$ and $U_2$ contains two more strands that span distinct elements of $E(N_0) \setminus \cl_N(S)$.
If all of these four strands are strands for $N_0$ in $N'$ that do not distinguish $N$ and $N'$, then $\cl_{N'}(U_1) = L_1$ and $\cl_{N'}(U_2) = L_2$, so $\cl_{N'}(S) = L_1 \cap L_2 = \cl_N(S)$, contradicting the fact that $S$ distinguishes $N$ and $N'$.
Therefore, there exists a strand $S'$ for $N_0$ in $N$ with $\cl_N(S') \cap E(N_0) \neq \cl_N(S) \cap E(N_0)$ such that either $S'$ is a strand for $N_0$ in $N'$ that distinguishes $N$ and $N'$, or $S'$ is not a strand for $N_0$ in $N'$.

Now we observe that $S \Delta S' \subseteq \Sigma(N,N')$, and since $\cl_N(S) \cap E(N_0) \neq \cl_N(S') \cap E(N_0)$, we have $|S \Delta S'| \geq 2$.
This proves (\ref*{clm:strandspicture-strandspicturev1}).
\\

According to \autoref{lem:strand}, there is a strand for $N_0$ in one of $M$ or $M'$ that is either a strand for $N_0$ in the other matroid and distinguishes $M$ and $M'$, or is skew to $E(N_0)$ in the other matroid.
The result now follows from (\ref*{clm:strandspicture-strandspicturev1}).
\end{proof}

\section{Connectivity} \label{sec:connectivity}

In this section we prove \autoref{lem:keylemma}.
First, we state two useful results on matroid connectivity.

\begin{theorem}[Tutte, \cite{Tutte:connectivityinmatroids}] \label{thm:tutteconnectedtheorem}
 If $M$ is a connected matroid and $e \in E(M)$, then at least one of $M \d e$ and $M / e$ is connected.
\end{theorem}

The second is another theorem of Tutte \cite{Tutte:mengerstheoremformatroids} that generalizes Menger's Theorem to matroids.

\begin{tutteslinkingtheorem}
If $M$ is a matroid and $S, T \subseteq E(M)$ are disjoint then $\kappa_M(S, T) = \max\{ \sqcap_{M / Z}(S, T) : Z \subseteq E(M) \setminus (S \cup T) \}$.
\end{tutteslinkingtheorem}

Recall that sets $S$ and $T$ in a matroid $M$ are called skew if $\sqcap_M(S, T) = 0$. 
We can choose the set $Z$ that attains the maximum in Tutte's Linking Theorem so that it is skew to both $S$ and $T$.
We have the following stronger version of the theorem, for which an explicit proof can be found in \cite[Theorem 8.5.7]{Oxley} or \cite[Theorem 4.2]{GeelenGerardsWhittle:excludingaplanargraph}.

\begin{tutteslinkingtheorem2}
If $M$ is a matroid and $S, T \subseteq E(M)$ are disjoint then there is a set $Z \subseteq E(M) \setminus (S \cup T)$ such that $\sqcap_{M / Z}(S, T) = \kappa_M(S, T)$, $(M / Z) | S = M | S$, and $(M / Z) | T = M | T$.
\end{tutteslinkingtheorem2}

Before the main result of this section, we prove one last short lemma that is similar to Lemma~2.3 of \cite{GeelenGerardsKapoor}.

\begin{lemma} \label{lem:remove}
Let $\F$ be a finite field, $M$ a $3$-connected matroid with a restriction $N_0 \cong \PG(2, \F)$ and a deletion pair $x,y \in E(M) \setminus E(N_0)$, and $M'$ an $\F$-representable matroid with $M' \d x = M \d x$ and $M' \d y = M \d y$.
If there are sets $C, D \subset E(M)$ disjoint from $E(N_0) \cup \{x,y\}$ such that
\begin{enumerate}[(a)]
 \item $\{x, y\}$ is coindependent in $M \d D / C$, \label{hyp:remove-coindependent}
 \item $(M \d D / C) \d x$ and $(M \d D / C) \d y$ are stable, \label{hyp:remove-stable}
 \item $(M \d D / C) \d x,y$ is connected, and \label{hyp:remove-connected}
 \item $M \d D / C \neq M' \d D / C$, \label{hyp:remove-notequal}
\end{enumerate}
then $M \d D / C$ is not $\F$-representable.
\end{lemma}

\begin{proof}
Since $M\d x$ and $M\d y$ are $\F$-representable, so are $(M \d D /C) \d x$ and $(M \d D / C) \d y$. By Lemmas~\ref{lem:uniquerepresentablematroid1} and \ref{lem:uniquerepresentablematroid2} applied to $M \d D / C$, there is a unique $\F$-representable matroid $N$ such that $N \d x = (M \d D / C) \d x$ and $N \d y = (M \d D / C) \d y$. But $M' \d D / C$ satisfies this condition and is $\F$-representable, so $N = M' \d D / C$. Then $M \d D / C \neq N$, so $M\d D / C$ is not $\F$-representable.
\end{proof}

We now prove \autoref{lem:keylemma}, which we restate for convenience.
We will use the fact that a matroid $M$ that is not $3$-connected is a direct sum or a $2$-sum of matroids isomorphic to proper minors of $M$ (see \cite[Theorem 8.3.1]{Oxley} for a proof).

\keylemma*

\begin{proof}
First, we need the following easy fact.

\begin{claim} \label{clm:keylemma-nonbinary2sum}
Let $M = M_1 \oplus_2 M_2$ with $B = E(M_2) \setminus E(M_1)$ such that $|B| = 3$. If $M_2$ is non-binary, then $B$ is a triangle and a triad in $M$.
\end{claim}

Since $M_2$ is non-binary and has four elements, it is isomorphic to $U_{2,4}$ and has no series pairs. Thus $B$ is a triangle in $M$. Since $\lambda_M(B) = 1$ and $r_M(B) = 2$, we have $r^*_M(B) = 2$ and $B$ is a triad of $M$.
This proves (\ref*{clm:keylemma-nonbinary2sum}).
\\

By choosing $M_0$ minimally, we may assume that it has no proper minor that is $3$-connected, non-$\F$-representable, and has $N_0$ as a restriction.
We assume that $\lambda_{M_0}(E(N_0)) = 3$ to obtain a contradiction.

By \autoref{lem:getcounterexamplewithdeletionpair}, there is a $3$-connected, non-$\F$-representable matroid $M$ with $N_0$ as a modular
restriction and $\lambda_M(E(N_0)) = 3$ such that $M$ has a deletion pair $x,y \in E(M) \setminus E(N_0)$ and no proper minor of $M$ containing $N_0$ is
$3$-connected and non-$\F$-representable. Furthermore, $M \d x,y$ has at most one series pair.

By \autoref{lem:uniquegfqmatroid} there is an $\F$-representable matroid $M'$ such that $M \d x = M' \d x$ and $M \d y = M' \d y$.
We recall that $\Sigma(M, M')$ is the set of elements $e \in E(M) \setminus E(N_0)$ such that $M \d e \neq M' \d e$ and $M / e \neq M' / e$. By \autoref{lem:strandspicture}, $|\Sigma(M,M')| \geq 2$.

\begin{claim}\label{clm:keylemma-twounstableminors:1}
If $e \in E(M) \setminus E(N_0)$, then $M \d e$ and $M / e$ are $\F$-representable.
\end{claim}

We let $P$ be either $M \d e$ or $M / e$.
Suppose that $P$ is not $\F$-representable. Then the fact that $P$ is a proper minor of $M$ implies that it is not $3$-connected. Since $N_0$ is $3$-connected, there exists a matroid $P'$ containing $N_0$ such that $\si(P')$ is $3$-connected and $P$ is obtained by $2$-sums of $P'$ with matroids $M_1, \ldots, M_t$. Since $P'$ is isomorphic to a proper minor of $M$, it is $\F$-representable; this means that for some $i$, $M_i$ is not $\F$-representable.
We let $U = E(M_i) \setminus E(P')$ and $V = E(P) \setminus U$. Then $(U,V)$ is a $2$-separation of $P$, and $(U \cup \{e\}, V)$ is a $3$-separation of $M$.
By Tutte's Linking Theorem, there is a minor $N$ of $M$ such that $E(N) = E(N_0) \cup U \cup \{e\}$ and $\lambda_N(E(N_0)) = 2$. Since $\lambda_N(U \cup \{e\}) = \lambda_M(U \cup \{e\})$, we note that $N | (U \cup \{e\}) = M | (U \cup \{e\})$. This means that $N$ is $3$-connected, as $M$ is.
Furthermore, $\sqcap_N(U, E(N_0)) = 1$ so the modularity of $N_0$ implies that $N$ contains a restriction isomorphic to $M_i$. Therefore, $N$ is not $\F$-representable. But then $N$ is not a proper minor of $M$, so $N = M$, contradicting the fact that $\lambda_M(E(N_0)) \geq 3$.
This proves (\ref*{clm:keylemma-twounstableminors:1}).

\begin{claim}\label{clm:keylemma-twounstableminors:2}
If $e \in \Sigma(M, M')$, and $N$ is one of $M \d e$ or $M / e$, then either $N \d x$ is not stable, $N \d y$ is not stable, or $N\d x,y$ is not connected.
\end{claim}

Since $M \d x,y$ is connected, $\{x,y,e\}$ is not a triad of $M$ and $\{x,y\}$ is coindependent in $M\d e$.
With $C = \emptyset$ and $D = \{e\}$, hypotheses (\ref*{hyp:remove-coindependent}) and (\ref*{hyp:remove-notequal}) of \autoref{lem:remove} are satisfied by $M\d D/C$.
But $M\d e$ is $\F$-representable by (\ref*{clm:keylemma-twounstableminors:1}).
Therefore, (\ref*{hyp:remove-stable}) and (\ref*{hyp:remove-connected}) of \autoref{lem:remove} do not both hold, and we conclude that either $M\d e \d x$ is not stable, $M\d e \d y$ is not stable, or $M\d e \d x,y$ is not connected. 

Similarly, with $C = \{e\}$ and $D = \emptyset$, (\ref*{hyp:remove-coindependent}) and (\ref*{hyp:remove-notequal}) of \autoref{lem:remove} are satisfied by $M \d D / C$, and $M / e$ is $\F$-representable by (\ref*{clm:keylemma-twounstableminors:1}).
Therefore, (\ref*{hyp:remove-stable}) and (\ref*{hyp:remove-connected}) of \autoref{lem:remove} do not both hold, and either $M /e \d x$ is not stable, $M /e \d y$ is not stable, or $M/e \d x,y$ is not connected.
This proves (\ref*{clm:keylemma-twounstableminors:2}).

\begin{claim}\label{clm:keylemma-twounstableminors}
If $e \in \Sigma(M,M')$ then either $M \d x \d e$ and $M \d y /e$ are not stable, or $M \d x / e$ and $M \d y \d e$ are not stable.
\end{claim}

Up to symmetry between $x$ and $y$, (\ref*{clm:keylemma-twounstableminors:2}) implies that one of the following five cases occurs:\\

\noindent
\emph{(a)} $M \d x,y \d e$ is not connected and $M \d x,y /e$ is not connected,\\
\emph{(b)} $M \d x,y \d e$ is not connected and $M \d x / e$ is not stable,\\
\emph{(c)} $M \d x,y /e$ is not connected and $M \d x \d e$ is not stable,\\
\emph{(d)} $M \d x \d e$ is not stable and $M \d x / e$ is not stable, or\\
\emph{(e)} $M \d x \d e$ is not stable and $M \d y /e$ is not stable.\\

As $M\d x,y$ is connected, case (a) contradicts \autoref{thm:tutteconnectedtheorem}, and since $M\d x$ is $3$-connected, case (d) contradicts Bixby's Lemma.
We suppose case (b) holds.
Since $M\d x / e$ is not stable, Bixby's Lemma implies that $M \d x \d e$ is internally $3$-connected. Then since $M \d x,y \d e$ is not connected, $y$ is in a series pair of $M \d x \d e$, so $\{x,y,e\}$ is a triad of $M$. This contradicts the fact that $M \d x,y$ is connected.
Next, we suppose that case (c) holds.
Since $M \d x \d e$ is not stable, by Bixby's Lemma $M \d x / e$ is internally $3$-connected. There are no series pairs in $M\d x / e$, so $M \d x,y /e$ is connected, a contradiction.
We conclude that, up to symmety, case (e) holds, which proves (\ref*{clm:keylemma-twounstableminors}).

\begin{claim} \label{clm:keylemma-Mxynot3conn}
$M \d x, y$ is not $3$-connected.
\end{claim}

We say that a $2$-separation $(A, B)$ in a matroid $N$ corresponds to a $2$-sum of non-binary matroids if $N = N_1 \oplus_2 N_2$ for some non-binary matroids $N_1$ and $N_2$ with $A = E(N_1) \setminus E(N_2)$ and $B = E(N_2) \setminus E(N_1)$.

Let $e \in \Sigma(M,M')$. From (\ref*{clm:keylemma-twounstableminors}) we may assume that $M \d x \d e$ and $M \d y / e$ are not stable.
Suppose $M \d x,y$ is $3$-connected. Then by Bixby's Lemma, either $M \d x,y /e$ or $M \d x,y \d e$ is internally $3$-connected. 

First, assume that $M \d x,y / e$ is internally $3$-connected.
Let $(A, B)$ be a $2$-separation of $M \d y /e$ corresponding to a $2$-sum of two non-binary matroids, with $x \in B$. Since $M \d y / e \d x$ is internally $3$-connected, $|B| = 3$. Since $M\d x,y$ is $3$-connected, $M \d x, y / e$ has no series pairs, so $B \setminus \{x\}$ is a parallel pair of $M \d y / e$, contradicting the fact that $(A, B)$ is a $2$-separation that corresponds to a $2$-sum of two non-binary matroids.

Therefore, $M \d x, y / e$ is not internally $3$-connected, and $M \d x, y \d e$ is.
Let $(A, B)$ be a $2$-separation of $M \d x \d e$ corresponding to a $2$-sum of two non-binary matroids, with $y \in B$. Then $|B| = 3$ since $M \d x, y \d e$ is internally $3$-connected. 
So by (\ref*{clm:keylemma-nonbinary2sum}), $B$ is a triangle and a triad of $M \d x \d e$ containing $y$.
Denote the other two elements of $B$ by $a$ and $b$. Since $M \d x$ is $3$-connected, $\{a, b, y, e\}$ is a cocircuit of $M \d x$ that contains the triangle $\{a, b, y\}$, and $\{a, b, e\}$ is a triad of $M \d x, y$.

Let $(C, D)$ be an internal $2$-separation of $M \d x, y / e$. Then $e \in \cl_M(C) \cap \cl_M(D)$, so $a$ and $b$ are not both contained in the same one of $C$ or $D$ because $\{a, b, e\}$ is a cocircuit of $M \d x, y$.
Thus we may assume $a \in C$, $b \in D$.
So in $M \d x$ we have $\sqcap_{M \d x}(\{e, a\}, C \setminus \{a\}) = 1$, $\sqcap_{M \d x}(\{e, b\}, D \setminus \{b\}) = 1$, and $\sqcap_{M \d x}(\{a, b, y\}, C \cup D \setminus \{a, b\}) = 1$, but $\{a, b, y\}$ and $\{y, e\}$ are each skew to both $C \setminus \{a\}$ and $D \setminus \{b\}$.

We pick a second element $f \in \Sigma(M, M')$. Then either $M \d x \d f$ and $M\d y /e$ are not stable, or $M \d x /f$ and $M\d y \d f$ are not stable. First we assume that $M\d x \d f$ and $M \d y / f$ are not stable.
By the same argument that was applied to $e$, $M\d x$ has a cocircuit $\{c,d,f,y\}$ with a triangle $\{c,d,y\}$, and there is an internal $2$-separation of $M\d x,y/f$, $(U,V)$ with $c \in U, d \in V$.
So $\{e,a,b\}$ and $\{f,c,d\}$ are both triads in $M\d x,y$, and $\{y, a, b\}$ and $\{y, c, d\}$ are both triangles of $M \d x$.
But the only triangle of $M \d x$ containing $y$ is $\{y, a, b\}$, so we may assume that $a = c$ and $b = d$. 
Now $(V \setminus \{b\}) \cup \{f\}$ spans $b$, so $e \in (V \setminus \{b\}) \cup \{f\}$. But $e \neq f$ so $e \in V$. Symmetrically, $(U \setminus \{a\}) \cup \{f\}$ spans $a$ so we also have $e \in U$, a contradiction.

Therefore, there are only two elements of $\Sigma(M, M')$, $e$ and $f$, and
$M \d x / f$ and $M \d y \d f$ are not stable.

Applying to $f$ the same argument as for $e$ but with $x$ and $y$ swapped, we see that $M \d y$ has a cocircuit $\{c, d, x, f\}$ with a triangle $\{c, d, x\}$.

We assume that $e \not\in \{c, d\}$ and $f \not\in \{a, b\}$. Then $\{a, b, e\}$ and $\{c, d, f\}$ are disjoint because $\{a, b\}$ is a series class of $M \d x, y, e$ but $\{c, d\}$ is not.
By \autoref{lem:strandspicture}, there are two distinct sets $S, T \subseteq E(M) \setminus E(N_0)$ such that $S \Delta T = \{e, f\}$, $\cl_M(S) \cap E(N_0) \neq \cl_{M'}(S) \cap E(N_0)$ and $\cl_M(T) \cap E(N_0) \neq \cl_{M'}(T) \cap E(N_0)$. 
By symmetry, we may assume $S \setminus T = \{e\}$ and $T \setminus S = \{f\}$.
Note that $x, y \in S \cap T$. We may also assume that $S$ and $T$ are minimal, so neither contains $\{a, b\}$ or $\{c, d\}$, both of which are in triangles with $y$.
Suppose that $a \in S$. then $\cl_M(S) \cap E(N_0) = \cl_M((S \cup \{b\}) \setminus \{y\}) \cap E(N_0)$ since $\{a, b, y\}$ is a triangle. But then $\cl_M(S) \cap E(N_0) = \cl_{M'}((S \cup \{b\}) \setminus \{y\}) \cap E(N_0)$, which equals $\cl_{M'}(S)$ because $\{a, b, y\}$ is also a triangle of $M'$. This is a contradiction, so $a \not\in S$, and by the symmetric argument, $b \not\in S$.
Suppose that $a, b, e \not\in T$. Then since $\{a, b, e, y\}$ is a cocircuit of $M$, $\cl_M(T) \cap E(N_0) = \cl_M(T \setminus \{y\}) \cap E(N_0) = \cl_{M'}(T \setminus \{y\}) \cap E(N_0)$. But this equals $\cl_{M'}(T) \cap E(N_0)$, because $\{a, b, y, e\}$ is a union of cocircuits in $M' \d x$, a contradiction. Hence $T$ contains at least one of $a$ or $b$, and by symmetry we may assume $a \in T$. Then $a \in S \Delta T = \{e, f\}$, so we have $a = f$, contradicting our assumption that $f \not\in \{a, b\}$.

Therefore, we may assume by symmetry that $a = f$.
Hence $\{f, b, y\}$ is a triangle, in both $M$ and $M'$. But then $\cl_M(T) = \cl_M((T \setminus \{y\}) \cup \{b\})$ and $\cl_{M'}(T) = \cl_{M'}((T \setminus \{y\}) \cup \{b\})$, so these sets are equal, a contradiction.
This proves (\ref*{clm:keylemma-Mxynot3conn}).
\\

Note that, since our deletion pair $\{x, y\}$ was arbitrary up to the assumption that $M \d x, y$ has at most one series pair, (\ref*{clm:keylemma-Mxynot3conn}) implies that there is no deletion pair $x', y' \in E(M) \setminus E(N_0)$ such that $M \d x', y'$ is $3$-connected.
Whenever $u, v \in E(M)$ are elements such that $M \d u, v$ is $3$-connected, then $M \d u$ and $M \d v$ are internally $3$-connected. But they have no parallel pairs so they are actually $3$-connected, and $\{u, v\}$ is a deletion pair of $M$. Therefore, there are no two distinct elements $u, v \in E(M) \setminus E(N_0)$ such that $M \d u, v$ is $3$-connected.

\begin{claim} \label{clm:keylemma-enotinseriespair}
If $e \in \Sigma(M, M')$, then $e$ is not in a series pair of $M \d x, y$.
\end{claim}

From (\ref*{clm:keylemma-twounstableminors}) and the symmetry between $x$ and $y$, we may assume that $M \d y / e$ is not stable.
Suppose $e$ is in a series pair of $M \d x,y$. Since $M \d x,y$ is internally $3$-connected with at most one series pair, $M \d x,y /e$ is $3$-connected. This
contradicts the fact that $M \d y / e$ is not internally $3$-connected, proving (\ref*{clm:keylemma-enotinseriespair}).
\\

By \autoref{lem:strandspicture}, there is an element $e \in \Sigma(M, M')$ and sets $S, T \subseteq E(M) \setminus E(N_0)$ such that $e \in S \setminus T$, $\cl_M(S) \cap E(N_0) \neq \cl_{M'}(S) \cap E(N_0)$, and $\cl_M(T) \cap E(N_0) \neq \cl_{M'}(T) \cap E(N_0)$.
From (\ref*{clm:keylemma-twounstableminors}) we may assume that $M \d x \d e$ and $M \d y / e$ are not stable.
By (\ref*{clm:keylemma-Mxynot3conn}), $M \d x, y$ has exactly one series pair; we denote it by $\{a, b\}$.
Then $M \d x, y / a$ is $3$-connected, and by Bixby's Lemma either $M \d x, y / a \d e$ or $M \d x, y / a / e$ is internally $3$-connected.

\begin{claim} \label{clm:keylemma-abnotindistinguishingsets}
For any set $H \subseteq E(M)$, if $a \in H$ or $b \in H$ then $\cl_M(H) = \cl_{M'}(H)$.
\end{claim}

Since $M \d x, y / a$ is $3$-connected, it follows that $M \d x / a$ and $M \d y / a$ are both stable and $M \d x, y / a$ is connected. Also, $M / a$ is $\F$-representable by (\ref*{clm:keylemma-twounstableminors:1}). Therefore, \autoref{lem:remove} implies that $M / a = M' / a$. 
For any set $H \subseteq E(M)$ containing $a$, $\cl_{M / a}(H \setminus \{a\}) = \cl_{M' / a}(H \setminus \{a\})$, which means that $\cl_M(H) = \cl_{M'}(H)$.
The same argument applies with $b$ in place of $a$, proving (\ref*{clm:keylemma-abnotindistinguishingsets}).

\begin{claim}\label{clm:keylemma-trianglethingcase:1}
$M \d x, y, e$ is not internally $3$-connected.
\end{claim}

We assume that $M \d x, y, e$ is internally $3$-connected.
Since $M \d x \d e$ is not stable, it has an internal $2$-separation $(A, B)$ where $y \in B$, and by (\ref*{clm:keylemma-nonbinary2sum}), $B$ is a triangle and a triad.
We let $c$ and $d$ be the other two elements of $B$. 
Then $B$ is coindependent in $M \d x$ and $\{c, d\}$ is coindependent in $M \d x, y$. Thus at most one of $a, b$ is in $\{c, d\}$; but $\{y, a, b\}$ is a triad of $M \d x$ and $\{y, c, d\}$ is a triangle, so $\{c, d\}$ contains exactly one of $a$ or $b$. We may assume that $c = b$ so $B = \{y, b, d\}$ and $d \neq a$.
Then $a \in A$ and $\{a, b\}$ is not a series pair of $M \d x, y, e$, because $\{b, d\}$ is and $M \d x, y, e$ is internally $3$-connected. This contradicts the fact that $\{a, b\}$ is a series pair of $M \d x, y$, proving (\ref*{clm:keylemma-trianglethingcase:1}).

\begin{claim} \label{clm:keylemma-trianglethingcase:2}
 $M \d x, y / a, e$ is internally $3$-connected if and only if $M \d x, y / e$ is internally $3$-connected.
\end{claim}

Suppose that $M \d x, y / e$ is internally $3$-connected. Since $\{a, b\}$ is a series pair of $M \d x, y$, it is a series pair of $M \d x, y / e$, so $M \d x, y / a, e$ is also internally $3$-connected.

Conversely, suppose that $M \d x, y / a, e$ is internally $3$-connected and $M \d x, y / e$ is not internally $3$-connected.
Since $\{a, b\}$ is a series pair, $M \d x, y / e$ has an internal $2$-separation $(W, Z)$ with $a, b \in Z$. But $M \d x, y / a, e$ is internally $3$-connected, so $|Z| = 3$; let $c$ denote the third element of $Z$. Recall that $M \d x, y / a$ is $3$-connected, so $M \d x, y / a, e$ has no series pairs. Hence $\{b, c\}$ is not a series pair of $M \d x, y / a, e$ so it is a parallel pair, and $Z$ is a triangle of $M \d x, y / e$. It is not a triad, however, since $\{a, b\}$ is a series pair, so $c \in \cl_{M \d x, y / e}(W)$. Note that $(W, \{a, b, c\})$ is the unique (up to ordering parts) $2$-separation of $M \d x, y / e$.
Hence since $M \d y / e$ is not stable, $x$ must lie in $\cl_{M \d y / e}(\{a, b\})$ but not in $\cl_{M \d y / e}(W)$. So $(W, \{a, b, c, x\})$ is a $2$-separation of $M \d y / e$ and $\{a, b, c, x\}$ is a four-point line in $M \d y / e$.
Therefore, $(W, \{a, b, c, e, x\})$ is a $3$-separation of $M \d y$ and $c, e \in \cl_M(W) \cap \cl_M(\{a, b, x\})$.
By our remarks after (\ref*{clm:keylemma-Mxynot3conn}), we know that $M \d y, c$ is not $3$-connected, so it has a $2$-separation $(A, B)$ with at least two of $a, b, x$ in $B$. Since $\{a, b, x\}$ is a triad, $\lambda_{M \d y, c}(B \cup \{a, b, x\}) = 1$, but $(A \setminus \{a, b, x\}, B \cup \{a, b, x\})$ is not a $2$-separation of $M \d y, c$ since $c \in \cl_M(\{a, b, x\})$. Therefore, $|A| = 2$, $A$ is a series pair of $M \d y, c$, and $A$ contains one element of $\{a, b, x\}$. 
This implies that there is a triad $L$ of $M \d y$ containing $c$ and precisely one element of $\{a, b, x\}$. Since $e \in \cl_M(W)$, $e \not\in L$.
Therefore, either $c \not\in \cl_M(\{b, x, e\})$, $c \not\in \cl_M(a, x, e\})$, or $c \not\in \cl_M(\{a, b, e\})$. But $r_M(\{a, b, x, e, c\}) = 3$, so one of $\{b, x, e\}$, $\{a, x, e\}$ and $\{a, b, e\}$ is a triangle of $M \d y$, contradicting the fact that $\{a, b, c, x\}$ is a four-point line in $M \d y / e$.
This proves that $M \d x, y / e$ is internally $3$-connected and proves (\ref*{clm:keylemma-trianglethingcase:2}).

\begin{claim} \label{clm:keylemma-circuitflat}
 If $M \d x, y / e$ is internally $3$-connected then $\{a, b, x, e\}$ is a circuit and a flat of $M$ and $M$ has a three- or four-element cocircuit $L$ containing $e$, one of $a$ or $b$, and another element $c \in E(M) \setminus \{a, b, x, y, e\}$, and if $|L| = 4$ then $y \in L$.
\end{claim}

Since $M \d y / e$ is not stable, it has an internal $2$-separation $(U, V)$ with $x \in V$. Since $(U, V \setminus \{x\})$ is not an internal $2$-separation of $M \d x, y / e$, $|V| = 3$ and $V$ is closed in $M \d y / e$. Hence by (\ref*{clm:keylemma-nonbinary2sum}), $V$ is a triangle and a triad of $M \d y / e$. Therefore, $V \setminus \{x\}$ is a series pair of $M \d x, y$, so $V = \{a, b, x\}$. Also, $(U, V \cup \{e\})$ is a $3$-separation of $M \d y$ with $e \in \cl_{M \d y}(U)$, so $r_M(\{a, b, x, e\}) = 3$, and since $\{a, b, x\}$ is independent in $M \d y / e$, $\{a, b, x, e\}$ is a circuit of $M$. Since $M \d x / e$ is internally $3$-connected, $(U, \{a, b, y\})$ is not a $2$-separation of $M \d x / e$, so $y$ is not in $\cl_M(\{a, b, e\})$. Since $V$ is closed in $M \d x, y / e$, $V \cup \{e\}$ is a flat of $M$. 

By Bixby's Lemma, $M \d y, e$ is internally $3$-connected. But it is not $3$-connected, so it has a series pair and thus $e$ is contained in a triad $L$ of $M \d y$. At least one other element of the circuit $\{a, b, x, e\}$ is in $L$, but $x \not\in L$ since $e$ is not in a series pair in $M \d x, y$. So $L$ contains one of $a$ and $b$. Recall that $\{a, b, x\}$ is a triad of $M \d y$, and $e \in \cl_M(\{a, b, x\}$; thus as $M \d y$ is $3$-connected, $e \in \cl_M(E(M) \setminus \{a, b, x, y\})$ and so $\{a, b, e\}$ is not a triad. Therefore, $L$ contains exactly one of $a$ and $b$, plus another element $c$. 
Since $L$ is a triad of $M \d y$, either $L$ is a triad of $M$ or $L \cup \{y\}$ is a cocircuit of $M$, which proves (\ref*{clm:keylemma-circuitflat}).

\begin{claim} \label{clm:keylemma-contracteisinternally3conn}
If $M \d x, y / a, e$ is internally $3$-connected, then there is an element $c \in E(M) \setminus \{x, y, a, b, e\}$ such that $\lambda_M(\{x, y, a, b, c, e\}) = 2$, $\{a, b, x, e\}$ is a circuit, $\{a, b, y, c\}$ is a circuit, and neither $e$ nor $c$ is in $\cl_M(E(M) \setminus \{x, y, a, b, c, e\})$.
\end{claim}

Assume that $M \d x, y / a, e$ is internally $3$-connected. Then by (\ref*{clm:keylemma-trianglethingcase:2}), $M \d x, y / e$ is internally $3$-connected, and by (\ref*{clm:keylemma-circuitflat}), $\{a, b, x, e\}$ is a circuit and a flat of $M$ and there is an element $c \in E(M) \setminus \{a, b, x, y, e\}$ such that either $\{e, c\}$ or $\{e, c, y\}$ along with one of $a$ or $b$ forms a cocircuit of $M$. By symmetry, we may assume that either $\{e, c, b\}$ or $\{e, c, b, y\}$ is a cocircuit.
Note that $y \not\in \cl_M(\{a, b, x\})$ because then $\lambda_{M \d x / e}(\{a, b, y\}) = 1$ and $M \d x / e$ is not internally $3$-connected, contradicting Bixby's Lemma and the fact that $M \d x, e$ is not stable. Therefore, $\{a, b, x, y\}$ is an independent cocircuit of $M$ and $\lambda_M(\{a, b, x, y\}) = 3$.
Let $U = E(M) \setminus \{a, b, x, y, e, c\}$.
Since $\{b, c\}$ is a series pair of $M \d y, e$, we have $c \not\in \cl_M(U)$, hence $\sqcap_M(\{a, b, x, y\}, U) = 2$.
If $c \in \cl_M(\{a, b, y\})$, then we have $\lambda_M(\{x, y, a, b, c, e\}) = 2$ and we have proved (\ref*{clm:keylemma-contracteisinternally3conn}), so we may assume $c \not\in \cl_M(\{a, b, y\})$.
Also, since $e \not\in \cl_M(U)$, $(\{a, b, x, e\}, U)$ is a $2$-separation of $M \d y, c$. 
But $M \d y$ is $3$-connected, so $c \in \cl_M(\{a, b, x\} \cup U)$. 

Recall that there are sets $S$ and $T$ such that $e \in S \setminus T$, $\cl_M(S) \cap E(N_0) \neq \cl_{M'}(S) \cap E(N_0)$ and $\cl_M(T) \cap E(N_0) \neq \cl_{M'}(T) \cap E(N_0)$. We note that $x, y \in S \cap T$.

Next, we claim that $\{x, y\}$ and $U$ are skew in both $M$ and $M'$. First, suppose $\sqcap_M(\{x, y\}, U) = 1$ and $\sqcap_{M'}(\{x, y\}, U) = 1$. Note that $\sqcap_M(\{x, y, a\}, U)$ and $\sqcap_{M'}(\{x, y, b\}, U)$ cannot both equal two; we may assume that $\sqcap_M(\{x, y, a\}, U) = 1$. Then $\sqcap_M(\{x, a\}, U) = \sqcap_M(\{y, a\}, U) = 0$ so we have also $\sqcap_{M'}(\{x, y, a\}, U) = 1$. But then $\cl_M(S \cup \{a\}) \cap E(N_0) = \cl_M(S) \cap E(N_0)$ and $\cl_{M'}(S \cup \{a\}) \cap E(N_0) = \cl_{M'}(S) \cap E(N_0)$, contradicting (\ref*{clm:keylemma-abnotindistinguishingsets}).
Next, suppose that $\{x, y\}$ and $U$ are skew in one of $M$ and $M'$, which we call $N$, and not in the other, which we call $N'$. By Tutte's Linking Theorem, there is a set $C \subseteq E(N)$ such that $N / C$ has $N | \{x, y, a, b, e, c\}$ and $N_0$ as restrictions, and $\sqcap_{N / C}(\{x, y, a, b, e, c\}, E(N_0)) = 2$. Then $\sqcap_{N / C}(\{x, y\}, E(N_0)) = 1$ and by the modularity of $N_0$, there is an element $z \in \cl_{N / C}(\{x, y\}) \cap E(N_0)$. By (\ref*{clm:keylemma-abnotindistinguishingsets}), $z \in \cl_{N' / C}(\{x, y, a\})$ and $z \in \cl_{N' / C}(\{x, y, b\})$. But this implies that $z \in \cl_{N' / C}(\{x, y\})$, which means that $\sqcap_{N'}(\{x, y\}, U) = 1$, a contradiction.
Hence $\{x, y\}$ and $U$ are skew in both $M$ and $M'$.

Since $\cl_M(T) \neq \cl_{M'}(T)$, it follows from (\ref*{clm:keylemma-abnotindistinguishingsets}) that $a, b \not\in T$. 
We also see that $c \in T$; if not then $T \setminus U = \{x, y\}$ and since $\{x, y\}$ is skew to $U$ in both $M$ and $M'$, we would have $\cl_M(T) \cap E(N_0) = \cl_M(T \setminus \{x, y\}) \cap E(N_0)$ and $\cl_{M'}(T) \cap E(N_0) = \cl_{M'}(T \setminus \{x, y\}) \cap E(N_0)$, and these two sets are equal, a contradiction.

Since $c \in T$, $\cl_{M / c}(T \setminus \{c\}) \neq \cl_{M' / c}(T \setminus \{c\})$, so $M / c \neq M' / c$. By \autoref{lem:remove}, this implies that either $M \d x / c$ is not stable, $M \d y / c$ is not stable, or $M \d x, y / c$ is not connected. But $M \d x, y$ is internally $3$-connected and $c$ is not in its unique series pair $\{a, b\}$, so $M \d x, y / c$ is connected.
We note that $\lambda_{M \d y, c}(\{a, b, x, e\}) = 1$ so $M \d y, c$ has an internal $2$-separation, and by Bixby's Lemma, $M \d y / c$ is internally $3$-connected and stable. We conclude that $M \d x / c$ is not stable.

Let $(W, Z)$ be a $2$-separation of $M \d x / c$. Since $\{a, b, y\}$ is a triad of $M \d x / c$, we may assume that $a, b, y \in W$. Since $M \d x$ is $3$-connected, $c \in \cl_M(W) \cap \cl_M(Z)$. But $c \not\in \cl_M(U)$, so $Z \not\subseteq U$ and thus we have $e \in Z$. 
Since $c \not\in \cl_M(\{a, b, y\})$, $W \setminus \{a, b, y\}$ is not empty and has an element not in $\cl_M(Z)$.
But then $M \d y / c$ has a $2$-separation $(W \setminus \{a, b, y\}, Z \cup \{a, b, x\})$. But $M \d y / c$ is internally $3$-connected and $M \d y$ is $3$-connected, so $M \d y / c$ is connected and $W \setminus \{a, b, y\}$ is a parallel pair. But this implies that $c$ is in a triangle with two elements of $W \setminus \{a, b, y\}$, contradicting the fact that $c \not\in \cl_M(U)$.

\begin{claim} \label{clm:keylemma-gettrianglething}
There is an element $c \in E(M) \setminus \{x, y, a, b, e\}$ such that $\lambda_M(\{x, y, a, b, c, e\}) = 2$, $\{a, b, x, e\}$ is a circuit, $\{a, b, y, c\}$ is a circuit, and neither $e$ nor $c$ is in $\cl_M(E(M) \setminus \{x, y, a, b, c, e\})$.
\end{claim}

By (\ref*{clm:keylemma-contracteisinternally3conn}), we may assume that $M \d x, y / a, e$ is not internally $3$-connected; by Bixby's Lemma, $M \d x, y / a \d e$ is internally $3$-connected.

By (\ref*{clm:keylemma-trianglethingcase:1}), $M \d x, y, e$ has an internal $2$-separation $(U, V)$ with $a \in V$. Since $(U, V \setminus \{a\})$ is not an internal $2$-separation of $M \d x, y, e / a$, $|V| = 3$ and $V$ is a series class of $M \d x, y, e$. So $V$ consists of $\{a, b\}$ and another element $c \in E(M) \setminus \{x, y, a, b, e\}$. Moreover, up to ordering the parts, this is the unique internal $2$-separation of $M \d x, y, e$, so $\{a, b, c\}$ is a flat of $M \d x, y, e$.
We have $y \not\in \cl_M(E(M) \setminus \{a, b\})$ because $M \d x$ is $3$-connected. Thus if $y \not\in \cl_M(\{a, b, c\})$ then $M \d x, e$ is internally $3$-connected, a contradiction because it is not stable. Therefore, we have $y \in \cl_M(\{a, b, c\})$.
Also, $e \not\in \cl_M(E(M) \setminus \{x, y, a, b, c, e\})$ because then $\{a, b, c\}$, which is a series class of $M \d x, y, e$, would also be a series class of $M \d x, y$.

Let $(W, Z)$ be a $2$-separation of $M \d y / e$; since $\{a, b, x\}$ is a triad of $M \d y / e$, we may assume that $a, b, x \in W$. Since $M \d y$ is $3$-connected, $e \in \cl_M(W) \cap \cl_M(Z)$. But $e \not\in \cl_M(E(M) \setminus \{a, b, x, y, c\})$, for then $\lambda_{M \d x}(\{a, b, y, c\})$ would equal $\lambda_{M \d x, e}(\{a, b, y, c\}) = 1$, but $M \d x$ is $3$-connected. Thus $Z \cap \{a, b, x, y, c\}$ is non-empty and we have $c \in Z$.
If $e \in \cl_M(\{a, b, x\})$ then we are done.
If not, then $W \setminus \{a, b, x\}$ is not contained in $\cl_M(\{a, b, x\})$ and $\lambda_{M \d x / e}(W \setminus \{a, b, x\}) \leq 1$. Since $M \d x / e$ is internally $3$-connected, $W \setminus \{a, b, x\}$ is a parallel pair, since if it had a single element that element would be in $\cl_M(Z) \cap \cl_M(\{a, b, x\})$. But this implies that $e$ is in a triangle with two elements of $E(M) \setminus \{x, y, a, b, c\}$, which contradicts the fact that, as we pointed out above, $e \not\in \cl_M(E(M) \setminus \{a, b, x, y, c\})$.
This proves (\ref*{clm:keylemma-gettrianglething}).
\\

The last step in the proof is to show that $\{x, a\}$ is a deletion pair of $M$ and that $M \d x, a$ has one series pair; from this we will get a contradiction.

Let $(U, V)$ be a $2$-separation of $M \d x, a$; we may assume that $b, y \in V$ since $\{b, y\}$ is a series pair of $M \d x, a$. 
We claim that $V = \{b, y\}$. If not, then $|V| > 2$ and $(U, V \setminus \{b\})$ is a $2$-separation of $M \d x, a / b$.

Suppose that $|V| = 3$. Then $V \setminus \{b\}$ is a series pair of $M \d x, a / b$. But $\{e, c, y\}$ is a triad, so $V \setminus \{b\}$ consists of $y$ and another element $z \not\in \{e, c\}$. But then $\{b, y, z\}$ has corank at most two in $M \d x, a$, which means that $\{b, z\}$ has corank at most one in $M \d x, y, a$. But if it has corank one then $\{a, b, z\}$ is a series class of $M \d x, y$, a contradiction; and if it has corank zero then $\{a, b, z\}$ is a series class of $M \d x, y$, also a contradiction.

Otherwise, $|V| > 3$. Note that since $\lambda_{M \d x, a / b}(\{e, c, y\}) = \lambda_M(\{x, y, a, b, c, e\}) = 2$, there is no $2$-separation $(A, B)$ of $M \d x, a / b$ with $A$ or $B$ disjoint from $\{y, e, c\}$.
If $e, c \in U$, then since $\{y, e, c\}$ is a triad of $M \d x, a / b$, $(U \cup \{y\}, V \setminus \{y\})$ is a $2$-separation of $M \d x, a / b$; a contradiction as $V \setminus \{y\}$ is disjoint from $\{e, c, y\}$.
So $e$ and $c$ are not both contained in $U$. But then $(U \setminus \{e, c\}, V \cup \{e, c\})$ is a $2$-separation of $M \d x, a$, also a contradiction, unless $|U \setminus \{e, c\}| < 2$. This means that either $e$ or $c$ is contained in a series pair of $M \d x, a$, and hence $a$ is in a triad of $M \d x$ containing $e$ or $c$. But the only triad of $M \d x$ containing $a$ is $\{a, b, y\}$, since $\{a, b\}$ is a series class of $M \d x, y$.
This proves that $V = \{b, y\}$ and so $M \d x, a$ is internally $3$-connected and its unique series pair is $\{b, y\}$.

Suppose that $M \d a$ is not $3$-connected, and let $(W, Z)$ be a $2$-separation of $M \d a$ with $x \in Z$. Since $M \d x, a$ is internally $3$-connected and $M \d a$ has no parallel pairs, $|Z| = 3$ and $Z \setminus \{x\}$ is a series pair of $M \d x, a$. Therefore, $Z = \{x, y, b\}$. This means that $\lambda_M(\{x, y, a, b\}) = 2$, but $\{x, y, a, b\}$ is an independent cocircuit of $M$ and therefore we have $r_M^*(\{x, y, a, b\}) = 2$. Then $r_{M \d x}^*(\{y, a, b\}) = 1$, contradicting the fact that $M \d x$ is $3$-connected.
This proves that $M \d a$ is $3$-connected and that $\{x, a\}$ is a deletion pair of $M$.

Since $\{x, a\}$ is a deletion pair of $M$ and $M \d x, a$ has a unique series pair $\{b, y\}$, (\ref*{clm:keylemma-gettrianglething}) holds for the deletion pair $\{x, a\}$ in place of $\{x, y\}$, for possibly some new choice of $e$.
In particular, there are elements $e', c' \in E(M) \setminus \{x, y, a, b\}$ such that $\lambda_M(\{x, y, a, b, c', e'\}) = 2$, $\{a, b, x, e'\}$ and $\{a, b, y, c'\}$ are circuits of $M$, and neither $e'$ nor $c'$ is in $\cl_M(E(M) \setminus \{x, y, a, b, c', e'\})$. 
Let $U = E(M) \setminus \{x, y, a, b, e', c'\}$. We have $\sqcap_M(\{x, y, a, b\}, U) = 2$.

Since $\{a, b, x, e\}$ is a circuit of $M$ and neither $\{a, b, x, e'\}$ nor $\{a, b, x, c'\}$ are, at least one of $c'$ and $e'$ is distinct from $e$. Then since $e', c' \not\in \cl_M(E(M) \setminus \{x, y, a, b, e', c'\})$ and $e', c' \in \cl_M(\{x, y, a, b\})$, we have $\lambda_M(\{x, y, a, b\}) \geq 1 + \sqcap_M(\{x, y, a, b\}, E(M) \setminus \{x, y, a, b, e', c'\}) = 1 + \lambda_M(\{x, y, a, b\})$, a contradiction.
\end{proof}

\section{Vertically 4-connected matroids} \label{sec:findminimal3connminor}

We let $M$ be a vertically $4$-connected, non-$\F$-representable matroid with $N_0 \cong \PG(2,\F)$ as a modular restriction.
In this section, we show that $M$ has a minor that is a counterexample to \autoref{lem:keylemma}, and then we finish the proof of our main result, \autoref{thm:mainmodularpptheorem}.
Before proceeding, we state two useful facts; the first is proved in \cite[(5.2)]{GeelenGerardsWhittle:branchwidthandwqo}.

\begin{bixbycoullardinequality}
If $M$ is a matroid, $e \in E(M)$, and $(C_1, C_2)$ and $(D_1, D_2)$ are partitions of $E(M) \setminus \{e\}$, then $\lambda_{M \d e}(D_1) + \lambda_{M/e}(C_1) \geq \lambda_M(D_1 \cap C_1) + \lambda_M(D_2 \cap C_2) - 1$.
\end{bixbycoullardinequality}

The following lemma is proved in a more general form in \cite{GeelenGerardsWhittle:onrotas} but we only need a special case. Recall that for disjoint sets $A, B$ in a matroid $M$, $\kappa_M(A, B) = \min\{\lambda_M(U) : A \subseteq U \subseteq E(M) \setminus B\}$.

\begin{lemma}[{\cite[Lemma 4.3]{GeelenGerardsWhittle:onrotas}}] \label{lem:ordering1}
If $(A, B, V)$ is a partition of the elements of a matroid $M$ such that for each $e \in V$, either $\kappa_{M \d e}(A, B) < \kappa_M(A, B)$ or $\kappa_{M / e}(A, B) < \kappa_M(A, B)$, then there exists an ordering $v_1, \ldots, v_k$ of $V$ such that $\lambda_M(A) = \kappa_M(A, B)$ and for all $i = 1, \ldots, k$, $\lambda_M(A \cup \{v_1, \ldots, v_i\}) = \kappa_M(A, B)$.
\end{lemma}

The main lemma of this section is the following.

\begin{lemma} \label{lem:connectivity3}
If $\F$ is a finite field and $M$ is a simple vertically $4$-connected matroid that is not $\F$-representable and has a modular restriction $N_0 \cong \PG(2, \F)$, then there is a minor $M_0$ of $M$ that is minor-minimal subject to
\begin{enumerate}[(a)]
 \item $M_0$ is $3$-connected, \label{out:connectivity3-3conn}
 \item $N_0$ is a restriction of $M_0$, and \label{out:connectivity3-restriction}
 \item $M_0$ is not $\F$-representable, \label{out:connectivity3-notrepresentable}
\end{enumerate}
such that $\lambda_{M_0}(E(N_0)) \geq 3$.
\end{lemma}

\begin{proof}
First, we make two claims involving connectivity.

\begin{claim} \label{clm:connectivity3-ordering3} \label{clm:connectivity3-1}
Let $(A,B,C,D)$ be a partition of the ground set of a matroid $Q$ such that for each $e \in D$, $\kappa_{Q\d e}(A,B) < \kappa_Q(A,B)$ and for each $e \in C$,
$\kappa_{Q/e}(A,B) < \kappa_Q(A,B)$.
Then for $e \in D$, $\kappa_{Q/e}(A,B) = \kappa_Q(A,B)$ and for $e \in C$, $\kappa_{Q\d e}(A,B) = \kappa_Q(A,B)$.
\end{claim}

Let $e \in D$ and suppose that $\kappa_{Q/e}(A,B) < \kappa_Q(A,B)$. Since we also have $\kappa_{Q\d e}(A,B) < \kappa_Q(A,B)$, there exist
partitions $(W_1,W_2)$ and $(U_1, U_2)$ of $E(Q \d e)$ such that $B \subseteq W_1 \cap U_1$, $A \subseteq W_2\cap U_2$, and $\lambda_{Q \d e}(W_1) =
\kappa_Q(A,B) - 1$ and $\lambda_{Q/e}(U_1) = \kappa_Q(A,B) - 1$.
Then the Bixby-Coullard Inequality implies that
\[ 2 \kappa_Q(A,B) - 2 = \lambda_{Q \d e}(W_1) + \lambda_{Q/e}(U_1) \geq \lambda_Q(U_1 \cap W_1) + \lambda_Q(U_2 \cap W_2) - 1. \]
But the right hand side of this inequality is at least $2\kappa_Q(A,B) - 1$.
This proves the first half of (\ref*{clm:connectivity3-ordering3}) and the second half is the dual argument.

\begin{claim} \label{clm:connectivity3-ordering2}
If $(A,B,C,D)$ is a partition of the ground set of a matroid $Q$ such that for each $e \in D$, $\kappa_{Q\d e}(A,B) < \kappa_Q(A,B)$ and for each $e \in C$,
$\kappa_{Q/e}(A,B) < \kappa_Q(A,B)$, then there exists an ordering $v_1, \ldots, v_k$ of $C \cup D$ such that for all $v_i \in D$, $\lambda_{Q \d v_i}(A \cup
\{v_1, \ldots, v_{i-1}\}) < \kappa_Q(A,B)$ and for all $v_i \in C$, $\lambda_{Q / v_i}(A \cup \{v_1, \ldots, v_{i-1}\}) < \kappa_Q(A,B)$.
\end{claim}

We apply \autoref{lem:ordering1} with $V = C \cup D$ to obtain an ordering $v_1, \ldots, v_k$ of the elements of $C \cup D$ such that $\lambda_Q(A) =
\kappa_Q(A,B)$ and for all $i = 1, \ldots, k$, $\lambda_Q(A \cup \{v_1, \ldots, v_i\}) = \kappa_Q(A,B)$.

This implies that for each $v_i \in C \cup D$, $v_i$ is either in the closure of $A_i = A \cup \{v_1, \ldots, v_{i-1}\}$ and the closure of $B_i = \{v_{i+1},
\ldots, v_k\} \cup B$, or $v_i$ is in the coclosures of both sets.

If $v_i$ is in the closures of both $A_i$ and $B_i$, then $\lambda_{Q/v_i}(A_i) < \kappa_Q(A,B)$ and by (\ref*{clm:connectivity3-ordering3}), $e \in C$. Similarly, if $v_i$
is in the coclosures of these two sets then $\lambda_{Q\d v_i}(A_i) < \kappa_Q(A,B)$ and by (\ref*{clm:connectivity3-ordering3}), $e \in D$.
Thus for $v_i \in C \cup D$, $v_i \in C$ if and only if $\lambda_{Q/v_i}(A_i) < \kappa_Q(A,B)$ and $v_i \in D$ if and only if $\lambda_{Q\d v_i}(A_i) <
\kappa_Q(A,B)$.
This proves (\ref*{clm:connectivity3-ordering2}).
\\

We let $N'$ be a minor of $M$ that is minimal subject to (\ref*{out:connectivity3-3conn}), (\ref*{out:connectivity3-restriction}), and (\ref*{out:connectivity3-notrepresentable}). 
Then we let $M'$ be a minor of $M$ that is minimal such that
\begin{itemize}
\item $M'$ has a minor $N$ satisfying (\ref*{out:connectivity3-3conn}), (\ref*{out:connectivity3-restriction}), and (\ref*{out:connectivity3-notrepresentable}) with $|E(N)| = |E(N')|$, and
\item $\kappa_{M'}(E(N_0), E(N) \setminus E(N_0)) = 3$.
\end{itemize}
This exists since $M$ itself satisfies these conditions. We may assume that $N$ is a proper minor of $M'$, otherwise with $M_0 = M'$ we are done.

We let $X = E(N) \setminus E(N_0)$.
We may assume that $\lambda_N(E(N_0)) < 3$.
Let $(C,D)$ be the partition of $E(M') \setminus E(N)$ such that $N = M' \d D / C$. Then by the minimality of $M'$, for each $e \in D$, $\kappa_{M' \d
e}(E(N_0),X) = 2$ and for each $e \in C$, $\kappa_{M' / e}(E(N_0), X) = 2$.

By (\ref*{clm:connectivity3-ordering2}) applied with $Q = M'$, $A = E(N_0)$ and $B = X$, there exists an ordering $v_1, \ldots, v_k$ of $C \cup D$ such that for all $v_i \in D$,
$\lambda_{M' \d v_i}(E(N_0) \cup \{v_1, \ldots, v_{i-1}\}) < \kappa_{M'}(E(N_0), X)$ and for all $v_i \in C$, $\lambda_{M' / v_i}(E(N_0) \cup \{v_1, \ldots,
v_{i-1}\}) < \kappa_{M'}(E(N_0), X)$.

Since $N_0$ is modular and $M'$ is simple, $v_1 \not\in \cl_{M'}(E(N_0))$ so $v_1 \in D$.

\begin{claim}\label{clm:connectivity3-connectivity3:2}
$k \leq 2$ and if $k = 2$ then $v_2 \in C$.
\end{claim}

Suppose there are $v_j,v_{j+1} \in C$. Then $v_j, v_{j+1} \in \cl_{M'}(E(N_0) \cup \{v_1, \ldots, v_{j-1}\})$, and since $\lambda_{M' / v_j}(E(N_0) \cup
\{v_1, \ldots, v_{j-1}\}) = 2$, we have $\lambda_{M' / v_j, v_{j+1}}(E(N_0) \cup \{v_1, \ldots, v_{j-1}\}) \leq 1$, a contradiction.
Therefore we may assume that there exists $j > 1$ with $v_j \in D$.

Then $(E(N_0) \cup \{v_1, \ldots, v_{j-1}\}, X \cup \{v_{j+1}, \ldots, v_k\})$ is a $3$-separation of $M' \d v_j$.
Now since $\sqcap_{M'}(E(N_0), X) = 2$, $\sqcap_N(E(N_0), X) = 2$, $M' | E(N_0) = N | E(N_0) = N_0$ and $v_1 \not\in \cl_{M'}(E(N_0))$, it follows that $M' \d v_j
/ v_1$ has $N$ as a minor, and hence so does $M' / v_1$.
Then with (\ref*{clm:connectivity3-1}), the properties of $M' / v_1$ contradict the minimality of $M'$. This proves (\ref*{clm:connectivity3-connectivity3:2}).

\begin{claim}\label{clm:connectivity3-connectivity3:3}
$k = 2$.
\end{claim}

If not, then (\ref*{clm:connectivity3-connectivity3:2}) and the fact that $M' \neq N$ imply that $k = 1$.
We note that $v_1 \not\in \cl_{M'}(E(N_0))$ and $v_1 \not\in \cl_{M'}(X)$.
This implies that $M' / v_1 | E(N_0) = N_0$ and $M' / v_1 | X = M' | X$. Also, $\cl_{M'}(X) \cap E(N_0) \subseteq \cl_{M' / v_1}(X) \cap E(N_0)$.
Since $N$ is not $\F$-representable, $N|X = M' |X$ has no $\F$-representation extending any representation of $\cl_{M'}(X) \cap E(N_0)$ induced by a
representation of $N|E(N_0) = N_0$. Therefore, $M' / v_1 | X$ also has no $\F$-representation extending any representation of $\cl_{M'/v_1}(X) \cap
E(N_0)$ induced by a representation of $N_0$, and $M' / v_1$ is not $\F$-representable.

Therefore, $M' / v_1$ has a minor satisfying (\ref*{out:connectivity3-3conn}), (\ref*{out:connectivity3-restriction}), and (\ref*{out:connectivity3-notrepresentable}), has the same number of elements as $N$, and by (\ref*{clm:connectivity3-1}) satisfies
$\kappa_{M'/v_1}(E(N_0), E(M'/v_1) \setminus E(N_0)) = 3$. This contradicts the minimality of $M'$ and proves (\ref*{clm:connectivity3-connectivity3:3}).
\\

We may now assume from (\ref*{clm:connectivity3-connectivity3:3}) that $k = 2$.
Since $\lambda_{M' \d v_1}(E(N_0)) < \lambda_{M'}(E(N_0))$, $v_1 \not\in \cl_{M'}(X \cup \{v_2\})$, so $M' / v_1 | (X \cup \{v_2\}) = M' | (X \cup \{v_2\})$.
Also, since $v_1 \not\in \cl_{M'}(E(N_0))$, $M' / v_1 | E(N_0) = N_0$.
Therefore, as no $\F$-representation of $M' | E(N_0)$ extends to $M' | (X \cup \{v_2\})$, $M' / v_1$ is not $\F$-representable.
So $M'/ v_1$ contains a minor satisfying (\ref*{out:connectivity3-3conn}), (\ref*{out:connectivity3-restriction}), and (\ref*{out:connectivity3-notrepresentable}), and by (\ref*{clm:connectivity3-1}), $\kappa_{M' / v_1}(E(N_0), X) = 3$.
But since $N_0$ is modular, and $v_2 \in \cl_{M'}(E(N_0) \cup \{v_1\})$ by (\ref*{clm:connectivity3-connectivity3:2}), $v_2$ is parallel in $M' / v_1$ to an element of $E(N_0)$, so
$M' / v_1 \d v_2$ also contains a minor satisfying (\ref*{out:connectivity3-3conn}), (\ref*{out:connectivity3-restriction}), and (\ref*{out:connectivity3-notrepresentable}); but $|E(M' / v_1 \d v_2)| = |E(N)|$, so this minor is $M' / v_1 \d v_2$ itself. Then the fact
that $\lambda_{M' / v_1 \d v_2}(E(N_0)) = \kappa_{M' / v_1 \d v_2}(E(N_0), X) = 3$ completes the proof with $M_0 = M' / v_1 \d v_2$.
\end{proof}

Finally, we restate and prove \autoref{thm:mainmodularpptheorem}.

\mainmodularpptheorem*

\begin{proof}
We let $M$ be a vertically $4$-connected matroid with a modular restriction $N_0 \cong \PG(2, \F)$ and assume that $M$ is not $\F$-representable.
By \autoref{lem:connectivity3}, there is a $3$-connected, non-$\F$-representable matroid $M_0$ containing $N_0$ with $\lambda_{M_0}(E(N_0)) \geq 3$, such that $M_0$ has no $3$-connected, non-$\F$-representable proper minor $M$ containing $N_0$.
But \autoref{lem:keylemma} implies that $M_0$ has a $3$-connected, non-$\F$-representable minor $M$ containing $N_0$ with $\lambda_M(E(N_0)) = 2$, a contradiction. Therefore, $M$ is $\F$-representable.
\end{proof}

\section{General modular restrictions} \label{sec:generalform}

In this section, we show that \autoref{thm:generalform} is a corollary of \autoref{thm:mainmodularpptheorem}.
First, we have this useful equivalent characterization of modularity.

\begin{proposition} \label{prop:equivalentmodularitydefinition}
 A restriction $N$ of a matroid $M$ is modular if and only if $M$ has no minor $N'$ with a non-loop element $e$ such that $N' \d e = N$ and $e \in \cl_{N'}(E(N))$, but $e$ is not parallel to an element of $E(N)$ in $N'$.
\end{proposition}

\begin{proof}
 Recall that $N$ is modular if and only if for every flat $F$ of $M$,
 \begin{equation} \label{eq:modularity}
  r_M(F) + r(N) = r_M(F \cap E(N)) + r_M(F \cup E(N)).
 \end{equation}
 If such a minor $N' = M / C \d D$ exists then the flat $F = \cl_M(C \cup \{e\})$ violates equation (\ref{eq:modularity}).
 For the other direction, we choose $M$ to be minimal such that it has a restriction $N$ that is not modular, but no such minor $N'$ of $M$ exists.
 We choose a flat $F$ that violates equation (\ref{eq:modularity}). If $F \setminus \cl_M(E(N)) = \emptyset$, then since $r_M(F) > r_M(F \cap E(N))$, we can choose an element $e \in F$ that is not parallel to any element of $E(N)$ and set $N' = M | (E(N) \cup \{e\})$. Otherwise, we pick any element $c \in F \setminus \cl_M(E(N))$. Then $N$ is a restriction of $M/ c$. Moreover, $N$ is not modular in $M / c$, for $r_{M/c}(F \setminus \{c\}) = r_M(F) - 1$ and $r_{M/c}(F \setminus \{c\} \cup E(N)) = r_M(F \cup E(N)) - 1$, which contradicts the minimality of $M$.
\end{proof}


\generalform*

\begin{proof}
 We may assume that $M$ is simple.
 We let $\F$ be any finite field over which $N$ is representable and $A$ an $\F$-representation of $N$.
 Then $N$ is a restriction of a matroid $N' \cong \PG(r(N)-1, \F)$ such that $E(N')$ is disjoint from $E(M) \setminus E(N)$ and $N'$ has an $\F$-representation $A'$ with $A' | E(N) = A$.
 Since $N$ is modular in $M$, the modular sum $M \oplus_m N'$ exists; we denote it by $M'$.
 
 \begin{claim} \label{clm:generalform-modular}
  $N'$ is a modular restriction of $M'$.
 \end{claim}

 If not, then by \autoref{prop:equivalentmodularitydefinition}, there is a set $C \subseteq E(M') \setminus E(N')$ and an element $e \in E(M') \setminus E(N')$ such that $M' / C$ has $N'$ as a restriction and $e \in \cl_{M' / C}(E(N'))$ but $e$ is not parallel to an element of $E(N')$ in $M' / C$. But then $M / C$ has $N$ as a restriction, $e \in \cl_{M / C}(E(N))$, and $e$ is not parallel to an element of $E(N)$, contradicting the modularity of $N$ in $M$.
  
 \begin{claim} \label{clm:generalform-connected}
  $M'$ is vertically $4$-connected.
 \end{claim}

 Suppose $M'$ has a vertical $(\leq 3)$-separation $(A, B)$. Since $M$ is vertically $4$-connected, $(A \cap E(M), B \cap E(M))$ is not a vertical $(\leq 3)$-separation of $M$. Hence we may assume that $B \cap E(M) \subseteq \cl_M(A \cap E(M))$. Then as $(A, B)$ is a vertical $(\leq 3)$-separation of $M'$, $B \setminus E(M)$ is not contained in $\cl_{M'}(A)$. But this means that $B \setminus E(M)$ is not in $\cl_{M'}(E(M))$, a contradiction because $E(M') \setminus E(M) = E(N') \setminus E(N)$ which is in $\cl_{N'}(E(N))$ and hence in $\cl_{M'}(E(N))$. This proves (\ref*{clm:generalform-connected}).
 \\
 
 The relation of being a modular restriction is transitive (see \cite[Proposition 6.9.7]{Oxley}). Therefore, any $\PG(2, \F)$-restriction of $N'$ is modular in $M'$ by (\ref*{clm:generalform-modular}) and the fact that a $\PG(2, \F)$-restriction is modular in any $\F$-representable matroid.
 So $M'$ is vertically $4$-connected and has a modular $\PG(2, \F)$-restriction, and \autoref{thm:mainmodularpptheorem} implies that it is $\F$-representable.
 
 By the Fundamental Theorem of Projective Geometry, any $\F$-representation of $M'$ that extends an $\F$-representation of $N'$ can be transformed into one that extends $A'$ by row operations, scaling, and applying an automorphism of $\F$. 
 Thus there is an $\F$-representation of $M'$ that extends $A'$, and its restriction to $E(M)$ is an $\F$-representation of $M$ that extends $A$.
 
 Suppose that $D$ is an $\F$-representation of $M$ that extends $A$. Then there are matrices $C_1$ and $C_2$ such that
 \[ D = \bordermatrix{ & E(M) \setminus E(N) & E(N) \cr
                       &  C_1                 & 0  \cr
                       &  C_2                 & A     }. \]
 We let $A_1 = A' | (E(N') \setminus E(N))$, so that $A' = (A \; A_1)$. Then
 \[ D' = \bordermatrix{ & E(M) \setminus E(N) & E(N) & E(N') \setminus E(N) \cr
                   & C_1                 & 0    & 0                    \cr
                   & C_2                 & A    & A_1                     } \]
 is an $\F$-representation of $M'$.
 Hence every $\F$-representation $D$ of $M$ that extends $A$ can be extended to an $\F$-representation $D'$ of $M'$ that extends the representation $A'$ of $N'$. But \autoref{thm:modularflatstabilizes} implies that any two representations of $M'$ that extend $A'$ are equivalent. Therefore, any two representations of $M$ that extend $A$ are equivalent.
\end{proof}

Next, we prove \autoref{cor:representableorline}.

\representableorline*

\begin{proof}
 Suppose that $M$ is a vertically $4$-connected matroid with a $\PG(2, \F_q)$-restriction, $N$, and $M$ is not $\F_q$-representable. Then by \autoref{thm:mainmodularpptheorem}, $N$ is not modular in $M$.
 By \autoref{prop:equivalentmodularitydefinition}, $M$ has a $3$-connected minor $N'$ with an element $e$ such that $N' \d e = N$. Since every pair of lines in a projective plane intersect, $e$ lies in $\cl_{N'}(L)$ for at most one line $L$ of $N$. Hence $N' / e$ has at most one parallel class of size greater than one, so $\si(N' / e)$ is a simple rank-$2$ matroid with at least $|\PG(2, \F_q)| - q = q^2 + 1$ elements.
\end{proof}

\section{Excluded minors} \label{sec:excludedminors}

In this section we prove \autoref{cor:noexcludedminorhasapp}.
The rank-$n$ \defn{affine geometry} over a finite field $\F$, denoted $\AG(n-1, \F)$, is the matroid obtained from the projective geometry $\PG(n-1, \F)$ by deleting a hyperplane.
We use the fact that, for each finite field $\F$ and $n \geq 3$, the affine geometry $\AG(n-1, \F)$ is uniquely representable over $\F$. This is sometimes called the Fundamental Theorem of Affine Geometry; see \cite[Theorem 2.6.3]{Berger}. Note that this implies that any restriction of $\PG(2, \F)$ containing $\AG(2, \F)$ is uniquely representable over $\F$: each element of $\PG(2, \F)$ lies in the closures of two distinct lines of $\AG(2, \F)$ so its column in any representation is uniquely determined, up to scaling, by the representation of these lines.

\noexcludedminorhasapp*

\begin{proof}
 We let $M$ be an excluded minor for the $\F$-representable matroids and assume that $M$ has a $\PG(2, \F)$-restriction, $N_0$. Let $q$ be the order of $\F$.

 If $N_0$ is not modular in $M$, then by \autoref{prop:equivalentmodularitydefinition}, $M$ has a $3$-connected minor $N'$ with an element $e$ such that $N' \d e = N_0$. 
 As we saw in the proof of \autoref{cor:representableorline}, $N'$ has a $U_{2, q^2+1}$-minor.
 Since $U_{2, q^2+1}$ is not $\F$-representable, $M$ is not an excluded minor for the class of $\F$-representable matroids.
 Hence $N_0$ is a modular restriction of $M$.
 The class of $\F$-representable matroids is closed under direct sums and $2$-sums so $M$ is $3$-connected. It follows from \autoref{lem:keylemma} that $\lambda_M(E(N_0)) = 2$. Let $L$ be the line of $N_0$ contained in $\cl_M(E(M) \setminus E(N_0))$ and let $e \in E(N_0) \setminus L$. 
 Then $M | L$ is modular in $M \d (E(N_0) \setminus L)$ so $N_0 \d e$ is modular in $M \d e$, and $M$ is equal to the modular sum $(M \d e) \oplus_m N_0$. As $M$ is an excluded minor for $\F$-representability, $M \d e$ is $\F$-representable. We remarked above that $N_0 \d e$ is uniquely representable over $\F$, so it follows from \autoref{prop:modularsumpreservesrepresentability} that $M$ is $\F$-representable, a contradiction. 
\end{proof}

\end{document}